\documentclass[a4paper,10pt]{amsart}

\usepackage[utf8]{inputenc}
\usepackage[T1]{fontenc}
\usepackage{lmodern}
\usepackage{amsmath}
\usepackage{amssymb}
\usepackage{amsthm}
\usepackage{mathrsfs}
\usepackage{tikz}
\usepackage{url}
\usetikzlibrary{matrix}
\usetikzlibrary{positioning}
\usetikzlibrary{trees}
\usetikzlibrary{automata}
\usepackage{stmaryrd} 

\newtheorem{theorem}{Theorem}[section]

\newtheorem{proposition}[theorem]{Proposition}
\newtheorem{lemma}[theorem]{Lemma}
\newtheorem*{hypotheseh}{Hypothesis H}
\newtheorem{fact}[theorem]{Fact}
\newtheorem{corollary}[theorem]{Corollary}

\theoremstyle{definition}

\newtheorem{remark}[theorem]{Remark}

\newtheorem{construction}[theorem]{Construction}

\renewcommand{\d}{\mathrm{d}}
\renewcommand{\phi}{\varphi}
\renewcommand{\epsilon}{\varepsilon}

\renewcommand{\div}{\operatorname{div}}
\newcommand{\Id}{\mathrm{Id}}
\newcommand{\im}{\mathrm{im}}
\newcommand{\tr}{\mathrm{tr}}
\newcommand{\Vect}{\mathrm{Span}}
\newcommand{\dom}{\mathrm{dom}}
\newcommand{\bb}{\mathbb}
\renewcommand{\t}{\text}
\newcommand{\Sp}{\operatorname{Sp}}

\renewcommand{\o}{\mathrm{int}}
\newcommand{\cotan}{\mathrm{cotan}}

\title{Reeb periodic orbits after a bypass attachment}

\author{Anne Vaugon}
\address{ Anne Vaugon, Université de Lyon, CNRS UMR 5669, ENS Lyon, UMPA 
46, allée d’Italie, 69364 Lyon Cedex 07, France}
\email{anne.vaugon@ens-lyon.fr}

\begin{document}
\begin{abstract}
On a $3$-dimensional contact manifold with boundary, a bypass attachment is an elementary change of the contact structure consisting in the attachment of a thickened half-disc with a prescribed contact structure along an arc on the boundary. 
We give a model bypass attachment in which we describe the periodic orbits of the Reeb vector field created by the bypass attachment in terms of Reeb chords of the attachment arc. As an application, we compute the contact homology of a product neighbourhood of a convex surface after a bypass attachment, and the contact homology of some contact structures on solid tori.
\end{abstract}

\maketitle

\section{Introduction}
We describe the effect on Reeb dynamics of an elementary change of the contact structure on a $3$-manifold with boundary known as a \emph{bypass attachment}. Our study is based on the explicit construction of a bypass.
We encode the dynamics of the associated Reeb vector field and give a symbolic representation of the new periodic orbits. As an application, we compute the contact homology of a product neighbourhood of a convex surface after a bypass attachment, and the contact homology of some contact structures on solid tori.

Honda \cite{Honda00} introduced bypass attachments to classify contact structures on solid tori, thickened tori, and lens spaces. Bypasses may be seen as basic building-blocks of contact structures. 
In particular, cobordisms are constructed out of bypasses as contact structures on a thickened surface are obtained from an invariant contact structure and a finite number of bypass attachments and removals (see \cite[Section~11.1]{HondaLN}). 

Describing new periodic orbits after a bypass attachment is the first step toward computing the \emph{contact homology} of the new contact manifold. Introduced in the vein of Floer homology by Eliashberg, Givental, and Hofer \cite{EGH00} in 2000, contact homology is an invariant of a contact structure on a closed manifold defined through a Reeb vector field. 
Colin, Ghiggini, Honda, and Hutchings~\cite{CGHH10} generalised it to an invariant of contact structures on manifolds with boundary called \emph{sutured contact homology}. The simplest associated complex is the $\mathbb Q$-vector space generated by Reeb periodic orbits and the differential ``counts'' pseudo-holomorphic cylinders in the symplectisation of the contact manifold.  Gromov \cite{Gromov85} introduced pseudo-holomorphic curves in symplectic geometry in 1985. Hofer \cite{Hofer93} generalised them to symplectisations in 1993. 
Our theorem is similar to a theorem of Bourgeois, Ekholm, and Eliashberg \cite{BEE09} describing the new Reeb periodic orbits after a surgery along a Legendrian sphere~$\Lambda$ in terms of Reeb chords of~$\Lambda$. In addition, the authors deduce exact triangles between contact homology, symplectic 
homology and 
Legendrian contact homology. Finding an analogous triangle would be a natural extension to this paper.

The computation of contact homology hinges on finding periodic orbits and solving elliptic partial differential equations and thus is usually out of reach. To our knowledge, Golovko's work \cite{Golovko10,Golovko11} contains the only explicit computations in the sutured case. 
Actual computations are of importance to clarify our intuition and understand connections between sutured Heegaard-Floer homology and sutured embedded contact homology. Sutured embedded contact homology is a variant of sutured contact homology introduced in~\cite{CGHH10} in the vein of Hutchings' work \cite{Hutchings02}. Taubes \cite{Taubes08} proved that it is an invariant of the manifold. In the closed case, Kutluhan, Lee, and Taubes \cite{KLT10} and, independently, Colin, Ghiggini, and Honda \cite{CGh11} announced an isomorphism between the two homologies. 
In addition, computations can be of use to understand the algebraic structures associated to sutured contact homology and obtain a comprehensive gluing theorem from the partial theorem in~\cite{CGHH10}.

\subsection*{Outline} This paper is derived from the PhD thesis of the author \cite{Vaugon11}. It is organised as follows. In Section \ref{section_main_results}, we present our main theorems. In Section \ref{section_contact_geometry}, we recall some usual definitions in contact geometry and contact homology that will be used in the sequel. In Section \ref{section_surface_epaissie}, we apply our main theorem to the simplest manifold with boundary: the product neighbourhood of a convex surface. 
In Section \ref{section_application}, we compute the contact homology of a product neighbourhood of a convex surface after a bypass attachment, and the contact homology of some contact structures on solid tori. The proof of our main theorem (Theorem~\ref{theoreme_principal}) is technical. 
We give a sketch of proof in Section \ref{section_sketch}  and a detailed proof in Section \ref{section_hyperbolic_proof}. Section \ref{section_Maslov} is devoted to the proof of Theorem \ref{theoreme_Maslov} which describe the Conley-Zehnder index of the new Reeb periodic orbits.

\section{Main results}\label{section_main_results}

\subsection{Bypasses}
Let us review some basic definitions (see Section \ref{section_contact_geometry} for more details). Let $M$ be a $3$-manifold. A $1$-form $\alpha$ on $M$ is called a \emph{contact form} if $\alpha\wedge\d \alpha$ is a volume form. A \emph{contact structure} $\xi$ is a plane field locally defined as the kernel of a contact form. To any contact form $\alpha$, we associate the vector field, called the \emph{Reeb vector field}, such that $\iota_{R_\alpha}\alpha=1$ and $\iota_{R_\alpha}\d\alpha=0$. A \emph{Reeb chord} of an arc $\gamma_0$ is a Reeb arc with endpoints on $\gamma_0$. A curve tangent to $\xi$ is called \emph{Legendrian}. 
In a contact manifold, ``pleasant'' surfaces are \emph{convex} surfaces (see Section~\ref{subsection_convex} for a precise definition). In a neighbourhood of a convex surface $S$, the contact structure is encoded by a smooth multi-curve $\Gamma$, the \emph{dividing set}, separating $S$ into positive and negative regions (Giroux \cite{Giroux91}). In what follows we specify the dividing set associated to a convex surface $S$ by the pair $(S,\Gamma)$. 
Convexity is a natural condition to impose to the boundary of a contact manifold. To deal with contact forms as opposed to contact structures, for instance to define sutured contact homology, one usually refines this condition as follows.
 A contact form $\alpha$ is \emph{adapted} to a convex surface $(S,\Gamma)$ if $\Gamma$ is the set of tangency points between $R_\alpha$ and $S$ and, along $\Gamma$, the vector field $R_\alpha$ points toward the sub-surface $S_+$ where $R_\alpha$ is positively transverse to $S$. 
On a manifold with convex boundary, the dividing set of the boundary is a \emph{suture} as defined by Gabai \cite{Gabai83}.

An \emph{attaching arc} of the convex boundary $(S,\Gamma)$ of $(M,\xi)$ is a Legendrian arc which intersects the dividing set $\Gamma$ in precisely three points, namely, its two endpoints and one interior point.
A \emph{bypass attachment} along an attachment arc $\gamma_0$ is the gluing of a half-disc $D$ with a prescribed continuation of $\xi$ along $\gamma_0$. We get a new manifold with boundary by thickening $(D,\xi)$. 

\subsection{Main theorem}

Let $I_b=\left[-\frac{3\pi}{4},\frac{11\pi}{4}\right]$. Let $(M,\alpha)$ be a contact manifold with convex boundary $(S,\Gamma)$ and $\gamma_0$ be an attachment arc on $S$. We assume that
\begin{itemize}
  \item[(C1)] there exists a neighbourhood $Z$ of $\gamma_0$ with coordinates $(x,y,z)\in I_b \times [-y_\text{max},0] \times I_\t{max}$ where $I_\t{max}=[-z_\text{max},z_\text{max}]$ such that
  \begin{itemize}
    \item[$\bullet$] $\alpha=\sin(x)\d y+\cos(x)\d z$;
    \item[$\bullet$] $\gamma_0=[0,2\pi]\times\{0\}\times\{0\}$;
    \item[$\bullet$]  $S_Z=I_b\times\{1\}\times I_\t{max}=S\cap Z$
  \end{itemize}
  \item[(C2)] $\alpha$ is adapted to  $S\setminus S_Z$.
\end{itemize}
Fix $K>0$. Let $\delta_K(\gamma_0)$ denote the image of $\gamma_0\setminus\Gamma$ on $S$ by the Reeb flow for times smaller than $K$. Additionally, we assume that
\begin{itemize}
  \item[(C3)] $\delta_K(\gamma_0)$ is transverse to $\gamma_0$.
\end{itemize}
Condition (C3) is generic and ensures that the number of Reeb chord of $\gamma_0$ with period smaller than $K$ is finite. We denote by $a_1,\dots a_N$ these chords. Let $l(a_{i_1}\dots a_{i_k})=T(a_{i_1})+\dots +T(a_{i_k})$ where $T(a_i)$ is the period of the Reeb chord $a_i$.

\begin{theorem}\label{theoreme_principal}Under conditions (C1), (C2) and (C3),
there exists a contact manifold $(M',S',\alpha')$ obtained from $(M,S,\alpha)$ after a bypass attachment along $\gamma_0$, such that 
\begin{itemize}
 \item $S'$ is convex;
 \item $\alpha'$ is adapted to $S'$ and arbitrarily close to $\alpha$ in $M$;
 \item Reeb periodic orbits of period smaller than $K$ intersecting the bypass correspond bijectively to words $\mathbf{a}$ on the letters $a_1,\dots, a_N$ such that $l(\mathbf{a})<K$ up to cyclic permutation.
 \end{itemize}
In addition, the periodic orbit $\gamma_{\mathbf{a}}$ associated to $\mathbf{a}=a_{i_1}\dots a_{i_k}$ intersects $S_Z$ in $2k$ points denoted by $p_1^-,p_1^+,\dots,p_k^-,p_k^+$ and is arbitrarily close to the chord $a_j$ between $p_j^-$~and~$p_j^+$.
\end{theorem}
Therefore, if the contact form $\alpha$ is non-degenerate on $M$, the Reeb periodic orbits of period smaller than $K$ on $M'$ are exactly the Reeb periodic orbits of $M$ of period smaller than $K$ and the orbits described in Theorem \ref{theoreme_principal}.
This theorem is proved in Sections \ref{section_sketch} and \ref{section_hyperbolic_proof}. We will see in the proof that the condition ``$\alpha'$ is adapted to $S'$'' is crucial to obtain this symbolic representation of the new periodic orbits. The following proposition ensures that conditions (C1) and (C2) are satisfied for any contact manifold after an isotopy.

\begin{proposition}\label{proposition_adapte}
 Let $(M,\xi)$ be a contact manifold with convex boundary $(S,\Gamma)$ and $\gamma_0$ be an attaching arc. There exists a contact structure $\xi'$ isotopic to $\xi$ and a contact form $\alpha$ of $\xi'$ satisfying conditions (C1) and (C2).
\end{proposition}
This proposition derives from Giroux theory of convex surfaces (Sections \ref{section_contact_geometry} and~\ref{section_surface_epaissie}) and the explicit construction of Proposition \ref{proposition_perturbation}.
 
\subsection{Computations of contact homology} We now apply Theorem \ref{theoreme_principal} to compute some sutured contact homologies.
The sutured contact homology is an invariant associated to a contact structure with boundary $(M,\xi,\Gamma)$ where $\Gamma$ is the dividing set of the boundary.
Though commonly accepted, existence and invariance of contact homology remain unproven. In what follows this assumption will be called Hypothesis H (see Section \ref{subsection_sutured_CH} for more details). 

Let $S$ be a convex surface and $\Gamma=\bigcup_{i=0}^n \Gamma_i$ be a dividing set of $S$ without contractible components. Let $M=S\times[-1,1]$ be the product neighbourhood of $S$ with invariant contact structure\footnote{The multi-curve $\Gamma\times\{\pm1\}$ is a dividing set of the boundary.}. 

\begin{proposition}\label{homologie_surface_epaissie_sr_intro}
There exists a contact form $\alpha$ without contractible Reeb periodic orbits such that the cylindrical sutured contact homology of $(M,\alpha,\Gamma\times\{\pm1\})$ is the $\mathbb Q$-vector space generated by $n+1$ periodic orbits homotopic to $\Gamma_k\times\{0\}, k=0,\dots, n$ and by their multiples.
\end{proposition}

\begin{theorem}\label{homologie_surface_epaissie_intro}
Let $\gamma_0$ be an attachment arc in $S$ intersecting three distinct components of $\Gamma$. Let $\Gamma_0$ be the component intersecting the interior of $\gamma_0$. We denote by $(M',\xi')$ the contact manifold obtained from $(M,\xi)$ after a bypass attachment along $\gamma_0\times\{1\}$ and by $\Gamma'$ a dividing set of $\partial M'$.
Then, under Hypothesis H, the cylindrical sutured contact homology of $(M',\xi',\Gamma')$ is the $\mathbb Q$-vector space generated by $n$ periodic orbits homotopic to $\Gamma_k\times\{0\}, k=1,\dots, n$ and by their multiples. 
\end{theorem}

Thus, a bypass attachment removes $\Gamma_0$ and its multiples from the generators of the sutured contact homology. 

Contact structures with longitudinal dividing set on the boundary are characterised by the dividing set of any convex meridian disc \cite{Honda00}. Golovko \cite{Golovko11} computed the contact homology in the case where the dividing set of a meridian discs consists of segments parallel to the boundary. 
He also computed contact the homology of solid tori with non-longitudinal boundary dividing set \cite{Golovko10}. We extend his computations to contact structures such that (see Figure \ref{propriete_P_bis})
\begin{itemize}
  \item[(C4)] the boundary dividing set $\Gamma$ has $2n$ longitudinal components;
  \item[(C5)] if $(D,\Gamma=\bigcup_{i=0}^n \Gamma_i)$ is the dividing set of a convex meridian disc $D$ there exists a partition of $\partial D$ in two sub-intervals $I_1$ and $I_2$ such that
  \begin{itemize}
    \item[$\bullet$] $\partial I_1$ is contained in two bigons (called \emph{extremal bigons});
    \item[$\bullet$] if $I=\{i,\partial \Gamma_i\subset I_1 \text{ or } \partial \Gamma_i\subset I_2\}$ then any connected component of $D\setminus\left(\bigcup_{i\notin I} \Gamma_i\right)$ contains at most one component of $\Gamma$.
  \end{itemize}
\end{itemize}
\begin{figure}[here]
\includegraphics{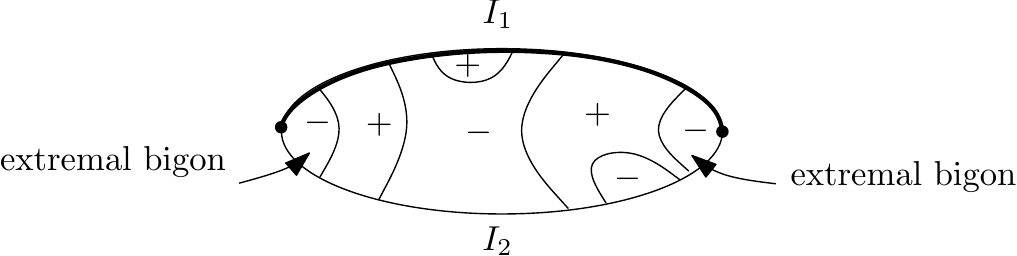}
\caption{A chord diagram satisfying condition (C5)}\label{propriete_P_bis}
\end{figure}

\begin{theorem}\label{homologie_intro_tore_plein}
Let $\xi$ be a contact structure on $M=D^2\times S^1$ satisfying conditions (C4) and (C5) above. Under Hypothesis H, the sutured contact homology of $(M,\xi,\Gamma)$ is the $\mathbb Q$-vector space generated by $n_+$ curves homotopic to $\{*\}\times S^1$, $n_-$ curves homotopic to $\{*\}\times(-S^1)$ and by their multiples where
\[n_\pm=\chi(S_\pm)+\#\{\mp\t{non-extremal bigons}\}-\#\{\pm\t{ bigons}\}.\] 
\end{theorem}

\subsection{Two improvements to the main theorem}\label{section_two_improvements}
We now describe the Conley-Zehnder index $\mu(\gamma_{\mathbf{a}})$ of the periodic orbit $\gamma_{\mathbf{a}}$ from Theorem \ref{theoreme_principal}. This index gives the graduation in contact homology and is associated to a trivialisation of the normal bundle of the orbit.
We first construct a ‘‘nice'' trivialisation extending trivialisations along the Reeb chords. In the setting of Theorem \ref{theoreme_principal}, we denote by $a^-$ and $a^+$ the inward and outward endpoints of a Reeb chord $a$, by $[p,p']$ the segment between $p$ and $p'$ in the chart associated to $Z$ and, if $p$ and $p'$ are on $\gamma_{\mathbf{a}}$, by $[p,p']_{\mathbf{a}}$ the arc of $\gamma_{\mathbf{a}}$ between $p$ and $p'$.
For each $i=1,\dots, N$, choose a collar neighbourhood $S_i$ of $a_i\cup[a_i^+,a_i^-]$. We obtain a collar neighbourhood $S_\mathbf{a}$ of the periodic orbit $\gamma_{\mathbf{a}}$ corresponding to $\mathbf{a}=a_{i_1}\dots a_{i_k}$ by gluing together 
\begin{itemize}
  \item collar neighbourhoods of $[p_{i_j}^-,p_{i_j}^+]_{\mathbf{a}}\cup[p_{i_j}^-,p_{i_j}^+]$ given by a small perturbation of $S_{i_j}$,
  \item and an immersed disc in the bypass with boundary $\bigcup_j[p_{i_j}^+,p_{i_{j+1}}^-]_{\mathbf{a}}\cup[p_{i_j}^+,p_{i_{j}}^-]$, embedded near its boundary.
\end{itemize}
For all $i$, the annulus $S_i$ gives a symplectic trivialisation $(e_1,e_2)$ of $\xi$ along $a_i$. Let $(R_t)_{t\in[0,T(a_i)]}$ denote the path of symplectic matrices induced by the differential of the Reeb flow along $a_i$.
 For all $t\in[0,T(a_i)]$, we denote by $\theta_t$ the angle between $e_1$ and $R_t(e_1)$. Let $\tilde{\mu}(a_i)$ be the integer such that $\theta_{T(a_i)}\in (\pi\tilde{\mu}(a_i),\pi(\tilde{\mu}(a_i)+1)]$. Then $\tilde{\mu}(a_i)$ depends only on the homology class of $S_i$.

\begin{theorem}\label{theoreme_Maslov}
If $\mathbf{a}=a_{i_1}\dots a_{i_k}$ is a word such that $l(\mathbf{a})\leq K$, then
$\mu(\gamma_{\mathbf{a}})= \sum_{j=1}^k\tilde{\mu}(a_{i_j})$
in the trivialisation $S_\mathbf{a}$.
\end{theorem}

In addition, our explicit construction of bypasses allows us to control all the new periodic orbits after a bypass attachment but with less precision. This property is used in our actual computations of contact homology.
Let $(M,\alpha)$ be a contact manifold with convex boundary $(S,\Gamma)$ and $\gamma_0$ an attachment arc satisfying conditions (C1) and (C2). We assume that
\begin{itemize}
  \item[(C6)] there exists $\lambda_0>0$ such that for all $\epsilon>0$ and for any small enough perturbation of $\alpha$, the distance between the dividing set and the endpoints of the Reeb chords of $[0,2\pi]\times\{1\}\times I_\t{max}$ is either smaller that $\epsilon$ or greater than $\lambda_0$.
\end{itemize}

\begin{theorem}\label{proposition_rocade_allegee}Under conditions (C1), (C2) and (C6),
there exists a contact manifold $(M',S',\alpha')$ obtained from $(M,S,\alpha)$ after a bypass attachment along $\gamma_0$, such that 
\begin{itemize}
 \item $S'$ is convex;
 \item $\alpha'$ is adapted to $S'$ and arbitrarily close to $\alpha$ in $M$;
 \item if $\psi :[\pi+\lambda_0,2\pi-\lambda_0]\times I_\t{max}\to  [\lambda_0,\pi-\lambda_0]\times I_\t{max} $ is the partial function induced by the Reeb flow in $M$ and $\phi$ is the map induced by the Reeb flow in the bypass then every periodic orbit intersecting $S_Z$ intersects $S^-_Z$ on a periodic point of $\phi\circ \psi$.
\end{itemize}
\end{theorem}

If the hypotheses of Theorem \ref{theoreme_principal} and Theorem \ref{proposition_rocade_allegee} are simultaneously satisfied, the associated constructions coincide. In addition, the periodic orbit $\gamma_\mathbf{a}$ associated to $\mathbf{a}=a_{i_1}\dots a_{i_k}$ corresponds to the unique fixed point of $\phi\circ\psi_{i_k}\circ\dots\circ\phi\circ\psi_{i_1}$ where $\psi_{i_j}$ is the restriction of $\psi$ to the connected component of $\dom(\psi)$ containing $a_{i_j}^-$.

\section{Contact geometry}\label{section_contact_geometry}
\subsection{Contact geometry and convex surfaces }\label{subsection_convex}
A more detailed presentation can be found in \cite{Geiges08}.
 Let $(M,\xi=\ker(\alpha))$ be a contact manifold. A vector field whose flow preserves $\xi$ is said to be \emph{contact}.
A fundamental step in the classification of contact structures in dimension $3$ was the definition of tight and overtwisted contact structures given by Eliashberg \cite{Eliashberg89} in the line of Bennequin's work \cite{Bennequin83}. A contact structure $\xi$ is \emph{overtwisted} if there exists an embedded disc tangent to $\xi$ on its boundary. Otherwise $\xi$ is said to be \emph{tight}. 

Eliashberg's work \cite{Eliashberg89,Eliashberg92} initiated the study of surfaces in contact manifolds.
The \emph{characteristic foliation} $\mathscr F$ of a surface $S$ is the singular $1$-dimensional foliation of $S$ such that
\begin{itemize}
  \item $x$ is a singular point if $\xi_x=T_x S$;
  \item $\mathscr F_x=\xi_x\cap T_x S$ if $x$ is non-singular. 
\end{itemize}
If $\omega$ is a volume form on $S$ and $i:S\to M$ is the inclusion, $\mathscr F$ is defined by the vector field $X$ satisfying $\iota_X\omega=i^*\alpha$. The characteristic foliation determines the germ of $\xi$ near $S$ \cite[Proposition II.1.2]{Giroux91}.

The development of convexity by Giroux \cite{Giroux91} following Eliashberg and Gromov's definition \cite{EliashbergGromov91} represents a major progress in the study of contact geometry.
A surface $S$ is \emph{convex} if there exists a  contact vector field transverse to $S$. If $S$ has a boundary, we require it to be Legendrian. Closed convex surfaces are generic \cite[Proposition II.2.6]{Giroux91}. The convexity of a surface is equivalent to the existence of a \emph{dividing set} for the characteristic foliation (Giroux, \cite[ Proposition II.2.1]{Giroux91}). A multi-curve $\Gamma$ on $S$ is a \emph{dividing set} for a singular $1$-dimensional foliation  $\mathscr F$ of $S$ if there exist two sub-surfaces $S_\pm$ of $S$, a vector field $Y$ and a volume form $\omega$ on $S$ such that
\begin{itemize}
  \item $\partial S_\pm=\Gamma$;
  \item $\div_\omega Y>0$ on $S_+$ and $\div_\omega Y<0$ on $S_-$;
  \item $Y$ points toward $S_+$ along $\Gamma$.
\end{itemize}
The dividing set $\Gamma$ inherits the orientation of $\partial S_+$. All dividing sets of a given foliation are isotopic. If $X$ is a contact vector field transverse to $S$, the set of tangency points between $X$ and $\xi$ along $S$ is a dividing set. 
The dividing set $\Gamma$ encodes $\xi$ near $S$ as any foliation divided by $\Gamma$ can be realised as the characteristic foliation of a perturbed surface. This property is due to Giroux \cite[Proposition II.3.6]{Giroux91} and known as the \emph{realisation lemma}. In favourable situations, the Reeb vector field provides us with a dividing set.

\begin{lemma}\label{lemme_convexite_Reeb}
Let $S$ be a compact surface in $(M,\alpha)$. If $R_\alpha$ is tangent to $S$ along a smooth curve $\Gamma$ and if, along $\Gamma$, the characteristic foliation of $S$ points toward $S_+$, the sub-surface where $R_\alpha$ is positively transverse to $S$, then $S$ is convex and $\Gamma$ is a dividing set.
\end{lemma}
\begin{proof}
In the definition of dividing set, choose any volume form $\omega$ of $S$ and $Y$ such that $\iota_Y\omega=i^*\alpha$ where $i:S\to M$ is the inclusion.
\end{proof}

\subsection{Bypasses}\label{subsection_bypass} Let $S$ be a closed convex surface without boundary in a contact manifold. A \emph{bypass} for $S$ is an embedded half-disc $D$ in $M$ such that
\begin{itemize}
  \item $D$ is transverse to $S$;
  \item $D$ has a Legendrian boundary denoted by $\gamma_1\cup\gamma_2$ and $D\cap S=\gamma_1$;
  \item the singularities of the characteristic foliation of $D$ are (see Figure \ref{feuilletage_disque_rocade})
  \begin{itemize}
    \item a negative elliptic singularity in the interior of $\gamma_1$;
    \item two positive elliptic singularities at the endpoints of $\gamma_1$;
    \item positive singularities along $\gamma_2$ alternating between elliptic and hyperbolic singularities.
  \end{itemize}
\end{itemize}
The arc $\gamma_1$ is called the \emph{attaching arc} of the bypass.
\begin{figure}[here]
\begin{center}
 \includegraphics{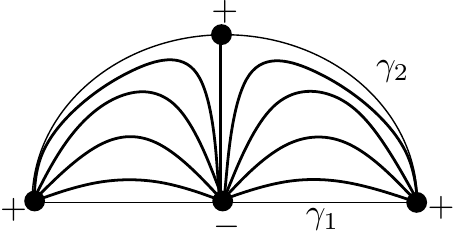}
\end{center}
 \caption{The characteristic foliation of a bypass}\label{feuilletage_disque_rocade}
\end{figure}
\begin{proposition}[Honda \cite{Honda00}]\label{lemme_decoupage_rocade}
Let $D$ be a bypass for $S$ with attaching arc $\gamma_1$. There exists an neighbourhood of $S\cup D$ diffeomorphic to $S\times [0,1]$ such that
\begin{itemize}
  \item $S\simeq S\times\{\epsilon\}$;
  \item the contact structure is invariant in $S\times [0,\epsilon]$;
  \item the surfaces $S\times\{0\}$ and $S\times\{1\}$ are convex with dividing sets $\Gamma$ and $\Gamma'$ where $\Gamma$ and $\Gamma'$ are identical except in a neighbourhood of $\gamma_1$ on which the arrangements are shown in Figure \ref{decoupage_rocade}.
\end{itemize}
\begin{figure}[here]
\begin{center}
 \includegraphics{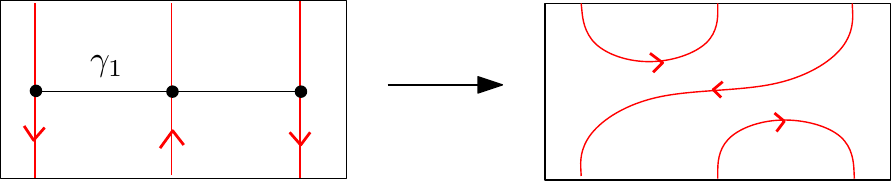}
\end{center}
 \caption{Dividing sets $\Gamma$ and $\Gamma'$}\label{decoupage_rocade}
\end{figure}
\end{proposition}
Let $(M,\xi)$ be a contact manifold with convex boundary $S$. Let $\Gamma$ be a dividing set of $S$ and $\gamma_1$ be an attaching arc. A \emph{bypass attachment along $\gamma_1$} is a contact manifold $(M',\xi')$ with convex boundary $S'$ extending $(M,\xi)$ such that there exists a neighbourhood $S\times[0,1]$ of $S'$ satisfying
 \begin{itemize}
  \item $S'\simeq S\times\{1\}$;
  \item $S\times\{0\}$ is convex and is the image of $S$ by the flow of an inward contact vector field;
  \item there exists a contact retraction of $S\times[0,1]$ on an arbitrarily small neighbourhood of $ S\times\{0\}\cup D $ where $D$ is a bypass for $S\times\{0\}$ with attaching arc the image of $\gamma_1$ on $S\times\{0\}$.
 \end{itemize}
The differences between the dividing sets of $S$ and $S'$ are shown on Figure \ref{decoupage_rocade}. Honda  \cite{Honda00} constructed an explicit bypass attachment on a convex boundary satisfying condition (C1) (see Section \ref{subsubsection_contruction_explicite}). 

There exist two degenerate bypass attachments:  the trivial one that does not change the contact structure up to isotopy and the overtwisted one that creates an overtwisted contact structure (see Figure \ref{rocades_triviale_vrillee}).
\begin{figure}[here]
\includegraphics{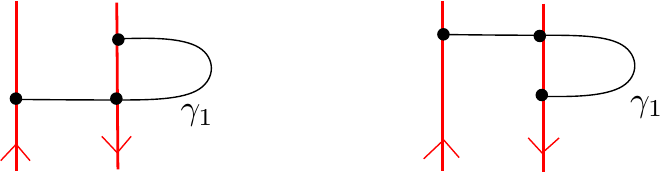}
\caption{Trivial (left) and overtwisted (right) bypasses }\label{rocades_triviale_vrillee}
\end{figure}

Giroux \cite{Giroux00} and Honda \cite{Honda00} independently classified contact structures on solid tori. Honda's proof hinges on bypasses. We follow Mathews' presentation \cite{Mathews09}. A \emph{chord diagram} is a finite set of disjoint properly embedded arcs in the disc $D^2$ up to isotopy relative to the boundary.
 
\begin{theorem}[Giroux \cite{Giroux00}, Honda \cite{Honda00}]\label{theoreme_classification_tore}
Let $F\subset S^1$ be a set with $2n$ elements and $\mathscr F$ be a singular foliation on $T^2$ divided by $\Gamma=F\times S^1$ and containing a meridian leaf which intersects $\Gamma$ in $2n$ points.
Tight contact structures on $D^2\times S^1$ with characteristic foliation $\mathscr F$ on the boundary up to isotopy relative to the boundary correspond bijectively to chord diagrams of $n$ chords with boundary in~$F$.
In addition, the associated chord diagram is the dividing set of any convex meridian disc intersecting $\Gamma$ in $2n$ points.
\end{theorem}

By the realisation lemma, we can assume that the characteristic foliation of the boundary satisfies the hypothesis of Theorem \ref{theoreme_classification_tore}.

\begin{proposition}[Honda \cite{Honda00}]\label{proposition_rocade_tore}
Let $\xi$ be a contact structure on $D^2\times S^1$ such that the boundary dividing set $\Gamma$ has $2n$ longitudinal components. Let $D'$ be a convex meridian disc intersecting $\Gamma$ in $2n$ points. 
Fix an attaching arc $\gamma\subset\partial D'$. Then, the contact structure $\xi'$ on $D^2\times S^1$ obtained after a bypass attachment along $\gamma$ has a boundary dividing set with $2(n-1)$ longitudinal components. 
In addition, the chord diagram associated to $\xi'$ is obtained from the diagram associated to $\xi$ by gluing the endpoints of the two chords intersecting\footnote{This operation corresponds to an annihilation in \cite{Mathews09}.} $\gamma$ (see Figure \ref{rocade_tore_plein}).
\end{proposition}
\begin{figure}[here]
\includegraphics{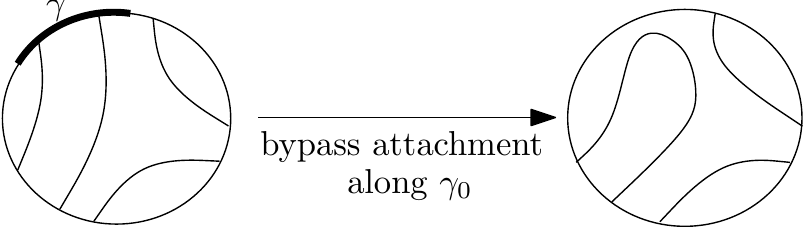}
\caption{Effect of a bypass attachment on solid torus}\label{rocade_tore_plein}
\end{figure}

\subsection{Sutured contact homology}
We consider the generalisation of contact homology to manifolds with boundary called \emph{sutured contact homology} and introduced by Colin, Ghiggini, Honda and Hutchings \cite{CGHH10}. Let $(M,\xi=\ker(\alpha))$ be a contact manifold.

\subsubsection{Holomorphic cylinders}
The differential of contact homology ‘‘counts`` pseudo-holomorphic curves in the symplectisation of the contact manifold.
One can refer to \cite{McDuffSalamon04} for more information.
The \emph{symplectisation} of $(M,\xi=\ker(\alpha))$ is the non-compact symplectic manifold $(\mathbb R\times M,\d(e^\tau \alpha))$ where $\tau$ is the $\mathbb R$-coordinate.
An \emph{almost complex structure} on a even-dimensional manifold $M$ is a map $J :TM\to TM$ preserving the fibres and such that $J^2=-\Id$. 
An almost complex structure $J$ on $\mathbb R\times M$ is \emph{adapted} to $\alpha$ if $J$ is $\tau$-invariant, $J\frac{\partial}{\partial \tau}=R_\alpha$, $J\xi=\xi$ and $\omega(\cdot, J\cdot)$ is a Riemannian metric.
A map $u:(M_1,J_1)\to (M_2,J_2)$ is \emph{pseudo-holomorphic} if $\d u\circ J_1=J_2\circ \d u$.  
Here we consider pseudo-holomorphic cylinders $u :(\mathbb R\times S^1,j)\to \mathbb R\times M$. The simplest non-constant pseudo-holomorphic maps are trivial cylinders: 
\begin{equation*}
\begin{array}{ccc}
 \bb R\times S^1&\longrightarrow & \mathbb R\times M\\
(s,t)&\longmapsto& (Ts,\gamma(Tt)).
\end{array}
\end{equation*}
where $\gamma$ is a $T$-periodic Reeb orbit. Note that there also exist trivial pseudo-holomorphic maps over any Reeb orbit.
For every non-constant map \[u:(\mathbb R\times S^1,j)\to\mathbb (R\times M,J)\] which is not a trivial cylinder, the points $(s,t)$ such that $\d u=0$ or $\frac{\partial}{\partial \tau}\in\im(\d u(s,t)) $ 
are isolated (see {\cite[Lemma 2.4.1]{McDuffSalamon04}}).

The map $u=(a,f):\mathbb R\times S^1\to \bb R\times M$ is \emph{positively asymptotic} to a $T$-periodic
orbit $\gamma$ at $+\infty$ if $\lim_{s\to +\infty} a(s,t)=+\infty$ and $\lim_{s\to+\infty} f(s,t)=\gamma\left(- Tt\right)$. It is \emph{negatively asymptotic} to $\gamma$ at $-\infty$ if $\lim_{s\to-\infty} a(s,t)=-\infty$ and $\lim_{s\to-\infty} f(s,t)=\gamma\left(+ Tt\right)$. 
It is a theorem of Hofer \cite[Theorem 31]{Hofer93} that holomorphic curves $u:(\mathbb R\times S^1,j)\to (\mathbb R\times M,J)$ with finite Hofer energy are asymptotic to a Reeb periodic orbit at $\pm \infty$ if the contact form $\alpha$ is non-degenerate.

\subsubsection{Conley-Zehnder index}\label{subsubsection_CZ} The Conley-Zehnder index gives the graduation in contact homology.
Consider $(M,\xi=\ker(\alpha))$ a contact manifold, $\gamma$ a $T$-periodic Reeb orbit and $p\in\gamma$. If $\phi_t$ denote the Reeb flow, the map $\d\phi_T(p) :(\xi_p,\d\alpha)\to(\xi_p,\d\alpha)$ is a symplectomorpism. 
A non-degenerate periodic orbit $\gamma$ is called \emph{even} if $\d\phi_T(p)$ has two real positive eigenvalues and \emph{odd} if $\d\phi_T(p)$ has two complex conjugate or two real negative eigenvalues. 
In addition, if $\d\phi_T(p)$ has real eigenvalues, the orbit is said \emph{hyperbolic}. If it has two complex conjugate eigenvalues, the orbit is called \emph{elliptic}. 
Let $\gamma_m$ be the $m$-th multiple of a simple orbit $\gamma_1$. Then $\gamma_m$ is said to be \emph{good} if $\gamma_1$ and $\gamma_m$ have the same parity, otherwise $\gamma_m$ is said to be \emph{bad}. 

The \emph{Conley-Zehnder index} was introduced in \cite{ConleyZehnder84} for paths of symplectic matrices. Our short presentation follows \cite{Laudenbach04}. Let $\Sp(2)$ denote the set of symplectic matrices in $\mathcal M_{2}(\mathbb R)$. 
The open set $\Sp^*=\{A\in \Sp(2), \det(A-I)\neq 0\}$ has two connected components and they are contractible. Any path $R:[0,1]\to \Sp(2)$ such that $R_0=I$ and $R_1\in \Sp^*(2)$ can be extended by a path $(R_t)_{t\in[1,2]}$ in $\Sp^*$ such that
\[R_2=W_+=\left(\begin{array}{cc}
 -1 &  0 \\
0 & -1
\end{array}\right) 
\text{ or }
R_2=W_-=\left(\begin{array}{cc}
 2 &  0 \\
0 & 1/2
\end{array}\right).\]
Using polar decomposition, we write $R_t=S_tO_t$ where $S_t$ is positive-definite and $O_t$ is a rotation of angle $\theta_t$. The \emph{Conley-Zehnder index} of $R$ is
$\mu\left(R\right)= \frac{\theta_2-\theta_0}{\pi}$.
It is an integer and does not depend on the choice of an extension of $R$.

As $\d\phi_t(p) :(\xi_p,\d\alpha)\to(\xi_{\phi_t(p)},\d\alpha)$ is a symplectomorpism, a trivialisation of $\xi$ along $\gamma$ provides us with a path of symplectic matrices. If $\gamma$ is non degenerate, its Conley-Zehnder index is well defined. It gives a relative (depending on a choice of trivialisation) grading of Reeb periodic orbits. Its parity matches with the above definition.

\subsubsection{Sutured contact homology}\label{subsection_sutured_CH}
We now assume that $(M,\xi)$ has a convex boundary $(S,\Gamma)$ and that $\alpha$ is a non-degenerate contact equation adapted to the boundary. 
We sketch the construction of cylindrical sutured contact homology chain complex $(C^\t{cyl}_*(M,\Gamma,\alpha),\partial)$ defined in \cite{CGHH10}. The chain complex $C^\t{cyl}_*(M,\Gamma,\alpha)$ is the $\bb Q$-vector space generated by good Reeb periodic orbits (here we consider simple periodic orbits and their good multiples). 
Choose an almost complex structure~$J$ adapted to the symplectisation. To define $\partial\gamma$, consider the set $\mathcal M_{[Z]}(J,\gamma,\gamma')$
of equivalence classes (modulo reparametrisation) of solutions of the Cauchy-Riemann equation with finite energy, positively asymptotic to $\gamma$, negatively asymptotic to $\gamma'$ and in the relative homotopy class $[Z]$. The $\mathbb R$-translation in $\mathbb R\times M$ induces a $\mathbb R$-action on $\mathcal M_{[Z]}(J,\gamma,\gamma')$.
Due to severe transversality issues for multiply-covered curves, there is no complete proof that $\overline{\mathcal M}_{[Z]}(J,\gamma,\gamma')=\mathcal M_{[Z]}(J,\gamma,\gamma')\raisebox{-0.5ex}{/} \raisebox{-1ex}{$\bb R$}$ admits a smooth structure. We use Hypothesis H to make this assumption.

\begin{hypotheseh}
There exists an abstract perturbation of the Cauchy-Riemann equation such that
$\overline{\mathcal M}_{[Z]}(J,\gamma,\gamma')$ is a union of branched labelled manifolds with corners and rational weights whose dimensions are given by $[Z]$ and the Conley-Zehnder indices of the asymptotic periodic orbits.
\end{hypotheseh}
There exists several approaches to the perturbation of moduli spaces due to Fukaya and Ono \cite{FukayaOno99}, Liu and Tian \cite{LiuTian98}, Hofer, Wysocki and Zehnder~\cite{Hofer08,HWZ11,HWZ} or Cieliebak and Oancea in the equivariant contact homology setting~\cite{BourgeoisOancea10, CieliebakOancea}. 
There also exist partial transversality results due to Dragnev~\cite{Dragnev04}.

The \emph{differential of a periodic orbit $\gamma$} is
\[\partial \gamma=\sum_{\gamma'} \frac{n_{\gamma,\gamma'}}{\kappa(\gamma')}\gamma'\]
where $\kappa(\gamma')$ is the multiplicity of $\gamma'$ and $n_{\gamma,\gamma'}$ denote the signed weighted counts of points in $0$-dimensional components of $\overline{\mathcal M}_{[Z]}(J,\gamma,\gamma')$ for all relative homology classes $[Z]$. In particular, the differential of an even (resp. odd) periodic orbit contains only odd (resp. even) periodic orbits.

Under Hypothesis H, it is reasonable to expect the following: if there exists an open set $U\subset \mathbb R\times M$ containing all the images of $J$-holomorphic curves positively asymptotic to $\gamma$, negatively asymptotic to $\gamma'$, then $U$ contains the images of all solutions of perturbed Cauchy-Riemann equations with the same asymptotics for all small enough abstract perturbations.  

Hypothesis H is the key ingredient to prove the existence and invariance of contact homology. The condition ``$\alpha$ adapted to the boundary'' implies that a family of holomorphic cylinders stays in a compact subset in the interior of $M$.

\begin{theorem}[Colin-Ghiggini-Honda-Hutchings]\label{theoreme_partial_2} Under Hypothesis H,
\begin{enumerate}
 \item $\partial^2=0$;
 \item the associated homology $HC^\t{cyl}_*(M,\xi,\Gamma)$ does not depend on the choice of the contact form, complex structure and abstract perturbation. 
 \end{enumerate}
\end{theorem}

If $\partial^2=0$ for some contact form $\alpha$, we denote $HC^\t{cyl}_*(M,\Gamma,\alpha,J)$ the associated homology.
\begin{theorem}[Golovko \cite{Golovko11}]\label{theoreme_Golovko}
Let $\xi$ be a contact structure on $D^2\times S^1$ such that the boundary dividing set has $2n$ longitudinal components and the dividing set of a meridian disc has $n$ components parallel to the boundary.
Then the sutured cylindrical contact homology is the $\mathbb Q$-vector space generated by $n-1$ orbits homotopic to $\{*\}\times S^1$ and by their multiples.
\end{theorem}

\subsubsection{Positivity of intersection}
In dimension $4$, two distinct pseudo-holomorphic curves $C$ and $C'$ have a finite number of intersection points and that each of these points contributes positively to the algebraic intersection number $C\cdot C'$. This result is known as \emph{positivity of intersection} and was introduced by Gromov \cite{Gromov85} and McDuff \cite{McDuff94}.
In this text we will only consider the simplest form of positivity of intersection: let $M$ be a $4$-dimensional manifold, $C$ and $C'$ be two $J$-pseudo-holomorphic curves and $p\in M$ so that $C$ and $C'$ intersect transversely at $p$. 
Consider $v\in T_p C$ and $v'\in T_p C'$ two non-zero tangent vectors. Then $(v,Jv,v',Jv')$ is a direct basis of $T_p M$ ($J$ induces a natural orientation on $T_p M$). In the symplectisation of a contact manifold, positivity of intersection of a pseudo-holomorphic curve with a trivial holomorphic map results in the following lemma.

\begin{lemma}
Let $(M,\xi)$ be a contact manifold, $\alpha$ be a contact form and $J$ be an adapted almost complex structure. Consider $U$ an open subset of $\mathbb C$, $u=(a,f): U\to \mathbb R\times M$ a $J$-pseudo holomorphic curve and $p\in U$ such that $\d f_p$ is injective and transverse to $R(f(p))$.Then, $R(f(p))$ is positively transverse to $\d f_p$.
\end{lemma}

The hypothesis ``$\d f_p$ injective and transverse to $R(f(p))$'' is generic. We will use positivity of intersection in the following situation to carry out explicit computations of sutured contact homology in Sections \ref{section_surface_epaissie} and \ref{section_application}. 
Let $(M,\xi=\ker(\alpha))$ be a contact manifold with convex boundary and $\alpha$ be a contact form. We assume there exist two sets of Reeb chords of $\partial M$, denoted by $X_+$ and $X_-$, with non-empty interior. Let $J$ be an almost complex structure adapted to $\alpha$.

\begin{lemma}\label{proposition_cordes_holomorphes}
Let $u :(\bb R\times S^1,j)\to (\bb R\times M,J)$ be a $J$-holomorphic cylinder asymptotic to $\gamma_+$ and $\gamma_-$. Assume that for any Reeb chord $c\in X_\pm$ there exists a path of properly embedded arcs in $M\setminus(\gamma_+\cup\gamma_-)$ connecting $c$ and a Reeb chord in $\o(X_\mp)$ with reversed orientation. Then $\im(u)$ is disjoint from $\o(X_+)\cup\o(X_-)$.
\end{lemma}

\begin{proof}
Generically a Reeb chord is transverse to $u$. Let $c_+$ be a Reeb chord in $X_+$ transverse to $u$. There exists $c_-$ in $X_-$ transverse to $u$ and connected to $-c_+$ by a path of properly embedded arcs in $M\setminus(\gamma_+\cup\gamma_-)$. By positivity of intersection, $c_\pm\cdot u\geq 0$. Yet $c_+\cdot u=-c_-\cdot u$ and $c_+$ does not intersect $\im(u)$.
\end{proof}

\section{Thickened convex surfaces}\label{section_surface_epaissie}

In this section we study the simplest example of contact manifold with boundary. We compute its sutured contact homology and apply our main theorem. 
Let $S$ be a convex surface and $\Gamma=\bigcup_{i=0}^n \Gamma_i$ be a dividing set of $S$. Assume $\Gamma$ has no contractible component. Let $M=S\times[-1,1]$ be the product neighbourhood of $S$ with invariant contact structure. 
This contact structure is tight \cite[Théoreme 4.5a]{Giroux01}. Let $\gamma_0$ be an attachment arc. The multi-curve $\Gamma\times\{\pm1\}$ is a dividing set of the boundary. Giroux \cite[Proposition 2.1]{Giroux91} proved that there exists a contact form $\alpha_0$ such that 
\begin{itemize}
 \item[(C7)] for all $n=1,\dots,n$, there exists a neighbourhood $U_i$ of $\Gamma_i$ with coordinates \[(x,y,z)\in [-x_\t{max},x_\t{max}]\times [-1,1]\times S^1\] such that
 \begin{itemize}
  \item[$\bullet$] $S\cap U_i\simeq [-x_\t{max},x_\t{max}]\times\{0\}\times S^1$;
  \item[$\bullet$] $\Gamma_i\simeq\{0\}\times\{0\}\times S^1$;
  \item[$\bullet$] $\alpha_0=f(x)\d y+\cos(x)\d z$ where $f :[-x_\t{max},x_\t{max}]\to\mathbb R$ is non-decreasing,  $f=\pm1$ near $\pm x_\t{max}$ and $f=\sin$ near $0$;
 \end{itemize}
 \item[(C8)]  $\alpha_0=\beta_\pm\pm\d y$ on $S_\pm\times[-1,1]\setminus U$ where $\pm\d\beta_\pm>0$ and $U=\bigcup_{i=0}^n U_i$.
\end{itemize}
The Reeb periodic orbits are exactly the curves $\Gamma_i\times\{t\}$ for all $t\in[-1,1]$ and $i=0,\dots, n$ (see Figure \ref{figure_surface_epaissie}).
\begin{figure}[here]
\includegraphics{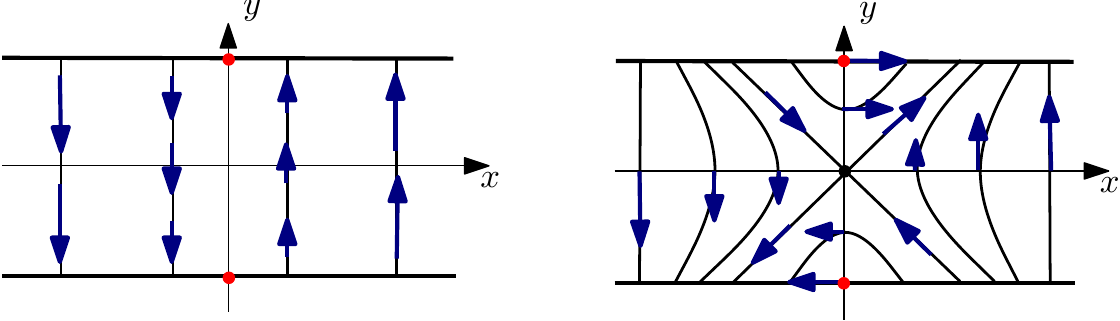}
\caption{The vector fields $R_{\alpha_0}$ and $R_{\alpha_p}$ projected on the $(x,y)$-plane}\label{figure_surface_epaissie}\label{figure_perturbation_surface_epaissie}
\end{figure}
The contact form is degenerate and is not adapted to the boundary. In $U_i$, near $x=0$, we perturb $\alpha_0$ into
\[\alpha_p=\sin(x)\d y+(1+k(x)l(y))\cos(x)\d z\]
where
\begin{itemize}
 \item[(C9)] $k$ is a cut-off function such that $k=1$ near $0$ and $k=0$ near $\pm x_\t{max}$;
 \item[(C10)] $l$ is a $\mathcal{C}^\infty$-small strictly convex function with minimum $0$ at $0$.
\end{itemize}
The associated Reeb vector field is 
\[R_{\alpha_p}=\frac{1}{p(x,y)}\left(\begin{array}{c}
 l'(y)k(x)\cos(x)  \\
p(x,y)\sin(x)-k'(x)l(y)\cos(x) \\
\cos(x)
\end{array}\right).\] 
\begin{proposition}\label{proposition_OP_alpha_p}
The contact form $\alpha_p$ is non-degenerate, adapted to the boundary. Its Reeb periodic orbits are the curves $\Gamma_i\times\{0\}$ for $i=0,\dots, n$. These orbits are even and hyperbolic.
\end{proposition}

\begin{proof}

In $M\setminus U$, the Reeb vector field is $R=\pm\frac{\partial}{\partial y}$. Thus the Reeb periodic orbits are contained in $U$. In addition, if $l$ is small enough, they are contained in the neighbourhood where $k=1$. In this neighbourhood, the projection of the Reeb vector field in the plane $(x,y)$ is collinear to the Hamiltonian vector field of $(x,y)\mapsto l(y)\cos(x)$. 
There are no closed levels so the Reeb periodic orbits correspond to critical points of the Hamiltonian (see Figure~\ref{figure_perturbation_surface_epaissie}).
By linearising the Reeb flow $\phi_t$ along $\Gamma_i\times\{0\}$ we get \[\tr\left(\d\phi_t(0,0,0)_{\vert\xi}\right)=2\cosh\left(\sqrt{tl''(0)}\right)>2.\] Thus the Reeb periodic orbits are even and hyperbolic.
\end{proof}

\begin{proof}[Proof of Proposition \ref{homologie_surface_epaissie_sr_intro}]
Proposition \ref{homologie_surface_epaissie_sr_intro} is a corollary of Proposition \ref{proposition_OP_alpha_p}. Indeed $\partial=0$ in $C_*^\t{cyl}(M,\Gamma\times\{\pm1\},\alpha_p)$ as all the orbits are even.
\end{proof}

We now apply our main theorem to $M$. 
By Proposition \ref{proposition_adapte} and \cite[Proposition~2.1]{Giroux91} there exists an isotopic the contact structure with a contact form satisfying conditions (C1) and (C2) and coordinates in a neighbourhood of $\Gamma$, compatible with the coordinates in $Z$ and satisfying conditions (C7) and (C8). We perturb the contact form into 
\[\alpha_b=f(x)\d y+(1+k(x)l(y)m(z))\cos(x)\d z \] where $k$ and $l$ satisfy conditions (C9) and (C10) and
\begin{itemize}
 \item[(C11)] $m$ is a smooth cut-off function, $m=0$ on $[-z_\t{max},z_\t{max}]$ and $m=1$ outside a small neighbourhood of $[-z_\t{max},z_\t{max}]$.
\end{itemize}
Let $p(x,y,z)=1+k(x)l(y)m(z)$. The Reeb vector field is 
\begin{equation*}\label{R_F_2}
R_{\alpha_b}=\frac{1}{p-k'(x)l(y)m(z)\cos(x)\sin(x)}\left(\begin{array}{c}
 l'(y)k(x)m(z)\cos(x)  \\
p\sin(x)-k'(x)l(y)m(z)\cos(x) \\
\cos(x)
\end{array}\right) .
\end{equation*}

\begin{proposition}\label{proposition_OP_alpha_p_bis}
The contact form $\alpha_b$ is non-degenerate, adapted to the boundary. Its Reeb periodic orbits are the curves $\Gamma_i\times\{0\}$ for $i=0,\dots, n$. These orbits are even and hyperbolic.
\end{proposition}

\begin{proof}
As in the proof of Proposition \ref{proposition_OP_alpha_p}, for $l$ small enough, the Reeb periodic orbits are contained in the neighbourhood where $k=1$. 
In this neighbourhood, any Reeb orbit intersecting the set $xy\leq 0$ meets the boundary or is equal to $\Gamma_i\times\{0\}$ (see Figure \ref{figure_perturbation_surface_epaissie}). 
In addition, any Reeb orbit intersecting the set $xy\geq 0$ meets the set $xy\leq 0$. 
Thus the Reeb periodic orbits are the curves $\Gamma_i\times\{0\}$. By linearising the Reeb flow $\phi_t$ along $\Gamma_i\times\{0\}$, we get \[\tr\left(\d\phi_t(0,0,0)_{\vert\xi}\right)>
\tr\left(\d\phi_0(0,0,0)_{\vert\xi}\right)=2.\qedhere\]
\end{proof}

We denote by $\Gamma_0$ the connected component of $\Gamma$ which intersects the interior of $\gamma_0$. If $\gamma_0$ intersects three distinct components of $\Gamma$, then $\gamma_0$ intersect $U_0$ along $[-x_\t{max},x_\t{max}]\times\{1\}\times\{0\}$. 
If $\gamma_0$ intersects only two distinct components of $\Gamma$, there exists $z_1\in S^1$ such that $\gamma_0$ intersects $U_0$ along $[-x_\t{max},x_\t{max}]\times\{1\}\times\{0\}$ and $[0,x_\t{max}]\times\{1\}\times\{z_1\}$ (reverse the orientation of $S$ if necessary). 
If $c$ is a Reeb chord of $\gamma_0$, we denote by $\overline{c}$ the union of $c$ and the arc of $\gamma_0$ joining $c^+$ and $c^-$. Recall that $c^+$ and $c^-$ are the endpoints of $c$.

\begin{proposition}\label{proposition_cordes}The contact form $\alpha_b$ satisfies condition (C3) for all $K>0$. In addition (see Figure \ref{figure_cordes}),
\begin{itemize}
  \item if $\gamma_0$ intersects three distinct components of $\Gamma$, the set of Reeb chords of $\gamma_0$ is $\{c_k, k\in\mathbb N^*\}$ and $\left[\overline{c_k}\right]=\left[\Gamma_0\right]^k$;
  \item if $\gamma_0$ corresponds to a trivial bypass attachment, the set of Reeb chords of $\gamma_0$ is $\{c_k,d_k, k\in \mathbb N^*\}$ where $\left[\overline{c_k}\right]=[\overline{d_k}]=\left[\Gamma_0\right]^k$ and the $z$-coordinates of $c_k^+$ and $d_k^+$ are respectively $0$ and $z_1$;
  \item if $\gamma_0$ corresponds to an overtwisted bypass attachment, the set of Reeb chords of $\gamma_0$ is $\{d_0,c_k,d_k, k\in \mathbb N^*\}$ where $c_k$ and $d_k$ are as above for $k\in\mathbb N^*$, $\overline{d_0}$ is contractible and the $z$-coordinate of $d_0^+$ is $z_1$.
\end{itemize}
In a trivialisation that does not intersect small translations of the Reeb chord along~$\frac{\partial}{\partial y}$, we have $\tilde{\mu}(c_k)=1$ and $\tilde{\mu}(d_k)=0$.
\end{proposition}
\begin{figure}[here]
\begin{center}
 \includegraphics{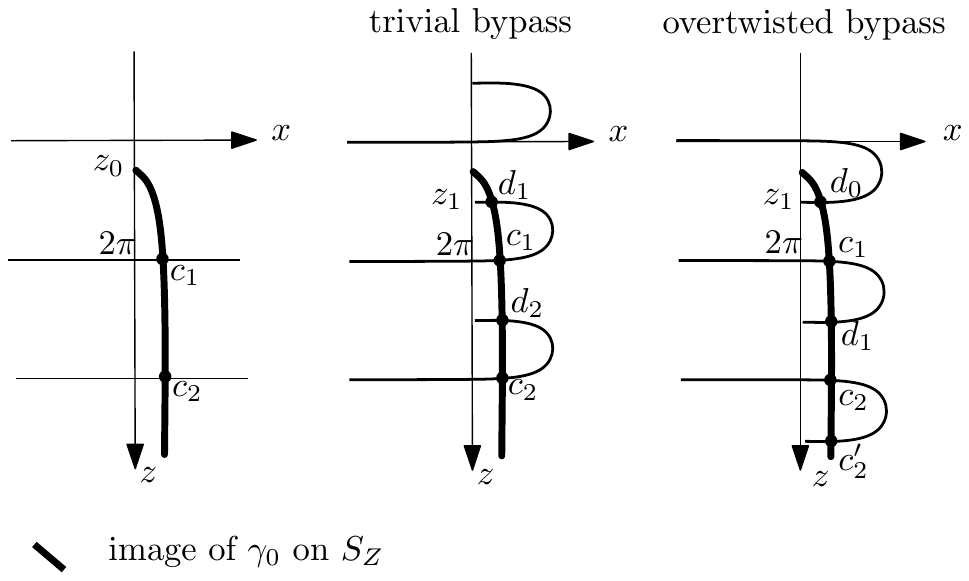}
\end{center}
 \caption{Reeb chords}\label{figure_cordes}
\end{figure}
\begin{corollary}\label{corollaire_surface_epaissie}
We assume that $\gamma_0$ intersects three distinct components of $\Gamma$. Let $L>0$. There exists a contact manifold $(M',\alpha')$ obtained from $(M,\alpha_b)$after a bypass attachment along $\gamma_0$ such that for all $l<L$ the set of Reeb periodic orbits homotopic to $\left[\Gamma_0\right]^l$ is 
\[
  \raisebox{1ex}{$\bigl\{\kappa=(k_1,\dots,k_m)\in(\mathbb N^*) ^l,l=k_1+\dots+k_m\bigr\}$}
  /
  \raisebox{-0.5ex}{$\{\t{cyclic permutation}\}.$}\]
In addition, $\mu(\gamma_\kappa)=m$ if $\gamma_\kappa$ is the orbit associated to $\kappa=(k_1,\dots,k_m)$.
\end{corollary}
\begin{corollary}
An overtwisted bypass attachment creates a contractible Reeb periodic orbit.
\end{corollary}

\begin{proof}[Proof of Proposition \ref{proposition_cordes}]
For $l$ small enough the Reeb chords are contained in $U_0$ and more precisely in the neighbourhood where $k=1$. 
Let $z_0=\sup\{z,m([0,z])=0\}$. 
We lift the $S^1$-coordinate in $U_0$ to $\mathbb R$ and denote by $\delta=(x_S, z_S)$ the image of $\gamma_0$ on $[-x_\t{max},x_\t{max}]\times\{1\}\times\mathbb R$ (see Figure \ref{figure_cordes}). 
There exists $s_\t{max}<x_\t{max}$ such that $\dom(\delta)=(-s_\t{max},0)$. In addition $\lim_{s\to -s_\t{max}} z_S(s)=+\infty$ and $\lim_{s\to 0} z_S(s)=z_0$.

Let $(x_{z_1},y_{z_1})$ denote the image of $\gamma_0$ by the Reeb flow in the plane $z=z_1$. For all $z\geq z_0$ and for all small enough $\eta>0$, $\frac{\d}{\d z}\big(x_z(s+\eta)-x_z(s)\big)$ and $\frac{\d}{\d z}\big(y_z(s+\eta)-y_z(s)\big)$ are non-negative. It holds that
\begin{equation}\label{equation_z'_x'}
x'_{z}(s)\geq1  \hskip1cm y'_{z}(s)\geq\frac{1}{\cos^2(x_0)}z_0
\end{equation} 
(renormalise $R_{\alpha_b}$ and use condition (C10)). Thus, if $(s_0, 1,0)$ and $(x, 1,z)$ are the endpoints of the lift of a Reeb chord, we have
\[z_S'(s_0)=-\frac{1}{\tan(x_z(s_0))}y'_z(s_0)<0\] 
and $\delta$ intersects the segments $[-x_\t{max},x_\t{max}]\times\{1\}\times\{2k\pi\}$ exactly for $k\in \mathbb N^*$. In addition, there is only one intersection point. 
This point is the endpoint of a Reeb chord. If the bypass is trivial or overtwisted, let $z_1$ denote to the smallest positive lift of $z_1\in S^1$. Then $\delta$ also intersects $[0,x_\t{max}]\times\{1\}\times\{2k\pi+z_1\}$ if and only if $k\in\mathbb N^*$ and there is exactly one intersection point. This concludes the description of the Reeb chords when $\gamma_0$ intersects three distinct components of $\Gamma$.

We now compute $\tilde{\mu}(a)$ for some Reeb chord $a$. In the coordinates $(R,e_1,e_2)$ given by the symplectic trivialisation of $\xi$ along $a$, let $v$ denote the projection of $\frac{\partial}{\partial y}$ on $(e_1,e_2)$. 
Then $v\neq e_1$ and $v$ is positively collinear to $e_2$ at $t=0$. Let $R_t$ denote the symplectic matrix induced by the differential of the Reeb flow on $\xi_{a(t)}$. The vector $\d R_t\cdot e_1$ does not cross $\mathbb R_+ v$ as $x'_{z}(s)\geq1$. Write $\d R_t\cdot e_1=r(t)e^{i\theta(t)}$. 
If $a=c_k$, then $v$ is positively collinear to $-e_2$ at $t=T(a)$. The tangent vector to the image of $\gamma_0$ on $\xi_{a^+}$ at the endpoint of $a$ is
\[\left(\begin{array}{l}
 x'_{z_0}(s_0) \\
y'_{z_0}(s_0)\cos^2(x_0)\\
-y'_{z_0}(s_0)\cos(x_0)\sin(x_0)
\end{array}\right)\]
where the endpoints of $a$ are $(s_0,1,0)$ and $(x_0,1,z_0)$.
Using (\ref{equation_z'_x'}), we get \[\theta(T(a))\in (\pi,\frac{3\pi}{2})+2\pi\mathbb Z.\] Thus $\theta(T(c_k))\in[\pi,2\pi]$ and $\tilde{\mu}(c_k)=1$.
If $a=d_k$, then $v$ is positively collinear to $e_2$ at $t=T(a)$. In addition, it holds that \[\theta(T(a))\in (0,\frac{\pi}{2})+2\pi\mathbb Z.\] Thus, we obtain $\theta(T(c_k))\in[0,\pi]$ and $\tilde{\mu}(d_k)=0$.
\end{proof}

\begin{proof}[Proof of Corollary \ref{corollaire_surface_epaissie}]
There exists $\nu>0$ such that for any small perturbation of $\alpha$, the Reeb chords of $[-\nu,\nu]\times\{1\}\times I_\t{max} $ are arbitrarily close to $\{0\}\times\{1\}\times I_\t{max}$. 
We apply Theorem \ref{theoreme_principal} and Proposition \ref{proposition_rocade_allegee} for $\lambda =\nu$ and $K$ such that $i_1+\dots+i_k\leq L$ implies $\Sigma_{j=1}^k T(c_{i_j})<K$. 

The set of Reeb periodic orbits homotopic to $[\Gamma_0]^l$ in Corollary \ref{corollaire_surface_epaissie} corresponds to Reeb periodic orbits with period smaller that $K$ described in Theorem \ref{theoreme_principal}. It remains to prove that there are no other Reeb periodic orbits homotopic to $[\Gamma_0]^l$. If $\gamma$ is such a Reeb periodic orbit then $\gamma$ is associated to a periodic point of $\phi\circ\psi$ (Proposition \ref{proposition_rocade_allegee}). 
We decompose $\psi$ into $(\Psi_k)_{k\in\mathbb N^*}$ so that, if $a$ is a Reeb chord that contributes to $\Psi_k$, then $[\overline a]=[\Gamma_0]^k$. For $k\leq L$, we have $\Psi_k=\psi_k$ (see Proposition \ref{proposition_rocade_allegee}). Therefore $\gamma$ corresponds to a fixed point of $\phi\circ\Psi_{i_k}\circ\dots\circ\phi\circ\Psi_{i_1}$ and $i_1+\dots+i_k=l$ and we obtain $T(\gamma)<K$.
\end{proof}

\section{Applications to sutured contact homology}\label{section_application}

In this section we apply Theorem \ref{theoreme_principal} and Proposition \ref{proposition_rocade_allegee} to prove Theorem \ref{homologie_surface_epaissie_intro} and Theorem \ref{homologie_intro_tore_plein}. To compute the contact homology of a convex surface after a bypass attachment and of some contact structures on solid tori, we construct suitable contact structures on the associated manifolds. 
On a thickened convex surface (Theorem~\ref{homologie_surface_epaissie_intro}) we start from the contact form described in Section \ref{section_surface_epaissie}. Theorem \ref{homologie_intro_tore_plein} gives the contact homology of a contact structure $\xi$ on $D^2\times S^1$ such that
\begin{itemize}
  \item[(C4)] the boundary dividing set $\Gamma$ has $2n$ longitudinal components;
  \item[(C5)] for a convex meridian disc dividing set $(D,\Gamma=\bigcup_{i=0}^n \Gamma_i)$ there exists a partition of $\partial D$ in two sub-intervals $I_1$ and $I_2$ such that
  \begin{itemize}
    \item $\partial I_1$ is contained in two bigons (called \emph{extremal bigons});
    \item if $I=\{i,\partial \Gamma_i\subset I_1 \text{ or } \partial \Gamma_i\subset I_2\}$ then $D\setminus\left(\bigcup_{i\notin I} \Gamma_i\right)$ contains at most one component of $\Gamma$.
  \end{itemize}
\end{itemize}
All these contact structures are obtained from the contact structure $(D^2\times S^1,\xi_{\sslash})$ with parallel dividing set on convex meridian discs after a finite number of bypass attachments. 
In Section \ref{section_solid_tori} we compute the  sutured contact homology of $(D^2\times S^1,\xi_{\sslash})$ and apply Theorem \ref{theoreme_principal} and Proposition \ref{proposition_rocade_allegee} to obtain the Reeb periodic orbits after a finite number of bypass attachments. To compute the contact homology (and prove Theorem \ref{homologie_intro_tore_plein}), it remains to control the differential. 
By Corollary \ref{corollaire_surface_epaissie}, a bypass attachment creates many homotopic periodic orbits. This complicates the direct study of the differential. 
We get round this difficulty in Section \ref{section_application_cylindres} by proving that all the holomorphic cylinders are contained in a standard neighbourhood. We deduce the differential in this neighbourhood by use of computations of contact homology in simple situations on 
a solid torus.

\subsection{Contact forms on solid tori}\label{section_solid_tori}

For $n\in\mathbb N^*$ and $0<\eta<\frac{\pi}{4}$, let
\[D_{n,\eta}=\left\{(x,y), x\in\left[-\pi+\eta-h(y),n\pi-\eta+h(y)\right],y\in[-1,1]\right\} \] 
where $h:[-1,1]\to\mathbb R$ is a strictly concave function with maximum $h(0)<\eta$, vertical and zero at $\pm1$ (see Figure \ref{Gamma'}). On $M_{n,\eta}=D_{n,\eta}\times S^1$, we consider the contact form
\[\alpha=f(x)\d y+\cos(x)\d z+g(x,y)\d x \]
where
\begin{itemize}
  \item $f$ is $2\pi$-periodic, $f=(-1)^{k+1}$ in a neighbourhood of $\left[k\pi+\frac{\pi}{4},k\pi+\frac{3\pi}{4}\right]$ and $f=\sin$ near $k\pi$, $k=-1,\dots,n$;
  \item $g$ does not depend on $y$ for $y\geq \frac{1}{2}$;
  \item $g=0$ for $y\leq 0$ and outside a neighbourhood of $x=-\frac{3\pi}{4}$ where $f=-1$;
  \item the leaves of the characteristic foliation of $\partial M_{n,\eta}$ are closed.
\end{itemize}
Then $(M_{n,\eta,},\ker(\alpha))$ satisfies property (C4) and a dividing set $\Gamma_{\sslash}$ of the boundary is given by the curves $x=-\pi+\eta-h(0)$, $x=k\pi+\frac{\pi}{2}, k=0,\dots, n-1$ and $x=n\pi-\eta+h(0)$.
Let $D$ be a disc in $M_{n,\eta}$ transverse to $\frac{\partial}{\partial z}$ with Legendrian boundary. As $\frac{\partial}{\partial z}$ is a contact vector field, $D$ is convex. A dividing set is given by condition $\frac{\partial}{\partial z}\in\xi$ and is thus composed of the curves $x=k\pi+\frac{\pi}{2}, k=-1,\dots, n-1$. Therefore $(M,\ker(\alpha))$ is diffeomorphic to $(D^2\times S^1,\xi_{\sslash})$ (Theorem \ref{theoreme_classification_tore}). 

\begin{proposition}\label{proposition_perturbation}
There exists a contact form $\alpha_p$ on $M_{n,\eta}$ without contractible Reeb periodic orbits such that $(M_{n,\eta},\ker(\alpha_p))$ is diffeomorphic to $(D^2\times S^1,\xi_{\sslash})$  and the sutured contact homology of $(M_{n,\eta},\alpha_p,\Gamma_{\sslash})$ is the $\mathbb Q$-vector space generated by $n_+=\left\lceil \frac{n-1}{2}\right\rceil$ curves homotopic to $\{*\}\times S^1$, $n_-=n-1-n_+$ curves homotopic to $\{*\}\times(-S^1)$ and by their multiples. 
\end{proposition}

\begin{proof}
As in Section \ref{section_surface_epaissie}, we perturb $\alpha$ into
\[\alpha_p=\sin(x)\d y+(1+k(x)l(y))\cos(x)\d z\]
in a neighbourhood of $x=k\pi, k=0,\dots n-1$ such that $x\mapsto k(x-k\pi)$ satisfies condition (C9) and $l$ satisfies condition (C10). 
By Proposition \ref{proposition_OP_alpha_p}, the contact form~$\alpha_p$ is non-degenerate, adapted to the boundary. Its Reeb periodic orbits are the curves $\{k\pi\}\times\{0\}~ S^1$ for $k=0,\dots, n-1$. 
These orbits are even and hyperbolic. Therefore $\partial=0$.
\end{proof}

Let $\xi$ be a contact structure on $M=D^2\times S^1$ satisfying conditions (C4) and (C5) for some $n_0\in\mathbb N^*$. There exist $n,m>0$, a family of integers $\{k_j\}$ and $\{\epsilon_j\}\in\{-1,1\}$ for $j=1,\dots, m $ such that
\begin{itemize}
  \item $1\leq k_j\leq n-2$ and $k_{j+1}-k_j\geq 3$;
  \item $(M,\xi)$ is diffeomorphic to the contact manifold obtained from $(M_{n,\eta},\ker(\alpha))$ after bypass attachments along the arcs 
  \[\gamma_j=[(k_j-1)\pi,(k_j+1)\pi]\times\{-\epsilon_j\}\times\left\{\frac{2\pi k_j}{n}\right\}\] 
  for $j=1,\dots, m$ (see Figure \ref{Gamma'}).
\end{itemize}
\begin{figure}[here]
\begin{center}
 \includegraphics{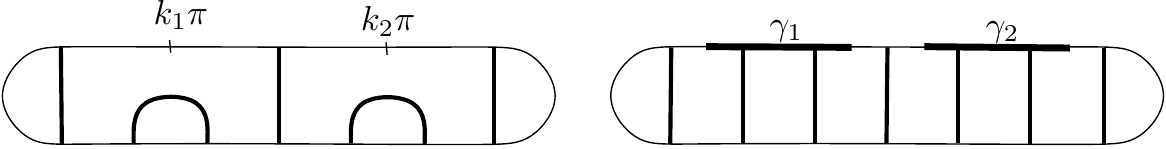}
\end{center}
 \caption{Attaching arcs $\gamma_j$}\label{Gamma'}
\end{figure}
Let $\delta_k=\{k\pi\}\times\{0\}\times S^1$. We describe a contact structure on $M_{n,\eta}$ adapted to the bypass attachments along the curves $\gamma_j$.

\begin{proposition}\label{proposition_alpha'}
Let $L>0$ and $0<x_\t{max}\ll1$. There exists an arbitrarily small perturbation $\alpha_b$ of $\alpha$ in $M_{n,\eta}$ such that
\begin{itemize}
  \item $\gamma_j$ satisfies conditions (C1) and (C3) for all $K>0$ and for all $j=1,\dots,m$;
  \item $\alpha_b$ is adapted to $S\setminus \bigcup S_{Z_j}$;
  \item $\alpha_b=\alpha$ outside $U_k=[-x_\t{max}+k\pi,x_\t{max}+k\pi]\times[-1,1]\times S^1$ for $k=0,\dots, n-1$;
  \item the Reeb periodic orbits are the curves $\delta_k$ for $k=0,\dots, n-1$, these orbits are even and hyperbolic;
  \item the set of Reeb chords of $\gamma_j$ is $\{c_{l,j}, l\in\mathbb N^*\}$ and $\left[\overline{c_{l,j}}\right]=\left[\{*\}\times  S^1\right]^{(-1)^{k_j}l}$. 
\end{itemize}
\end{proposition}

\begin{proof}[Proof of Proposition \ref{proposition_alpha'}]
 In $U_k$, consider the contact form 
\[\alpha_b=f(x)\d y+(1+k(x)l(y)m(z))\cos(x)\d z \]
such that $x\mapsto k(x-k\pi)$ satisfies condition (C9), $l$ satisfies condition (C10) and $m=0$ in $Z=\bigcup Z_j$ and $m=1$ outside a neighbourhood of $Z$. By Propositions \ref{proposition_OP_alpha_p_bis} and \ref{proposition_cordes}, we obtain the desired conditions.
\end{proof}

Let $\sigma_+=\{k_j, k_j \t{ even} \}$ and $\sigma_-=\{k_j, k_j \t{ odd} \}$. As in Corollary \ref{corollaire_surface_epaissie} we apply Theorem \ref{theoreme_principal} and Proposition \ref{proposition_rocade_allegee} to deduce the Reeb periodic orbits after the bypass attachments along the curves $(\gamma_j)_{j=1,\dots,m}$.

\begin{proposition}\label{proposition_orbites_tore_rocade}
Fix $L>0$. Let $(M',\alpha')$ be the contact manifold obtained from $(M_{n,\eta},\alpha_b)$ after bypass attachments along the arcs $(\gamma_j)_{j=1,\dots, m}$ for $K$ large enough. 
\begin{itemize}
  \item For all $1\leq l\leq L$, the Reeb periodic orbits homotopic to $\left[S^1\right]^l$ are the curves $\delta_{2k}^l$ for $1\leq 2k\leq n-1$ and the periodic orbits $\gamma_\mathbf{a}$ associated to $\mathbf{a}=c_{i_1,j}\dots c_{i_k,j}$ with $i_1+\dots+i_k=l$ and $k_j\in\sigma_+$.
  \item For all $-L\leq l\leq -1$, the Reeb periodic orbits homotopic to $\left[S^1\right]^l$ are the curves $\delta_{2k+1}^l$ for $1\leq 2k+1\leq n-1$ and the periodic orbits $\gamma_\mathbf{a}$ associated to $\mathbf{a}=c_{i_1,j}\dots c_{i_k,j}$ with $i_1+\dots+i_k=-l$ and $k_j\in\sigma_-$. 
\end{itemize}
\end{proposition}
We denote by $\mathcal B_j$ the bypass attached to the attaching arc $\gamma_j$.

\subsection{Holomorphic cylinders}\label{section_application_cylindres}
It remains to control the holomorphic cylinders in the symplectisation of the contact manifold $(M',\alpha')$ given in Proposition \ref{proposition_orbites_tore_rocade}.
Let $E^l_j$ be the $\mathbb Q$-vector space generated by the periodic orbit $\delta_{k_j}$ and by the periodic orbits obtained after the bypass attachment along $\gamma_j$ homotopic to $\left[S^1\right]^l$. 
Let $E^l_+$ and $E^l_-$ be the $\mathbb Q$-vector spaces generated by the periodic orbits $\delta_{2k}^l$ for $2k\notin \sigma_+$ and $\delta_{2k+1}^l$ for $2k+1\notin \sigma_-$.
The complex of contact homology is written
\begin{align*}
C_*^{[S^1]^l}(M',\alpha')=\bigoplus_{k_j\in\sigma_+} E^l_j \oplus E^l_+ &\t{ if } l>0,\\
C_*^{[S^1]^l}(M',\alpha')=\bigoplus_{k_j\in\sigma_-} E^l_j \oplus E^l_-&\t{ if } l<0.
\end{align*}
Let $I_b^j =\left[k_j\pi-\frac{7\pi}{4},k_j\pi+\frac{7\pi}{4}\right]$. Consider $I^j_\t{max}$ such that \[S_{Z_j}=[-x_\t{max}+k_j\pi,x_\t{max}+k_j\pi]\times\{1\}\times I^j_\t{max}.\]

\begin{lemma}\label{lemme_courbes_holomorphes_n_quelconque}
For any adapted almost complex structure on the symplectisation of $(M',\alpha')$
\begin{itemize}
  \item $\partial_{\vert E^l_\pm}=0$;
  \item for all $j=1,\dots,m$, we have $\partial(E^l_j)\subset E^l_j$ and any holomorphic cylinder with finite energy and asymptotics in $E_j$ is contained in $\mathcal U_j\cup \mathcal B_j$ (see Figure~\ref{ensemble_U_bis}) where
  \[\mathcal U_j=U_{k_j-1}\cup U_{k_j}\cup U_{k_j+1}\cup\big(I_b^j \times[-1,1]\times I^j_\t{max}\big). \]
\end{itemize}
\begin{figure}[here]
\begin{center}
 \includegraphics{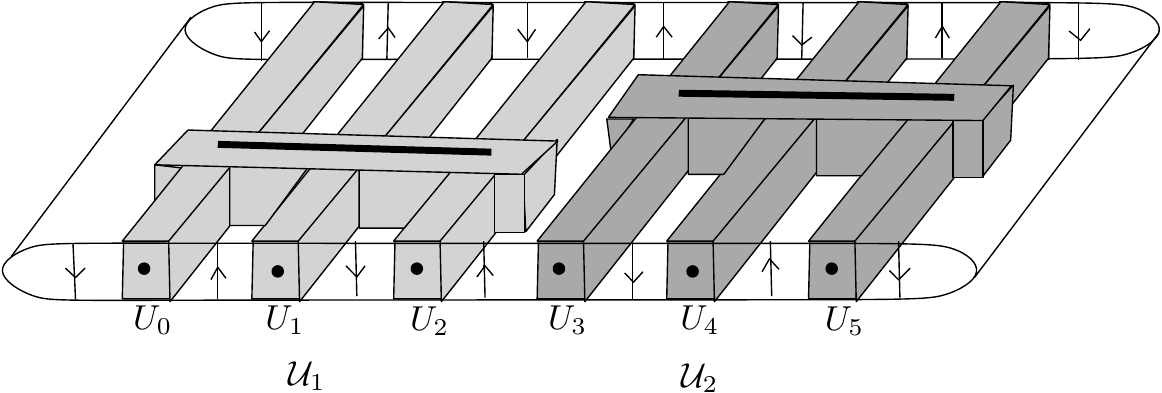}
\end{center}
 \caption{Neighbourhoods $\mathcal U_j$}\label{ensemble_U_bis}
\end{figure}
\end{lemma}

\begin{proof}
Let 
\begin{align*}\mathcal U&=\bigcup_{j=1\dots n} U_{k_j}\cup \bigcup_{j\in\sigma_\pm}\big(I_b^j \times[-1,1]\times I^j_\t{max}\big)\cup \mathcal B_j,\\
\mathcal W&=\left(\left[-\pi+\eta,n\pi-\eta\right]\times[-1,1]\times S^1\right)\setminus \mathcal U.
\end{align*}
The connected components of $\mathcal U$ are the sets $U_l$ for $l\not\in \sigma_\pm\cup(\sigma_\pm-1)\cup(\sigma_\pm+1)$ and $\mathcal U_j$ for $j=1,\dots,m$. In $\mathcal W$, all Reeb orbits are Reeb chords joining the planes $y=1$ and $y=-1$.
Let $\gamma$ be a Reeb periodic orbit (see Proposition \ref{proposition_orbites_tore_rocade}). 
If $\gamma=\delta_j^l$ let $ U_\gamma=U_j$. 
If $\gamma$ is derived from the bypass attachment along $\gamma_j$, let \[U_\gamma =U_j\cup \big(I_b^j \times\{-\epsilon_j\}\times I^j_\t{max}\big)\cup \mathcal B_j.\]
In both cases $\gamma\subset U_\gamma$.
Let $\gamma_+$ and $\gamma_-$ be two homotopic Reeb periodic orbits and $X=\mathcal W \setminus\left(U_{\gamma_+}\cup U_{\gamma_-}\right)$. 
As $\gamma_+$ and $\gamma_-$ are homotopic and $k_{j+1}-k_j\geq 3$,  for any Reeb chord $c$ in $X$ there exists a path of properly embedded arcs in $X $ connecting $c$ and a Reeb chord in $\o(X)$ with reversed orientation. Thus, by positivity of intersection (Proposition \ref{proposition_cordes_holomorphes}), all holomorphic cylinders are contained in $\mathcal U$.
\end{proof}

Let $(M,\xi_0=\ker(\alpha))$ be the contact structure obtained from $(M_{4,\eta},\alpha_b)$ after a bypass attachment along $\gamma_1=[0,2\pi]\times\{1\}\times\left\{\pi\right\}$ (Theorem \ref{theoreme_principal}). Let $\Gamma$ an adapted dividing set of the boundary.
The dividing set of a convex meridian disc contains exactly three connected components which are parallel to the boundary (Proposition \ref{proposition_rocade_tore}). 
By Golovko's result (Theorem \ref{theoreme_Golovko}), the contact homology of $(M,\alpha,\Gamma)$ is generated by two periodic orbits homotopic to $S^1$ and by their multiples. 
We now compare this result with our construction and obtain useful properties of $\partial_{E^l_1}$.
There are two periodic orbits of $R_{\alpha'}$ homotopic to $S^1$ (Proposition \ref{proposition_orbites_tore_rocade}).
If $l>0$, we have $C_*^{[S^1]^l}(M',\alpha')= E^l_+$ and $\partial_{E^l_+}=0$ (Lemma \ref{lemme_courbes_holomorphes_n_quelconque}).
If $l<0$, we have $C_*^{[S^1]^l}(M',\alpha')= E^l_1$. 
Thus $\ker(\partial_{E^l_1})/\im(\partial_{E^l_1})=0$. Let $\overline{\mathcal U}$ denote the set $\mathcal U_1 $ given by Lemma \ref{lemme_courbes_holomorphes_n_quelconque}.

\begin{proof}[Proof of Theorem \ref{homologie_intro_tore_plein}]
Let $\xi$ be a contact structure on $M=D^2\times S^1$ satisfying conditions (C4) and (C5). We choose the contact form given by Proposition \ref{proposition_alpha'}.
The neighbourhoods $\mathcal U_j$ described in Lemma \ref{lemme_courbes_holomorphes_n_quelconque} are contactomorphic to $\overline{\mathcal U}$. Thus  $\ker(\partial_{E^l_j})/\im(\partial_{E^l_j})=0$. Therefore $HC^{[S^1]^{\pm l}}(M,\xi_0,\Gamma)=E^l_\pm$ for $l>0$. As 
\[\dim(E^l_\pm)+ \#\{\sigma_\pm\} + \#\{\pm \t{ extremal bigon}\}=\chi(S_\pm)+ \#\{\sigma_\mp\}, \]
we obtain the desired dimension.
\end{proof}

We now turn to the case of a thickened convex surface. Let $S$ be a convex surface and $\Gamma=\bigcup_{i=0}^n \Gamma_i$ be a dividing set of $S$ without contractible components. Let $\gamma_0$ be an attachment arc in $S$ intersecting three distinct components of $\Gamma$. Let $M=S\times[-1,1]$ be the product neighbourhood of $S$ and $\xi$ be the associated invariant contact structure. 
Choose a contact form $\alpha_b$ of $\xi$ given in Section \ref{section_surface_epaissie}. 
We denote by $(M',\xi')$ the contact manifold obtained from $(M,\xi)$ after a bypass attachment along $\gamma_0\times\{1\}$ and by $\Gamma'$ a dividing set of $\partial M'$.
Fix $K>0$ and apply Theorem \ref{theoreme_principal} to obtain a contact form $\alpha'$ of $\xi'$. Let $\mathcal B$ denote the attached bypass. 
The proof of the following lemma is similar to the proof of Lemma \ref{lemme_courbes_holomorphes_n_quelconque}.

\begin{lemma}\label{lemme_courbes_holomorphes_surface}
For any adapted almost complex structure, the $J$-holomorphic curves in the symplectisation of $(M',\alpha')$ are contained in (see Figure \ref{ensemble_U_ter})
\[\mathcal U=\bigcup_{i=0}^n U_i\cup (I_b\times[-1,1]\times I_\t{max})\cup \mathcal B. \]
\begin{figure}[here]
\begin{center}
 \includegraphics{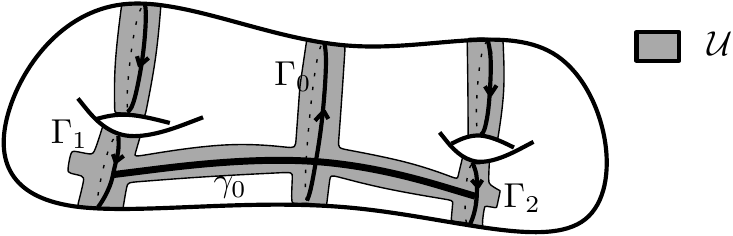}
\end{center}
 \caption{The neighbourhood $\mathcal U$ projected on $S$}\label{ensemble_U_ter}
\end{figure}
In addition, $\mathcal U $ is contactomorphic to $\overline{\mathcal U} $.
\end{lemma}

\begin{proof}[Proof of Theorem \ref{homologie_surface_epaissie_intro}]
Theorem \ref{homologie_surface_epaissie_intro} is a corollary of Lemma \ref{lemme_courbes_holomorphes_surface}. The proof is similar to the proof of Theorem \ref{homologie_intro_tore_plein}.
\end{proof}

\section{Sketch of proof of the bypass attachment theorem}\label{section_sketch}

We now sketch the proof of our main theorem (Theorem \ref{theoreme_principal}). A complete proof is given in Section \ref{section_hyperbolic_proof}.
Fix $K>0$. Let $(M,\xi=\ker(\alpha))$ be a contact manifold with convex boundary $(S,\Gamma)$ and $\gamma_0$ be an attaching arc satisfying condition (C1), (C2) and (C3). To describe the Reeb periodic orbits after a bypass attachment, we study the maps $\phi_B$ and $\psi_M$ induced on $S_Z$ by the Reeb flow in the bypass and in $M$. 
Their domains and ranges consist of rectangles and these maps contract or expand the associated fibres. The maximal invariant set of the composite function is hyperbolic and this function is similar to a ``generalised horseshoe'' (see \cite{KatokHasselblatt95,Moser73}). The Reeb periodic orbits correspond to the periodic points and are given by the symbolic dynamics.

In Section \ref{subsection_notation} we present our notations. We describe the Reeb dynamics in $M$ in Section \ref{subsection_dynamics_M} and in the bypass in Section \ref{subsection_dynamics_B}. In Section \ref{subsection_sketch_proof}, we prove that these dynamic properties indeed give the symbolic description of Reeb periodic orbits. Finally, in Section \ref{subsection_hyperbolic_bypasses}, we sketch the construction of a \emph{hyperbolic bypass}: a bypass with Reeb dynamics described in Section \ref{subsection_dynamics_B}.

\subsection{Notations}\label{subsection_notation}
We say that a (partial) function $\phi:X\to Y$ is \emph{decomposed} into $(\phi_i)_{i\in I}$ if $\dom(\phi_i)$ is a union of connected components of $\dom(\phi)$, $(\dom(\phi_i))_{i\in I}$ is a partition of $\dom(\phi)$ and $\phi_i=\phi_{\vert\dom(\phi_i)}$.

In coordinates $(x,y,z)$, let $S_{y_0}$ denote the plane $y=y_0$, $X^{\leq y_0}=X\cap\{(x,y,z),y\leq y_0\}$, $X^{\geq y_0}=X\cap\{(x,y,z),y\geq y_0\}$ and $X^{[y_0,y_1]}=X^{\leq y_1}\cap X^{\geq y_0}$. For all $0<\lambda<\frac{\pi}{8}$, we consider the following subsets of $S_Z$ (see Figure~\ref{R_epsilon})
\begin{align*}
R_\lambda&=\left(\bigcup_{k=-1}^4\left[\frac{k\pi}{2}+\lambda,\frac{(k+1)\pi}{2}-\lambda\right]\right)\times I_\text{max}, \\
Q_\lambda&=\left(\bigcup_{k=0}^2\left[k\pi-\frac{\lambda}{2},k\pi+\frac{\lambda}{2}\right]\right)\times I_\text{max},\\
X&=\left(\left[0,\frac{\pi}{4}\right]\times\left[-z_\text{max},0)\right)
 \cup\left(\left[\frac{\pi}{4},\frac{3\pi}{4}\right]\times I_\text{max}\right)
 \cup\left(\left[\frac{3\pi}{4},\pi\right]\times(0,z_\text{max}\right]\right),\\
Y&=\left(\left[\pi,\frac{5\pi}{4}\right]\times\left[-z_\text{max},0)\right)
 \cup\left(\left[\frac{5\pi}{4},\frac{7\pi}{4}\right]\times I_\text{max}\right)
 \cup\left(\left[\frac{7\pi}{4},2\pi\right]\times(0,z_\text{max}\right]\right).
\end{align*}
\begin{figure}[here]
\begin{center}
 \includegraphics{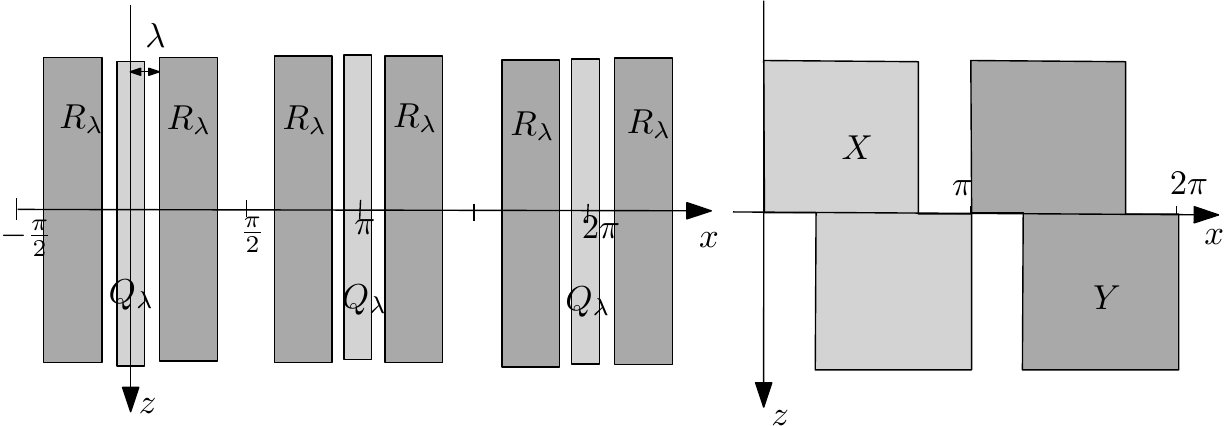}
\end{center}
 \caption{The subsets $R_\lambda$, $Q_\lambda$, $X$ and $Y$}\label{R_epsilon}
\end{figure}
For positive $z_\t{prod}$ and $y_\t{std}$, let $ I_\t{prod}=[-z_\t{prod},z_\t{prod}]$, $S_R=\left[\frac{\pi}{2},\pi\right]\times\{y_\t{std}\}\times I_\t{prod}$ and $S_{\epsilon,k}=\left[\frac{k\pi}{2}-\epsilon,\frac{k\pi}{2}+\epsilon\right]\times I_\text{max}$.
A \emph{rectangle} is a closed set diffeomorphic to $[0,1]\times[0,1]$. This set inherits horizontal and vertical fibres from $[0,1]\times[0,1]$. 

In $(\mathbb R^2,\langle\cdot,\cdot\rangle)$, let $D$ be a straight line and $\nu>0$. The \emph{$\nu$-cone centred at $D$} is the set 
\[\mathcal C(D,\nu)=\left\{w\in\mathbb R^2,\vert\langle w,v\rangle\vert<\nu\vert\langle w,u\rangle\vert \right\}\]
where $u$ is tangent to $D$ and $(u,v)$ is an orthonormal basis. We denote by $H$ and $V$ the horizontal and vertical axes. Let $U$ and $V$ be two open sets in $\mathbb R^2$ and $f:U\to V$ be a diffeomorphism. The image of a cone field $\mathcal C$ on $U$ is the cone field $f_* \mathcal C$ on $V$ defined by 
$\left(f_* \mathcal C\right)_p=\d f_{f^{-1}(p)}\left(\mathcal C_{f^{-1}(p)}\right)$. If $\mathcal C$ and $\mathcal C'$ are two cones fields on $U$ we write $\mathcal C\subset \mathcal C'$ if $\mathcal C_p\subset \mathcal C'_p$  for all $p\in U$. If $z\to\gamma(z)$ is a smooth curve in $\mathbb R^2$, let $\mathcal C_{x,z}(\gamma,\epsilon)=C(\gamma'(z),\epsilon)$ and $\mathcal C_{x,z}(\gamma^\perp,\epsilon)=C(\gamma'(z)^\perp,\epsilon)$.

\subsection{Reeb dynamics in $M$}\label{subsection_dynamics_M}
We now study the Reeb dynamics is the manifold $M$ with boundary.
To attach an adapted bypass we will perturb the contact form $\alpha$. We want to control the map $\psi_M$ induced by the Reeb flow in $M$ for times smaller than $K$ and for the contact form $\alpha$ and perturbations of $\alpha$.
The Reeb chords of $S_Z$ that contribute to $\psi_M$ for the contact form $\alpha$ are close to the Reeb chords of $\gamma_0$. Nevertheless, as $\Gamma\cap Z$ is contained in Reeb orbits, this decomposition is not stable by perturbation and some Reeb chords may appear near  the dividing set.

Let $\lambda>0$, the pair $(S_Z,\lambda)$ is said to be \emph{$K$-hyperbolic} if $\psi_M$ can be decomposed into $\left(\psi_j \right)_{j=0,\dots, N}$ and (see Figure \ref{R_0_et_Q_0}):
 \begin{enumerate}
  \item $\dom(\psi_0)\subset Q_{\lambda}$ and $\im(\psi_0)\subset Q_{\lambda}$;
  \item if $x\in\left[k\pi-\frac{\lambda}{2},k\pi+\frac{\lambda}{2}\right]$ and $(x,z)\in \dom(\psi_0)$ then 
  \begin{itemize}
    \item $(\psi_0)_x(x,z)\in\left[k\pi-\frac{\lambda}{2},k\pi+\frac{\lambda}{2}\right]$;
    \item $(\psi_0)_z(x,z)<z$ if $k$ is odd;
    \item $(\psi_0)_z(x,z)>z$ if $k$ is even;
  \end{itemize}
  \item for all $j\in \llbracket1, N\rrbracket$, $\dom(\psi_j)$ and $\im(\psi_j)$ are rectangles in $R_{\lambda}$ with horizontal fibres and the $\psi_j$ reverse the fibres.
 \end{enumerate}
Note that one can have $\dom(\psi_0)=\emptyset$ or $N=0$.
\begin{figure}[here]
\begin{center}
 \includegraphics{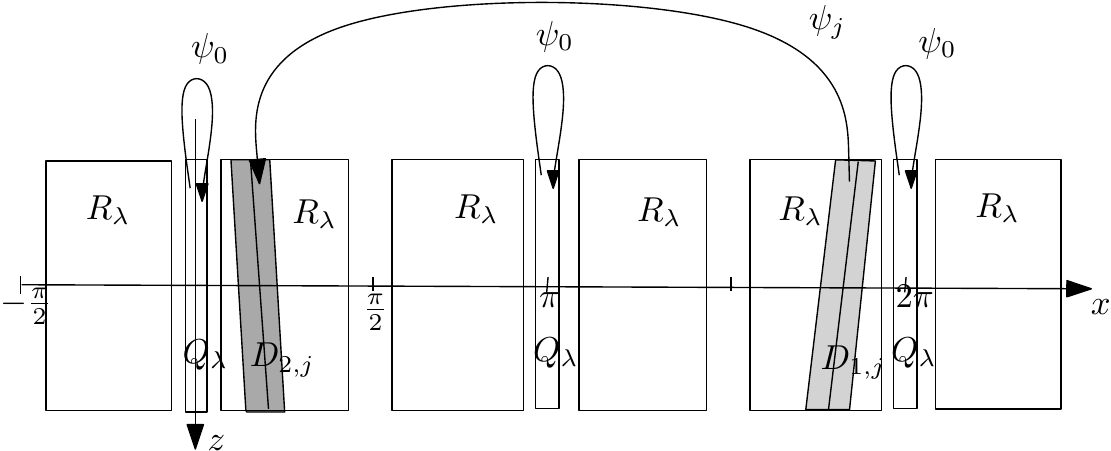}
\end{center}
 \caption{A $K$-hyperbolic surface}\label{R_0_et_Q_0}
\end{figure}
Let  $\mu$, $\nu$ and $\tau$ be real positive numbers.
A $K$-hyperbolic surface is \emph{dominated} by $\omega=(\mu,\nu,\tau)$ if for $j=1,\dots, N$ there exist segments $D_{1,j}\subset \dom(\psi_j)$ and $D_{2,j}\subset \im(\psi_j)$ with boundary on $z=\pm z_\t{max}$ such that (see Figure \ref{R_0_et_Q_0}):
 \begin{enumerate}
  \item the tangents of the vertical fibres of $\dom(\psi_j)$ and $\im(\psi_j)$ are respectively in $\mathcal C(D_{1,j},\nu)$ and $\mathcal C(D_{2,j},\nu)$;
  \item $\left(\psi_j\right)_*\mathcal C(H,\nu)\subset \mathcal C(D_{2,j},\mu)$ and $\left(\psi_j^{-1}\right)_*\mathcal C(H,\nu)\subset \mathcal C(D_{1,j},\mu)$;
  \item  the return time of $\psi_j$ is contained in $(T(a_j)-\tau,T(a_j)+\tau)$.
 \end{enumerate}
Let $\epsilon>0$. A $K$-hyperbolic ($\omega$-dominated) surface $(S_\gamma,\lambda)$ is \emph{$\epsilon$-stable} if for all $\epsilon$-perturbation of $\alpha$ preserving $\gamma_0$, $(S_\gamma,\lambda)$ remains $K$-hyperbolic (and $\omega$-dominated).

\begin{proposition}\label{proposition_retour_variete} Let $\tau>0$ and $\mu>0$. There exists a contact form arbitrarily close to $\alpha$, $z_\t{max}$ small and some real positive numbers $\nu$, $\lambda$ and $\epsilon$ such that
\begin{itemize}
  \item[$\bullet$] $\alpha$ satisfies conditions (C1), (C2) and (C3);
  \item[$\bullet$] $(S_Z,\lambda)$ is $K$-hyperbolic $(\mu,\nu,\tau)$-dominated and $\epsilon$-stable.
\end{itemize}
\end{proposition}

\begin{proof}
After a small perturbation of $\alpha$, we can assume that the images of $\gamma_0\setminus\Gamma$ on $S_Z$ by the Reeb flow and the opposite of the Reeb flow for times smaller that $K$ are transverse to $\gamma_0$. The intersection points correspond to the endpoints of the Reeb chords $a_1,\dots, a_N$. For 
$z_\t{max}$ small enough, the domain and range of $\psi_M$ are contained in a small neighbourhood of the endpoints of the Reeb chords. We choose $\dom(\psi_0)=\emptyset$. Let $D_{1,j}$ and $D_{2,j}$ be the tangent to the image of $\gamma_0\setminus\Gamma$ on $S_Z$ at the endpoints of the Reeb chords $a_1,\dots, a_N$ (see Figure \ref{segments_D}). We obtain the vertical fibres of $\dom(\psi_M)$ and $\im(\psi_M)$ as the inverse images and images of the horizontal segments in $\im(\psi_M)$ and  $\dom(\psi_M)$.
\begin{figure}[here]
\begin{center}
 \includegraphics{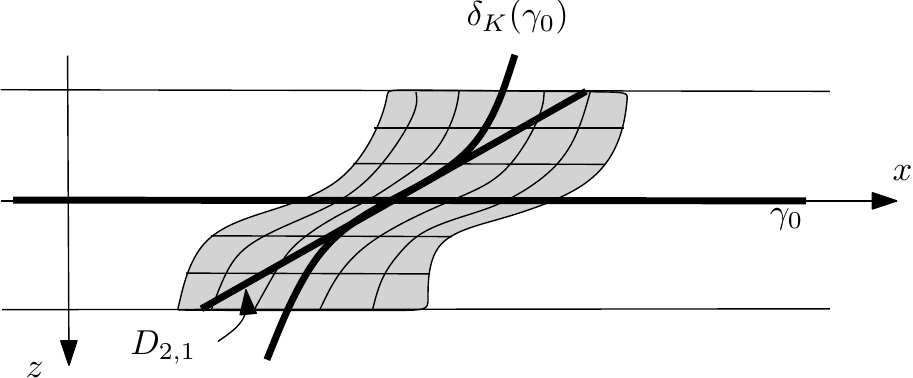}
\end{center}
 \caption{The segments $D_{2,1}$ and the rectangle structure of $\im(\psi_1)$}\label{segments_D}
\end{figure}
For small enough perturbations of~$\alpha$, the structure of $\psi_M$ is preserved outside $Q_\lambda$. In $Q_\lambda\times[-y_\t{max},0]$, the component $\vert R_z\vert$ is close to~$1$ and $\psi_0$ satisfies the desired conditions.
\end{proof}

\subsection{Reeb dynamics in the bypass}\label{subsection_dynamics_B}

We now describe the desired dynamics in the bypass in terms of horizontal rectangles.
A \emph{hyperbolic bypass} in $Z$ is a triple $(\mathcal B,\alpha_B, \lambda)$ where $(\mathcal B,\alpha_B)$ is a contact manifold in ~$Z$ and  $\lambda$ is a real positive number such that 
\begin{enumerate} 
\item $\mathcal B^{\leq 0}=Z^{\leq 0}$ and $\alpha_B$ is adapted to the boundary in $Z^{\geq 0}$;
\item the map $\phi_B$ induced on $S_Z$ by the Reeb flow  in $\mathcal B$ can be decomposed into maps $\phi_0$ and $\phi_1$ such that
\begin{itemize}
  \item $\dom(\phi_0)\subset Q_\lambda$ and $\im(\phi_0)\subset Q_\lambda$;
  \item $\dom(\phi_1)\subset X$ and $\im(\phi_1)\subset Y$;
  \item if $x\in\left[k\pi-\frac{\lambda}{2},k\pi+\frac{\lambda}{2}\right]$ and $(x,z)\in \dom(\phi_0)$ then 
  \begin{itemize}
    \item $(\phi_0)_x(x,z)\in\left[k\pi-\frac{\lambda}{2},k\pi+\frac{\lambda}{2}\right]$;
    \item $(\phi_0)_z(x,z)<z$ if $k$ is odd;
    \item $(\phi_0)_z(x,z)>z$ if $k$ is even;
  \end{itemize}
 \end{itemize}
\item the restriction of $\phi_1$ to $R_\lambda$ can be decomposed into
$(\phi_{i,j})_{i,j\in\{0,1\}}$ (see Figure \ref{T_i_j}) where $\dom(\phi_{i,j}) $ and $\im(\phi_{i,j})$ are rectangle as large as $R_\lambda$ with vertical fibres and the $\phi_{i,j}$ reverse the fibres.
\end{enumerate}
\begin{figure}[here]
\begin{center}
 \includegraphics{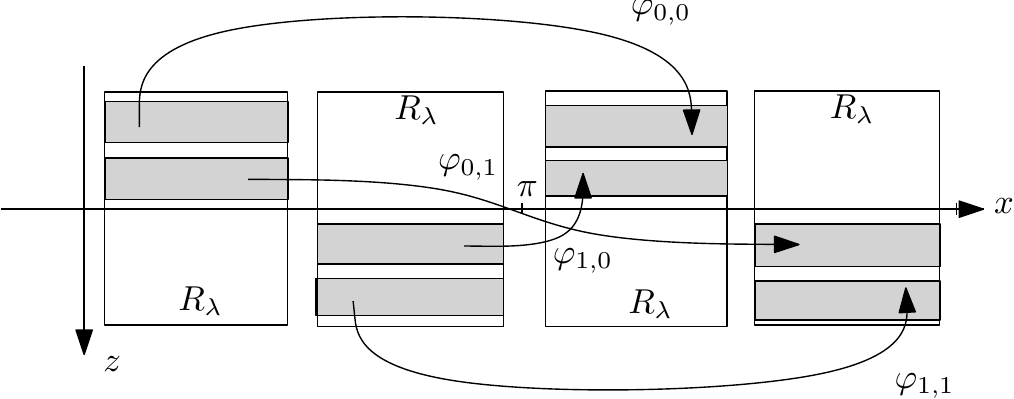}
\end{center}
 \caption{The rectangles $\dom(\phi_{i,j}) $ and $\im(\phi_{i,j})$}\label{T_i_j}
\end{figure}
As in the previous section we want to control the return time and obtain some cone-preservation properties. A hyperbolic bypass $(\mathcal B, \alpha_B,\lambda)$ is \emph{dominated} by $\omega_B=(\nu,\tau, A, \eta )$ if
\begin{enumerate}
 \item $\left(\phi_{i,j}\right)_*\mathcal C(V,A)\subset \mathcal C(H,\nu)$ and $\left(\phi_{i,j}^{-1}\right)_*\mathcal C(V,A)\subset \mathcal C(H,\nu)$;
 \item $\Vert\d\phi_{i,j}(p,v)\Vert>\frac{1}{\eta}\Vert v\Vert$ for all $p\in\dom(\phi_{i,j}) $ and $v\in \mathcal C_p(V,A)$;
 \item $\Vert\d\phi_{i,j}^{-1}(p,v)\Vert>\frac{1}{\eta}\Vert v\Vert$ for all $p\in\im(\phi_{i,j}) $ and $v\in \mathcal C_p(V,A)$;
 \item the return time in $R_\lambda$ is bounded by $8\tau$.
\end{enumerate}
The following theorem is the main ingredient in the proof of Theorem \ref{theoreme_principal}.

\begin{theorem}\label{recollement_K_hyperbolique}
Fix $K>0$. Let $(M,\xi=\ker(\alpha))$ be a contact manifold with convex boundary $(S,\Gamma)$ and $\gamma_0$ be an attaching arc satisfying conditions (C1), (C2) and (C3).
For all real positive numbers $\nu$, $\tau$, $A$, $\eta$, $\epsilon$ and $\lambda$ and for $z_\t{max}$ small enough there exists a hyperbolic bypass $(V_B,\alpha_B, \lambda)$ obtained from $(V,\xi_B)$ after a bypass attachment along $\gamma_0$, dominated by $(\nu,\tau,A,\eta)$ and such that $\alpha_B$ is $\epsilon$-close to $\alpha$.
\end{theorem}

If $(S_Z,\lambda)$ is $K$-hyperbolic, $\epsilon$-stable and $(\mu,\nu,\tau)$-dominated, we call the bypass attachment described in Theorem \ref{recollement_K_hyperbolique} a \emph{K-hyperbolic bypass attachment}.

\subsection{Reeb periodic orbits after a bypass attachment}\label{subsection_sketch_proof}

Before turning to the proof of Theorem \ref{recollement_K_hyperbolique}, we prove that Theorem \ref{recollement_K_hyperbolique} and Proposition \ref{proposition_retour_variete} imply Theorem \ref{theoreme_principal} and Proposition~\ref{proposition_rocade_allegee}. We prove this result in two steps:
\begin{enumerate}
  \item we obtain a symbolic representation of the Reeb flow in restriction to $R_\lambda$;
  \item we prove that all new Reeb periodic orbits intersect $R_\lambda$. 
\end{enumerate}
Fix $K>0$. Let $(M,\xi=\ker(\alpha))$ be a contact manifold with convex boundary $(S,\Gamma)$ and $\gamma_0$ be an attaching arc satisfying conditions (C1), (C2) and (C3). 
We assume $K\neq l({\mathbf{a}})$ for all words $\mathbf{a}$ on the letters $a_1,\dots, a_N$. Thus, there exists $\epsilon_K$ such that $\vert K-l({\mathbf{a}})\vert>\epsilon_K$ for all words $\mathbf{a}$. Let $l_0$ be such that $l(a_i)>l_0$ for all $i=1,\dots, N$. 
There exists $\epsilon_0$ such that these estimations remain satisfied for all $2\epsilon_0$-perturbations of $\alpha$. Without loss of generality $l_0<1$ and $\frac{\epsilon_K}{K}<1$.
Let $\tau<\frac{l_0}{18K}\epsilon_K$. We apply Proposition \ref{proposition_retour_variete} to obtain a $K$-hyperbolic surface, $(\mu,\nu,\tau)$-dominated and $\epsilon$-stable. Without loss of generality $\epsilon<\epsilon_0$, $\nu\leq\mu$ and $\mathcal C(D_{i,j},\mu)\cap \mathcal C(H,\mu)=\{0\}$ for $i=1,2$ and $j=1,\dots, N$. Choose $A>0$ such that $\mathcal C(D_{i,j},\mu)\subset \mathcal C(V,A)$ and $\mathcal C(V,A)\cap \mathcal C(H,\mu)=\{0\}$ for $i=1,2$ and $j=1,\dots, N$ and for all contact forms $\epsilon$-close to $\alpha$. 
Choose $M>0$ such that $\Vert \d\psi_j\Vert<M$ and $\Vert \d\psi_j^{-1}\Vert<M$ for $j=1,\dots, N$ and for all $\epsilon$-perturbations of $\alpha$. Let $\eta<\frac{1}{3M}$. Apply Theorem \ref{recollement_K_hyperbolique}, to obtain a hyperbolic bypass $(M_B,\alpha_B,\lambda)$ dominated by $(\nu,\tau,A,\eta)$ and such that $\alpha_B$ is $\epsilon$-close to $\alpha$.

To obtain a symbolic representation of the new Reeb periodic orbits, we apply a fixed point theorem in hyperbolic situations. The following proposition derives from \cite[Theorem 3.2]{Moser73}.
\begin{proposition}\label{proposition_point_fixe_hyperbolique}
Let $R$ and $R'$ be two rectangles in $[0,1]\times[0,1]$ such that the vertical boundaries of $R$ are contained in $\{0,1\}\times[0,1]$ and the horizontal boundaries of $R'$ are contained in $[0,1]\times\{0,1\}$. Let $F:R\to R'$ be a diffeomorphism such that, for some $A>0$, $\nu>0$ and $a>2$
\begin{itemize}
  \item $F_*\mathcal C(V,A)\subset \mathcal C(V,A)$ and $F^{-1}_*\mathcal C(H,\nu)\subset \mathcal C(H,\nu)$;
  \item $\Vert\d F^{-1}(p,v)\Vert\geq a\Vert v\Vert$ for all $p\in R' $ and $v\in \mathcal C_p(H,\nu)$;
  \item $\Vert\d F(p,v)\Vert\geq a\Vert v\Vert$ for all $p\in R $ and $v\in \mathcal C_p(V,A)$.
\end{itemize}
Then $F$ has a unique fixed point.
\end{proposition}

\begin{proposition}\label{proposition_point_fixe_F_a}
Let ${\mathbf{a}}=a_{i_1}\dots a_{i_k}$ and $F_{\mathbf{a}}=\psi_{i_k}\circ\phi_B\dots\psi_{i_1}\circ\phi_B$ in restriction to $R_\lambda$.
The map $F_{\mathbf{a}}$ has a unique fixed point. The period $T(\gamma_{\mathbf{a}})$ of the associated Reeb periodic orbit $\gamma_{\mathbf{a}}$ satisfies $T(\gamma_{\mathbf{a}})\in[l({\mathbf{a}})-9k\tau,l({\mathbf{a}})+9k\tau]$.
\end{proposition}

\begin{proof}
By induction on $k$, the map $F_{\mathbf{a}}$ can be decomposed into $F^{\mathbf{a}}_{0}$ and $F^{\mathbf{a}}_{1}$ such that (see Figure \ref{Fi_a}) : 
\begin{itemize}
  \item $\im(F^{\mathbf{a}}_{i})$ are rectangles as high as $R_\lambda$ with horizontal fibres;
  \item $\dom(F^{\mathbf{a}}_{i})$ are rectangles with vertical fibres contained in two different components of $R_\lambda$ and as large as the associated component;
  \item ${F_{\mathbf{a}}}_*\mathcal C(V,A)\subset \mathcal C(D_{2,i_k},\mu)$ and ${F_{\mathbf{a}}^{-1}}_*\mathcal C(H,\nu)\subset \mathcal C(H,\nu)$;
  \item $\Vert\d F_{\mathbf{a}}^{-1}(p,v)\Vert\geq\frac{1}{(\eta M)^k}\Vert v\Vert$ for all $p\in \im(F_{\mathbf{a}}) $ and $v\in \mathcal C_p(H,\nu)$;
  \item $\Vert\d F_{\mathbf{a}}(p,v)\Vert\geq\frac{1}{(\eta M)^k}\Vert v\Vert$ for all $p\in \dom(F_{\mathbf{a}})$ and $v\in \mathcal C_p(V,A)$.
\end{itemize}
\begin{figure}[here]
\begin{center}
 \includegraphics{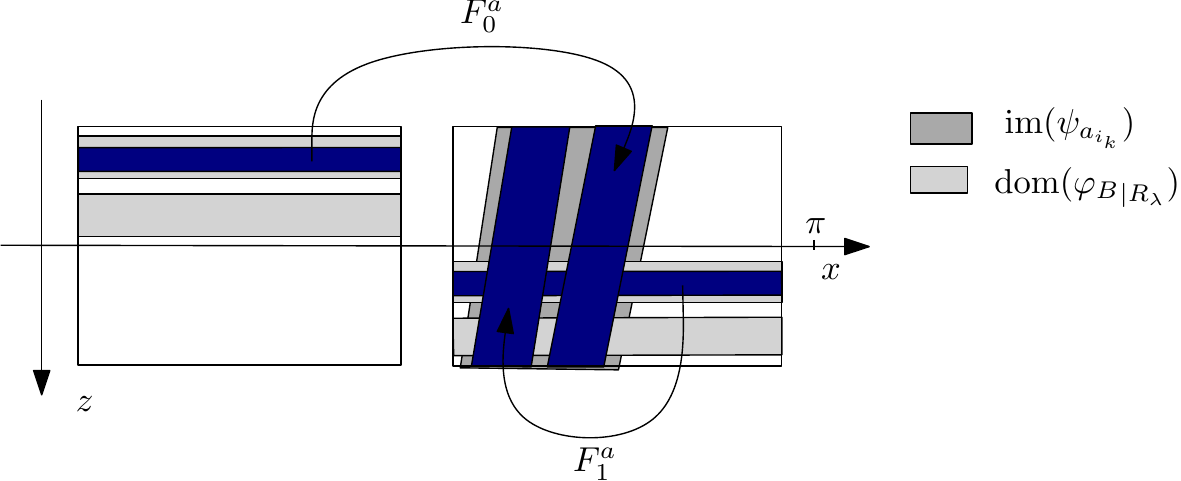}
\end{center}
 \caption{The maps $F^{\mathbf{a}}_{0}$ and $F^{\mathbf{a}}_{1}$}\label{Fi_a}
\end{figure} 
To obtain a unique fixed point, we apply Proposition \ref{proposition_point_fixe_hyperbolique} to the component of $F_{\mathbf{a}}$ such that $\dom(F^{\mathbf{a}}_{i})\cap \im(F^{\mathbf{a}}_{i})\neq \emptyset$. The estimates on the period of the associated Reeb periodic orbit derives from the estimates on the return time.
\end{proof}
We now turn to the second step of the proof.

\begin{proposition}\label{proposition_p_gamma}
Let $\gamma$ be a Reeb periodic orbit intersecting $S_Z$ in $p_\gamma$ and such that $T(\gamma)<K$. Then $p_\gamma\notin Q_\lambda$.
\end{proposition}

\begin{proof}
We control the Reeb orbits intersecting $Q_\lambda$ and prove that they are not periodic. Let $X_k=\left[k\pi-\frac{\lambda}{2},k\pi+\frac{\lambda}{2}\right]$ and $p_0=q_0\in\im(\psi_0)$. As long as these expressions are well-defined, let $p_{2l+1}={\psi_M}^{-1}(p_{2l})$, $p_{2l+2}={\phi_B}^{-1}(p_{2l+1})$, $q_{2l+1}={\phi_B}(p_{2l})$ and $q_{2l+2}={\psi_M}(q_{2l+1})$. Write $p_l=(x_l,y_l)$ and $q_l=(x'_l,y'_l)$. There exists $k$ such that $x_0\in X_k$. The following implications hold.
\begin{itemize}
 \item If $k$ is odd and $z_0\geq0$, then $x_l\in X_k$ and $z_{2l}<z_{2l+1}<z_{2l+2}$.
 \item If $k$ is even and  $z_0\leq0$, then  $x_l\in X_k$ and $z_{2l}>z_{2l+1}>z_{2l+2}$. 
 \item If $k$ is even and $z_0\geq0$, then $x'_l\in X_k$ and $z'_{2l}<z'_{2l+1}<z'_{2l+2}$.
 \item If $k$ is odd and $z_0\leq0$, then $x'_l\in X_k$ and $z'_{2l}>z'_{2l+1}>z'_{2l+2}$.
\end{itemize}
We give a detailed proof in the case $k$ odd and $z_0\geq0$. The proof of the other cases is similar. We prove the result by induction. If $x_{2l}\in X_k$, $z_{2l}\geq0$ and $p_{2l+1}$ is well-defined then $p_{2l}=\psi_M(p_{2l+1})$ and $p_{2l}\in\bigcup_j \im(\psi_i)$. As $p_{2l}\in Q_{\lambda}$, we have $p_{2l}\in \im(\psi_0)$. Therefore, $p_{2l}=\psi_0(p_{2l+1})$ and $p_{2l+1}\in \dom(\psi_0)$. Thus, we obtain $x_{2l+1}\in X_k$ and $z_{2l+1}>z_{2l}$.
If $x_{2l+1}\in X_k$, $z_{2l+1}\geq0$ and $p_{2l+2}$ is well-defined, then $p_{2l+1}=\phi_B(p_{2l+2})$ and $p_{2l+1}\in Y\cup \im(\phi_0)$. As $k$ is odd, we obtain $z_{2l+1}\geq0$ and $p_{2l+1}\in Q_{\lambda}$ and therefore $p_{2l+1}\in \im(\phi_0)$. Thus, we have $p_{2l+1}=\phi_0(p_{2l+2})$ and $p_{2l+2}\in \dom(\phi_0)$. Consequently, $x_{2l+2}\in X_k$ and $z_{2l+2}>z_{2l+1}$.

Let $\gamma$ be a Reeb periodic orbits intersecting $S_Z$ in $p_\gamma\in Q_\lambda$ and such that $T(\gamma)<K$. If $p_\gamma\in S_-$, then $\phi_B(p_\gamma)\in S_+\cap Q_\lambda$. Thus, without loss of generality, we can assume  $p_\gamma\in S_+$. Therefore $p_\gamma\in\dom(\phi_B)\cap\im(\psi_0)$ and $x_\gamma\in X_k$. If $k$ is odd and $z_\gamma\geq0$, then $p_l$ is well defined for all $l\in \mathbb N$, $z_l$ in increasing and $\gamma$ is not periodic. This leads to a contradiction. The proof of the other cases is similar.
\end{proof}

\begin{proof}[Proof of Theorem \ref{theoreme_principal}]
Let ${\mathbf{a}}=a_{i_1}\dots a_{i_k}$ be a word such that $l(\mathbf{a})<K$. By definition of $l_0$, we have $k\leq\frac{K}{l_0}$. Thus $T(\gamma_{\mathbf{a}})\in\left[l({\mathbf{a}})-\frac{\epsilon_K}{2},l({\mathbf{a}})+\frac{\epsilon_K}{2}\right]$
and $T(\gamma_{\mathbf{a}})<K$ (Proposition \ref{proposition_point_fixe_F_a}). 

Conversely, let $\gamma$ be a Reeb periodic orbit intersecting $S_Z$ and such that $T(\gamma)<K$. Let $p_1,\dots, p_k$ denote its successive intersection points with $S_+$ and $q_1,\dots, q_k$ its successive intersection points with $S_-$. By Proposition \ref{proposition_p_gamma}, for all $j=1\dots k$ there exists $i_j$ such that $q_j\in \dom(\psi_{i_j})$ and $p_{j+1}=\psi_{i_j}(q_j)$. Let ${\mathbf{a}}=a_{i_1}\dots a_{i_k}$. Then $p_0$ is the fixed point of $F_{\mathbf{a}}$. 
 Then $l({\mathbf{a}})<K+9k\tau$ (Proposition \ref{proposition_point_fixe_F_a}). Thus $k\leq \frac{K+9k\tau}{l_0}$ and $k\leq \frac{2K}{l_0}$. Therefore $l({\mathbf{a}})<K+\epsilon_K$ and $l({\mathbf{a}})<K$.
\end{proof}

\begin{proof}[Proof of Proposition \ref{proposition_rocade_allegee}]
There exists $\epsilon$ such that for any $\epsilon$-perturbation of $\alpha$, the map $\psi_M$ can be decomposed into $\psi_0$ and $\psi_1$ such that $\psi_0$ has properties similar to those described in the definition of $K$-hyperbolic surface, $\dom(\psi_1)\subset R_{\lambda_0}$ and $\im(\psi_1)\subset R_{\lambda_0}$. Apply Theorem \ref{recollement_K_hyperbolique} for $\lambda=\lambda_0$ and any $\nu$, $\tau$, $A$ and $\eta$. As in Proposition \ref{proposition_p_gamma}, if $\gamma$ is a Reeb periodic orbits intersecting $S_Z$ in $p_\gamma$, then $p_\gamma\notin Q_\lambda$. Thus any Reeb periodic orbit intersects $S_-$ at a periodic point of $\phi_B\circ\psi$.
\end{proof}

\subsection{Hyperbolic bypasses}\label{subsection_hyperbolic_bypasses}

We now give an overview of the proof of Theorem \ref{recollement_K_hyperbolique}. It is the main and last step in the proofs of Theorem \ref{theoreme_principal} and Proposition \ref{proposition_rocade_allegee}. The complete proof is technical and is the subject of Section \ref{section_hyperbolic_proof}.

Honda's construction (see Sections \ref{subsection_bypass} and \ref{subsubsection_contruction_explicite}) provides us with a bypass attachment $(M',\alpha')$ along the attaching arc $\gamma_0$ but $\alpha'$ is not adapted to the boundary. This attachment, if properly performed, does not create any Reeb periodic orbit. Indeed, near $S_Z$, the Reeb vector field is tangent to the planes $x=\t{cst}$ and its slope is $\tan(x)$. Thus, all Reeb orbits intersecting $S_Z$ outside a neighbourhood of $x=\frac{\pi}{2}$ or $x=\frac{3\pi}{2}$ go out of the bypass (see Figure \ref{figure_Reeb_non_adapte}).
\begin{figure}[here]
\begin{center}
 \includegraphics{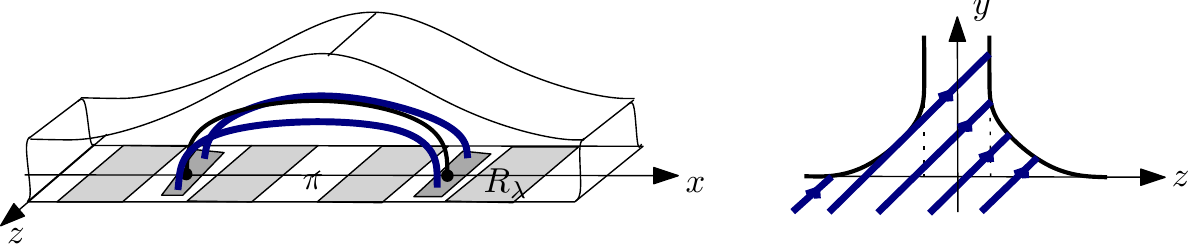}
\end{center}
 \caption{A non-convex bypass attachment}\label{figure_Reeb_non_adapte}
\end{figure} 
Therefore the domain and range of the map induced on $S_Z$ by the Reeb flow in the bypass are contained in neighbourhoods of $x=\frac{\pi}{2}$ and $x=\frac{3\pi}{2}$. By definition of a $K$-hyperbolic surface, there is no new Reeb periodic orbit.

To obtain a contact form adapted to the boundary, we use the \emph{convexification} process described in \cite{CGHH10}. It consists in gluing a small ``bump'' with prescribed contact form along the non-adapted part of the dividing set (see Figure \ref{figure_bosse_convexification}). 
\begin{figure}[here]
\begin{center}
 \includegraphics{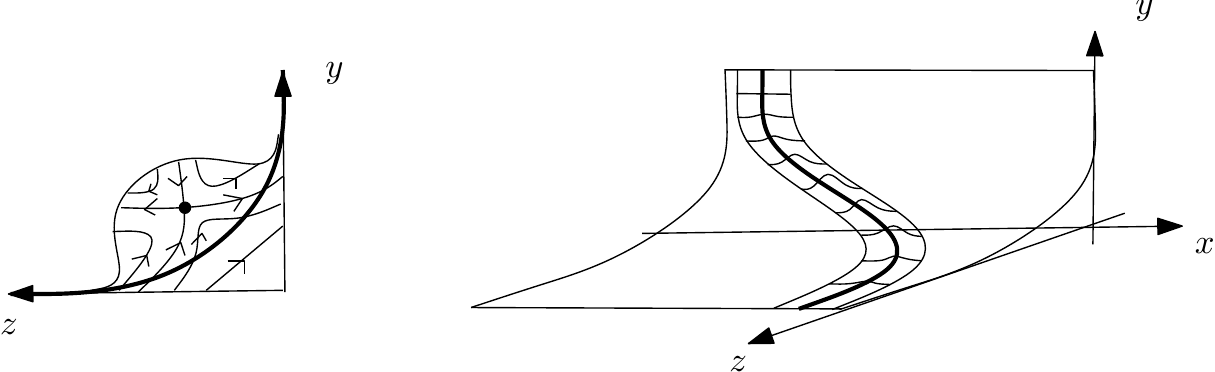}
\end{center}
 \caption{The convexification bump}\label{figure_bosse_convexification}
\end{figure} 
Inside this bump the Reeb vector field is nearly tangent to the dividing set. The restriction of $\phi_B$ to $R_\lambda$ is now non-empty.
\begin{figure}[here]
\begin{center}
 \includegraphics{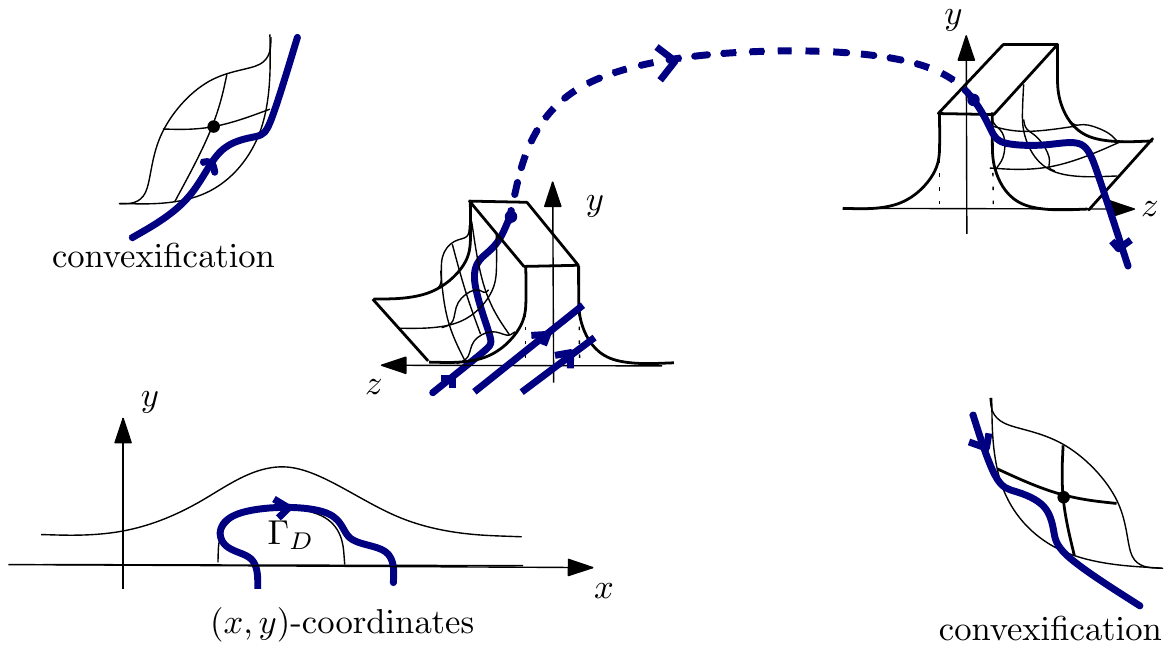}
\end{center}
 \caption{Reeb chords of $R_\lambda$}\label{retour_adapte}
\end{figure} 
The path of the associated Reeb chords is the following
\begin{itemize}
  \item[$\bullet$] they enter the convexification bump and follow the dividing set (see Figure~\ref{retour_adapte} left);
  \item[$\bullet$] then, they reach the area in $M'$ where the Reeb vector field is nearly tangent to the dividing set $\Gamma_D\times\{0\}$ and travel along $\Gamma_D\times\{0\}$ (see Figure \ref{retour_adapte} centre);
  \item[$\bullet$] finally, they go out of the bypass in a similar way and intersect $R_\lambda$ again (see Figure \ref{retour_adapte} right). 
\end{itemize}
To understand the Reeb dynamics and obtain cone-preserving properties, we describe the image of vertical curves on intermediate surfaces (see Figure \ref{retour_courbe_verticale}).
\begin{figure}[here]
\begin{center}
 \includegraphics{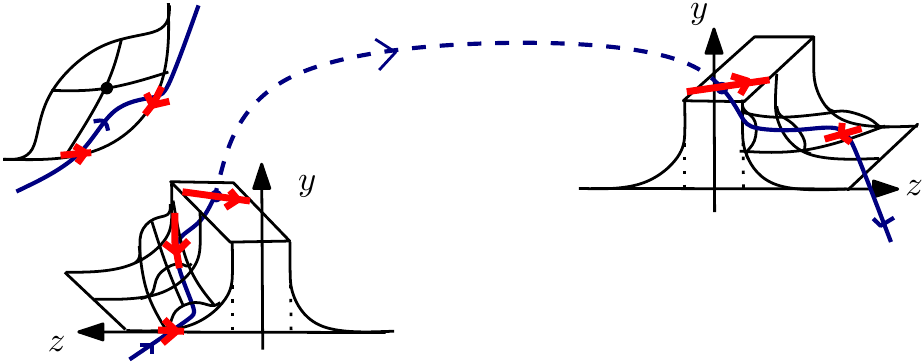}
\end{center}
 \caption{Images of vertical curves in the bypass}\label{retour_courbe_verticale}
\end{figure} 
A vertical curve is stretched into the convexification bump, then transported (and slightly stretched) in the upper part of the bypass. After a last visit to the convexification area, the curve becomes nearly horizontal. The effect on the level of rectangles is shown on Figure \ref{retour_rectangles}.
\begin{figure}[here]
\begin{center}
 \includegraphics{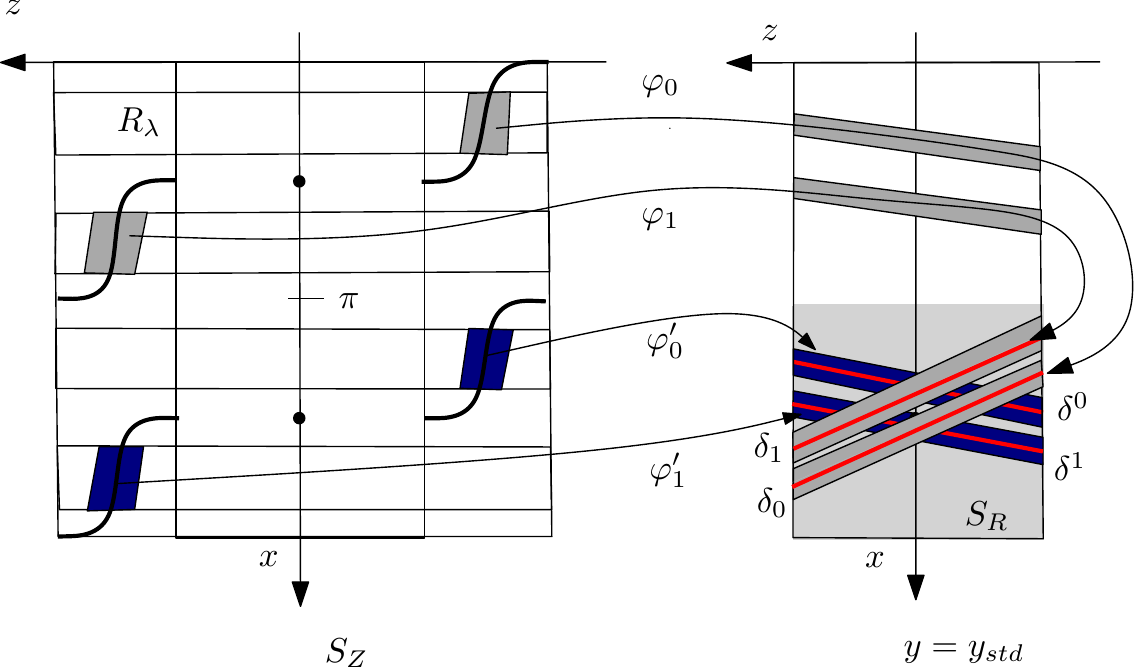}
\end{center}
 \caption{Rectangles on intermediate surfaces}\label{retour_rectangles}
\end{figure}

We now translate these intuitive pictures into more explicit conditions on the Reeb flow in the bypass. Section \ref{section_hyperbolic_proof} is devoted to the construction of a bypass satisfying these conditions. Let $(M,\xi=\ker(\alpha))$ be a contact manifold with convex boundary $(S,\Gamma)$ and $\gamma_0$ be an attaching arc satisfying conditions (C1), (C2) and (C3). We divide the bypass into two regions, the region $y\leq y_\t{std}$ where the contact form is standard and the convexification bump is added and the region $y\geq y_\t{std}$ where the contact structure corresponds to a thickened half overtwisted disc.
We use the notations from Section \ref{subsubsection_contruction_explicite}. 
Fix real positive numbers $K$, $\nu$, $\tau$, $A$, $\eta$, $\epsilon$ and $\lambda<\frac{\pi}{8}$. Let $(\mathcal B,\alpha_B)$ be a bypass such that the boundary $S_B$ is convex, $\alpha_B$ is adapted to $S_B$, $\alpha_B$ is arbitrarily close to $\alpha$ and the Reeb flow satisfies the following properties.
\\\emph{Reeb dynamics restricted to $R_\lambda$.} (To reduce the number of compositions we consider the map induced by the Reeb flow between $S_Z$ and $S_R$.)
\begin{itemize}
  \item[(B1)] There exist real positive numbers $\delta_1$, $\epsilon_R$ and $A_R$, two graphs in $z$ denoted $\delta_0$ and $\delta_1$ and a decomposition $(\phi_0,\phi_1)$ of the map induced between $S_Z$ and $S_R$ for positive time such that
  \begin{itemize}
    \item[$\bullet$] $\dom(\phi_i)$ are rectangles with vertical fibres and basis $[\frac{\pi}{2}+\lambda,\pi-\lambda]$ and $[\pi+\lambda,\frac{3\pi}{2}-\lambda]$ (see Figure \ref{retour_rectangles});
    \item[$\bullet$] $\im(\phi_i)$ are rectangles in an $\epsilon_R$-neighbourhood of $\delta_i$ with horizontal fibres and basis $I_\t{prod}$ (see Figure \ref{retour_rectangles});
    \item[$\bullet$] $\phi_i$ preserves the fibres (see Figure \ref{retour_courbe_verticale});
    \item[$\bullet$] $\left(\phi_i\right)_*\mathcal C(V,A)\subset \mathcal C(\delta_i,\epsilon_R)$ and $\left(\phi_i^{-1}\right)_*\mathcal C(\delta_i^\perp,A_R)\subset \mathcal C(H,\nu)$;
    \item[$\bullet$] $\Vert\d \phi_i^{-1}(p,v)\Vert>\frac{1}{\sqrt\eta}\Vert v\Vert$ for all $p\in\im(\phi_i) $ and $v\in \mathcal C_p(\delta_i^\perp,A_R)$;
    \item[$\bullet$] $\Vert\d \phi_i(p,v)\Vert>\frac{1}{\sqrt\eta}\Vert v\Vert$ for all $p\in\dom(\phi_i) $ and $v\in \mathcal C_p(V,A)$;
    \item the return time is bounded by $4\tau$.
  \end{itemize}
  \item[(B2)] The map induced between $S_R$ and $S_Z$ for positive times can be decomposed into $\phi'_0$ and $\phi'_1$ and there exist two graphs $\delta^0$ and $\delta^1$ satisfying similar properties.
  \item[(B3)]  
  $\mathcal C(\delta_i,\epsilon_R)\subset\mathcal C({\delta^j}^\perp,A_R)$, $\mathcal C(\delta^i,\epsilon_R)\subset\mathcal C(\delta_j^\perp,A_R)$ and $\delta$ and $\delta'$ intersect transversely in one point if $d_{\mathcal C^1}(\delta,\delta_i)<\epsilon_R$ and $d_{\mathcal C^1}(\delta',\delta^j)<\epsilon_R$.
  \end{itemize}
\emph{Reeb dynamics in $\mathcal B^{\geq y_\t{std}}$.}
\begin{itemize}
  \item[(B4)] The domain of the map induced by the Reeb flow on $S_{y_\t{std}}$ is contained in $\left[\frac{\pi}{2}-\frac{\lambda}{2},\frac{\pi}{2}+\frac{\lambda}{2}\right]\times I_\t{prod}$ and its range in $\left[\frac{3\pi}{2}-\frac{\lambda}{2},\frac{3\pi}{2}+\frac{\lambda}{2}\right]\times I_\t{prod}$.
\end{itemize}
\emph{Reeb dynamics in $\mathcal B^{[0,y_\t{std}]}$.}
\begin{itemize}
  \item[(B5)] There is no return map on $S_{y_\t{std}}$.
  \item[(B6)] The return map on $S_Z$ can be decomposed into $\theta_0$, $\theta_1$ and $\theta_2$ such that
  \begin{itemize}
    \item[$\bullet$] $\dom(\theta_k)\subset S_{\frac{\lambda}{2},2k}$ and $\im(\theta_k)\subset S_{\frac{\lambda}{2},2k}$;
    \item[$\bullet$] $(\theta_k)_z(x,z)<z$ if $k$ is odd;
    \item[$\bullet$] $(\theta_k)_z(x,z)>z$ if $k$ is even.
  \end{itemize}
  The Reeb chords which contribute to $\theta_k$ do not intersect a neighbourhood of $S_{y_\t{std}}$.
  \item[(B7)] The map induced between $S_Z$ and $S_{y_\t{std}}$ for positive times can be decomposed into two maps with domains in $X$ and $X+2\pi$ and ranges in $S_{\frac{\lambda}{2},1}$ and $S_{\frac{\lambda}{2},5}$.
   \item[(B8)] The map induced between $S_{y_\t{std}}$ and $S_Z$ for positive times can be decomposed into two maps with domains in $S_{\frac{\lambda}{4},3}$ and $S_{\frac{\lambda}{4},-1}$ and ranges in $Y$ and $Y-2\pi$.
\end{itemize}

\begin{proposition}
The bypass $(\mathcal B,\alpha_B,\lambda)$ is hyperbolic and $(\nu, \tau,A, \eta)$-dominated.
\end{proposition}

\begin{proof}
The map $\phi_B$ can be decomposed into $\phi_0$ and $\phi_1$ where the Reeb chords which contribute to $\phi_0$ do not intersect $S_{y_\t{std}}$ and the Reeb chords which contribute to $\phi_1$ intersect $S_{y_\t{std}}$. The properties of $\phi_0$ derive from condition (B6).
By condition (B5), $\phi_1$ is the the composite of the maps described in conditions (B7), (B4) and (B8). Thus $\dom(\phi_1)\subset X$ and $\im(\phi_1)\subset Y$. By conditions (B4) and (B5), the restriction of $\phi_1$ to $R_\lambda$ is the composite of the maps described in conditions (B1) and (B2). Conditions (B1) and (B3) ensures that the composite map is a map between rectangles. The hyperbolic properties derive from conditions (B1), (B2) and (B3).
\end{proof}

\section{Hyperbolic bypasses}\label{section_hyperbolic_proof}
In this section we construct a bypass satisfying conditions (B1) to (B8) and thus end the proofs of Theorem \ref{theoreme_principal} and Proposition \ref{proposition_rocade_allegee}. Fix some positive numbers  $\lambda<\frac{\pi}{8}$, $y_\text{std} $, $\tau$ and  $z_\text{prod}<z_\text{max}$. In what follows, we have $z_\text{max}\ll 1$.

The construction of a hyperbolic bypass is technical. We start from an explicit contact form on a bypass inspired from Honda \cite{Honda00} (Section \ref{subsubsection_contruction_explicite}). In Section \ref{subsection_preliminary_lemmas}, we present some preliminary lemmas ensuring a precise control of the Reeb flow. Section \ref{subsection_precconvex} presents a preparatory perturbation of the contact structure in the bypass called a \emph{pre-convex bypass}. This pre-convex bypass determines the curves $\delta_i$ and $\delta^j$ (condition (B1)). The actual construction begins in Section \ref{subsection_convexification_coordinates} with the description of adapted coordinates. In Section \ref{subsection_convexification}, we present the convexification contact form in the coordinates described in Section \ref{subsection_convexification_coordinates}. In Section \ref{subsection_conditions} we prove that our construction satisfies the desired conditions.

\subsection{Explicit constructions of bypasses}\label{subsubsection_contruction_explicite}

In this section we present an explicit construction of a bypass attachment. This construction is due to Honda \cite{Honda00} and is the first step of our explicit construction in the proof of Theorem \ref{theoreme_principal}. We construct a contact structure on the product of a smoothed half overtwisted disc and smoothen the product.

Let $(M,\alpha)$ be a contact manifold with convex boundary $(S,\Gamma)$ and $\gamma_0$ be an attachment arc satisfying condition (C1). In coordinates $(x,y)\in I_b\times \mathbb R_+$, let $U_{y}=\left[-\frac{\pi}{6},\frac{13\pi}{6}\right]\times [y,+\infty)$ and $\gamma_1=I_b\times\{0\}$. 
Consider closed set $\mathcal A$ diffeomorphic to a square (see Figure \ref{feuilletage_rocade}) such that $\mathcal A= I_b\times[0,y_0]$ outside $U_0$ and \[\partial\mathcal A=\gamma_1\cup \left(\left\{-\frac{3\pi}{4}\right\}\times[0,y_0]\right)\cup\left(\left\{\frac{11\pi}{4}\right\}\times[0,y_0]\right)\cup\gamma_2.\] Choose a $1$-form $\beta$ on $\mathcal A$ such that
\begin{enumerate}
\item there exists $y_\beta>2y_\text{std}$ such that $\beta=\sin(x)\d y$ in $\mathcal A\setminus U_{2 y_\text{std}}$ and $\beta$ is a positive multiple of $\sin(x)\d y$ for $y\leq y_\beta$;
 \item in $U_{y_\beta}\cap\mathcal A$ the singularities of $\beta$ are
  \begin{itemize}
    \item a half-elliptic negative singularity in $(\pi,y_\beta)$;
    \item two half-hyperbolic singularities in $(0,y_\beta)$ and $(2\pi,y_\beta)$;
    \item two positive elliptic singularities on $\partial \mathcal A$ for $x=0$ and $x=2\pi$;
    \item positive singularities on $\partial \mathcal A\cap U_{y_\beta}$;
  \end{itemize}
  \item\label{condition_decoupage} there exists a smooth proper multi-curve $\Gamma_A$ dividing $\mathcal A$ into two sub-surfaces $\mathcal A_\pm$ such that $\pm\d\beta>0$ on $\mathcal A_\pm$, $\partial A_+$ is oriented as $\Gamma_A$ and $\beta_{\Gamma_A}>0$;
  \item if $\Gamma_D$ is the component of $\Gamma_A$ joining $\left(\frac{\pi}{2},0\right)$ and $\left(\frac{3\pi}{2},0\right)$, there exist coordinates $(r,\theta)\in\left[\frac{3\pi}{2}-\epsilon,\frac{3\pi}{2}+\epsilon\right]\times[0,\theta_\t{max}]=\mathcal U$ near $\Gamma_D$ such that $\Gamma_D\simeq \left\{\frac{3\pi}{2}\right\}\times[0,\theta_\t{max}]$ and $\beta=\sin(r)\d \theta$.
\end{enumerate} 
 \begin{figure}[here]
\begin{center}
 \includegraphics{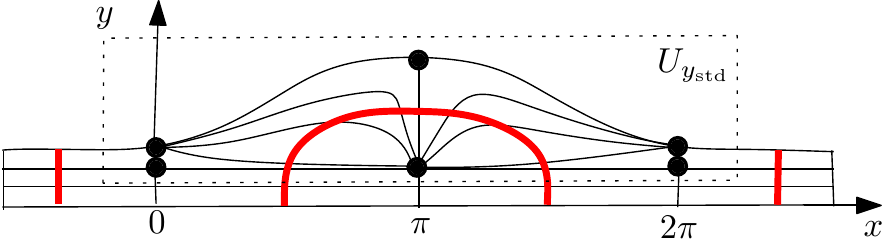}
\end{center}
 \caption{A smoothed half overtwisted disc}\label{feuilletage_rocade}
\end{figure}

\begin{remark}\label{remark_tau} One can assume that $\left\vert\int_{\Gamma_D}\beta\right\vert<\tau$ by replacing $\beta$ by $(1-b(x)c(y))\beta$ where $b$ and $c$ are suitable cut-off functions.
\end{remark}

We now follow \cite{Giroux91} to construct an invariant contact structure on $\mathcal A \times \mathbb R$. Let $\alpha=\beta+f(x,y)\d z$ where
\begin{itemize}
 \item $f(x,y)=\cos(x)$ for $y\leq \frac{3}{2}y_\t{std}$ or $x\notin \left[-\frac{\pi}{3},\frac{7\pi}{3} \right]$;
 \item in $\mathcal U$, the function $f$ depends only on $r$ and is decreasing, in addition $f(r,\theta)=\cos(r)$ near $r=\frac{3\pi}{2}$;
 \item in $U_{y_\beta}\setminus \mathcal U$, $f=\pm1$ if $\pm\d\beta>0$;
 \item\label{cond_interpolation_1} elsewhere $f(\cdot,y)$ has the same variations as $\cos$ and
 \begin{itemize}
  \item if $y\geq 2y_\t{std}$, the function $f$ does not depend on $y$ and interpolates between $\cos$ and~$1$
  \item if $y< 2y_\t{std}$, the function $f$ interpolates between $\cos$ and $f(\cdot,2y_\t{std})$.
 \end{itemize}
 \end{itemize}
We now smooth $\mathcal A\times I_\text{prod}$ to glue it on $S_Z$. There are three types of corners: the convex corners $\gamma_2\times\{\pm z_\text{prod}\}$, the concave corners $\gamma_1\times\{\pm z_\text{prod}\}$ and the corners associated to $x=-\frac{3\pi}{4}$ and $x=\frac{11\pi}{4}$.
We smooth the convex corners using a function \[l_\text{sup} :I_b\times I_\text{prod}\to\mathbb R_+^*\] independent of $x$ for $x\notin\left[-\frac{\pi}{6},\frac{13\pi}{6}\right]$ and such that $l_\text{sup}(\cdot,z) $ is strictly concave with maximum $\gamma_2(x)$ at $z=0$ (see Figure \ref{lissages}). 
Similarly we smooth the concave corners using an even function 
\[l_\text{inf} : [-z_\text{max},-z_\text{prod}]\cup[z_\text{prod},z_\text{max}]\to\mathbb [0,y_\t{smooth}]\] 
decreasing and strictly convex on $[z_\text{prod},z_0]$ and zero on $[z_0,z_\text{max}]$. In addition, we assume $y_\t{smooth}<y_\text{std}<\inf (l_\text{sup})$ and $\im(l_\text{sup})\cap (\mathcal U\times I_\t{prod})=\emptyset$. 
\begin{figure}[here]
\begin{center}
 \includegraphics{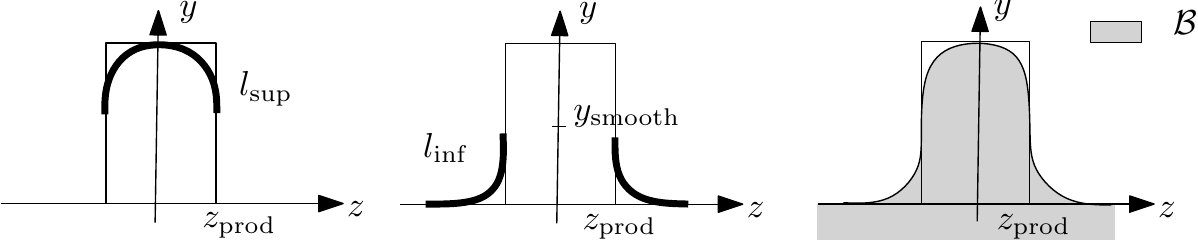}
\end{center}
 \caption{Bypass smoothing }\label{lissages}
\end{figure}
To smooth the remaining corners, we first perturb the previous smoothing for $z<0$ and $x$ close to $-\frac{\pi}{2}$ so that the perturbed boundary is the graph of a function
\[h :\left[-\frac{\pi}{2}-\epsilon_s,-\frac{\pi}{2}+\epsilon_s\right]\times(0,y_0)\to(-z_\t{max},0)\] 
such that $\frac{\partial h}{\partial y}\geq 0$ on $\dom(h)$ and $\frac{\partial h}{\partial y}(x,y)>\eta>0$ for all
$\left\vert x- \frac{\pi}{2}\right\vert\leq \frac{\epsilon_s}{2} $. 
There exists $\epsilon <\epsilon_s$ such that 
$\cotan\left(- \frac{\pi}{2}-\epsilon\right)<\eta$.
We smooth the boundary of the bypass for $x\in\left[-\frac{\pi}{2}-\epsilon,-\frac{\pi}{2}-\frac{\epsilon}{2}\right]$ so that the new boundary is the graph of
\[k :\left[-\frac{\pi}{2}-\epsilon,-\frac{\pi}{2}-\frac{\epsilon}{2}\right]\times[-z_\t{max},z_\t{max}]\to[0,y_0]\] and
\begin{itemize}
  \item $k=0$ for $x$ close to $-\frac{\pi}{2}-\epsilon $;
  \item $0\leq \frac{\partial k}{\partial z}<\frac{1}{\eta}$ for all $z<0$;
  \item $\frac{\partial k}{\partial z}\leq 0$ for all $z\geq 0$.
\end{itemize}
We smooth the boundary for $x$ close to $\frac{5\pi}{2}$ with a similar construction and denote by $\mathcal B$ the smoothed product. Let $M'=M\cup \mathcal B$. Then $M'$ is a smooth manifold with boundary $S'$.

\begin{proposition}\label{Gamma_smooth}
After an arbitrarily small perturbation of the contact form near $\gamma_2\times\{0\}$, the boundary $S'$ is convex. A dividing set, denoted by $\Gamma_\t{smooth}$, is given by the tangency points between the Reeb vector field and $S'$ (see Figure \ref{tangence_Reeb_rocade}).
\end{proposition}

\begin{figure}[here]
\begin{center}
 \includegraphics{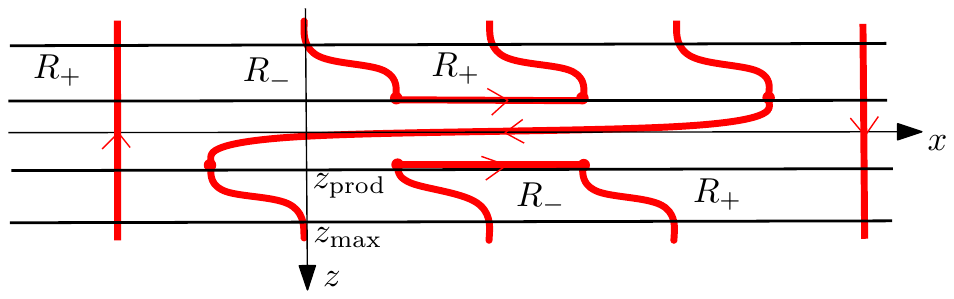}
\end{center}
 \caption{The dividing set $\Gamma_\t{smooth}$}\label{tangence_Reeb_rocade}
\end{figure}

\begin{remark}
The dividing set $\Gamma_\t{smooth}$ is disjoint from the graphs of $h$ and $k$. Indeed, the Reeb vector field is tangent to the graphs of $h$ or $k$ if and only if $\frac{\partial h}{\partial y}(x,y)=\cotan(x)$ or $\frac{\partial k}{\partial z}(x,z)=\tan(x)$.
\end{remark}

\begin{proof}
By use of the implicit function theorem, the set of tangency points between the Reeb vector field and $S'$ is a smooth curve $\Gamma$. In addition the characteristic foliation of $S'$ is positively transverse to $\Gamma$ except along $U_0\cap(\gamma_2\times\{0\})$. 
In a neighbourhood of $U_0\cap(\gamma_2\times\{0\})$, we consider the contact form $\alpha+\epsilon(x,y)\d x$ where $\epsilon$ is a non-positive, small and smooth function that does not depend on $y$ in a neighbourhood of $\left\{-\frac{\pi}{6},\frac{13\pi}{6}\right\}\times\{y_0\}$. 
This perturbation does not change $\Gamma$ and the characteristic foliation of $S'$ is now positively transverse to $\Gamma$ everywhere. We apply Lemma~\ref{lemme_convexite_Reeb} to obtain the convexity of $S'$.
\end{proof}

\begin{proposition}[Honda \cite{Honda00}]
$M'$ is obtained from $M$ after a bypass attachment along $\gamma_1$.
\end{proposition}

The dividing set is pictured on Figure \ref{tangence_Reeb_rocade}. Figure \ref{situation_concave} shows the associated Reeb vector field. In particular, $\alpha$ is not adapted to $S'$.

\begin{figure}[here]
\begin{center}
 \includegraphics{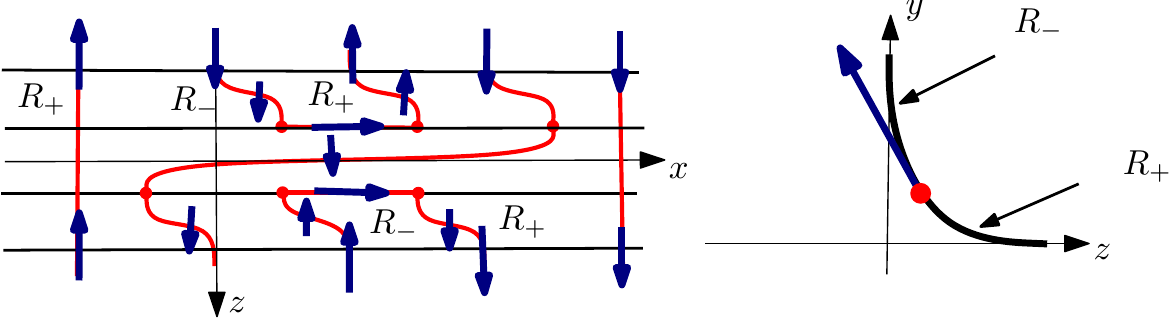}
\end{center}
 \caption{The Reeb vector field along $\Gamma_\t{smooth}$}\label{situation_concave}
\end{figure}

\subsection{Preliminary lemmas}\label{subsection_preliminary_lemmas}
This section can be skipped on first reading. We present bounds on $z_\t{max}$ and on the perturbations of the contact form ensuring a precise control of the Reeb flow.  Let $\Omega=I_b\times [0,y_\text{std}]\times I_\text{max}$ and $\alpha=\sin(x)\d y+\cos(x)\d z$.

\begin{lemma}\label{lemme_lambda}
For small enough perturbations of $\alpha$, the $x$-coordinate of any Reeb orbit in $\Omega$ covers an interval of length at most $\frac{\lambda}{8}$.
\end{lemma}
\begin{proof}
The $x$-coordinate of $R_\alpha$ is zero and the amount of time spent in $\Omega$ by a Reeb orbit is uniformly bounded.
\end{proof}

\begin{lemma}\label{epsilon_stab_retour_S_S_y}
For $z_\t{max}$ small enough and for any small enough perturbation of $\alpha$
\begin{itemize}
  \item condition (B6) is satisfied;
  \item for all $0\leq y\leq \frac{3}{4}y_\t{std}$, the map induced between $S_y$ and $S_{y_\t{std}}$ for positive times can be decomposed into two maps with domains in $S_{\frac{\lambda}{4},1}$ and $S_{\frac{\lambda}{4},5}$ and ranges in $S_{\frac{\lambda}{2},1}$ and $S_{\frac{\lambda}{2},5}$;
  \item the map induced between $S_{y_\t{std}}$ and $S_Z$ for positive times can be decomposed into two maps with domains in $S_{\frac{\lambda}{4},3}$ and $S_{\frac{\lambda}{4},-1}$ and ranges in $S_{\frac{\lambda}{2},3}$ and $S_{\frac{\lambda}{2},-1}$;
  \item the amount of time spent in $\Omega$ by a Reeb orbit is bounded by $2(y_\t{std}+z_\t{max})$.
\end{itemize}
\end{lemma}

\begin{proof}
For $z_\t{max}$ small enough, the domain and range of the map from $S_y$ to $S_{y_\t{std}}$ are contained in $S_{\frac{\lambda}{4},1}$ and $S_{\frac{\lambda}{4},5}$ for all $0\leq y\leq \frac{3}{4} y_\t{std}$. Similarly, the domain and range of the map from $S_{y_\t{std}}$ to $S_y$ are contained in $S_{\frac{\lambda}{4},3}$ and $S_{\frac{\lambda}{4},-1}$. Thus the conditions on the map between $S_y$ and $S_{y_\t{std}}$ are satisfied for any small perturbation of $\alpha$.

We now prove condition (B6) and the return time condition. For any small perturbation of $\alpha$
\begin{itemize}
  \item $\vert R_y\vert\geq\frac{1}{2}$ outside $\left(S_{\frac{\lambda}{8},0}\cup S_{\frac{\lambda}{8},2}\cup S_{\frac{\lambda}{8},4}\right)\times[0,y_\t{std}]$;
  \item $\vert R_z\vert\geq\frac{1}{2}$ into $\left(S_{\frac{\lambda}{2},0}\cup S_{\frac{\lambda}{2},2}\cup S_{\frac{\lambda}{2},4}\right)\times[0,y_\t{std}]$;
  \item the hypotheses of Lemma \ref{lemme_lambda} are satisfied.
\end{itemize}
Thus any Reeb orbit intersecting $S_Z$ outside $S_{\frac{\lambda}{4},0}\cup S_{\frac{\lambda}{4},2}\cup S_{\frac{\lambda}{4},4}$ is not a Reeb chord of $S_Z$ (Lemma \ref{lemme_lambda}). Any Reeb chord of $S_Z$ stays in $S_{\frac{\lambda}{2},2k}$ for some $k$. Along any Reeb orbit, $\vert R_y\vert\geq\frac{1}{2}$ or $\vert R_z\vert\geq\frac{1}{2}$ (Lemma \ref{lemme_lambda}) and we obtain the desired bound on the amount of time spent in $\Omega$.

Finally, the period of any Reeb chord of $S_Z$ is bounded by $2z_\t{max}$. Thus, for small enough perturbations, the $y$-coordinate covers an interval of length at most $2z_\t{max}$ and  $2z_\t{max}<y_\t{std}$ for $z_\t{max}$ small enough. Thus the Reeb chords which contribute to $\theta_k$ do not intersect a neighbourhood of $S_{y_\t{std}}$ and condition (B6) is satisfied.
\end{proof}

\begin{remark}\label{pas_de_retour_sur_S_y_std}
If the contact form is not perturbed near $x=k\pi$ there is no Reeb chord on $S_Z$ or $S_{y_\t{std}}$. More precisely, let  $y\in[0,y_\t{std})$ and $\alpha'$ be a small perturbation of $\alpha$ such that $\alpha'=\alpha$ in $\left(S_{\frac{\pi}{4},0}\cup S_{\frac{\pi}{4},2}\cup S_{\frac{\pi}{4},4}\right)\times[y,y_\t{std}]$. Then there is no Reeb chord of $S_{y_\t{std}}$ contained in $I_b\times[y,y_\t{std}]\times I_\t{max}$. 
\end{remark}

We smooth the corners of $\Omega$ as described in Section \ref{subsubsection_contruction_explicite}. The dividing set is given by the smooth curve $\Gamma_\t{smooth}$ (see Proposition \ref{Gamma_smooth}). We still denote by $\Gamma_\t{smooth}$ its restriction to $\Omega$. The convexification process will take place in a neighbourhood of $\Gamma_\t{smooth}$ and will radically change the contact form. To control the new Reeb orbits, we first control the Reeb chord of a neighbourhood of $\Gamma_\t{smooth}$.

\begin{lemma}\label{lemme_cordes_Reeb}
There exists an arbitrarily small neighbourhood $V_\t{smooth}$ of $\Gamma_\t{smooth}$ such that, for any small perturbation of $\alpha$, the Reeb chords joining two distinct connected components of $V_\t{smooth}$ are contained in $S_{\frac{\lambda}{4},2k}\times [0,y_\t{std}]$. In addition, the orientation of these Reeb chords is given by the sign of the $z$-component of $R_\alpha$.
\end{lemma}

\begin{proof}
There exists $\epsilon<\frac{\lambda}{2}$ such that for any small perturbation of $\alpha$, any Reeb orbit intersecting
$S_{\epsilon,2k+1} \times[0,y_\t{std}]\times[z_\t{prod}-\epsilon,z_\t{max}]$ remains in
$S_{\epsilon,2k+1} \times[0,y_\t{std}]\times\left[\frac{z_\t{prod}}{2},z_\t{max}\right]$ for $k=-1,\dots, 2$. In addition, we ask that the $x$-coordinate of any Reeb orbit in $\Omega$ covers an interval of length at most $\frac{\epsilon}{2}$.

Choose some neighbourhood $V_\t{smooth}$ of $\Gamma_\t{smooth}$ with radius smaller that $\frac{\epsilon}{2}$. Any connected component of $V_\t{smooth}$ is contained in 
\[\left[\frac{k\pi}{2}-\frac{\epsilon}{2}, \frac{(k+1)\pi}{2}+\frac{\epsilon}{2}\right]\times[0,y_\t{std}]\times I_\t{max}\] for $k=-1,\dots, 5$. 
Consider the connected component contained in \[\left[l\pi-\frac{\pi}{2}-\frac{\epsilon}{2}, l\pi+\frac{\epsilon}{2}\right]\times[0,y_\t{std}]\times I_\t{max}.\] Any Reeb chord connecting this component to another is contained in $S_{\epsilon,2l-1}\times [0,y_\t{std}]$ or in $S_{\epsilon,2l}\times [0,y_\t{std}]$. By definition of $\epsilon$, there is no Reeb orbit in $S_{\epsilon,2l-1}\times [0,y_\t{std}]$.
\end{proof}

\begin{lemma}\label{lemme_B4}
Let $(\mathcal A,\beta)$ be a bypass foliation. For $z_\t{max}$ small enough, for any smoothing as described in Section \ref{subsubsection_contruction_explicite} and for any small perturbation $\alpha'$ of $\alpha$ such that $\alpha'=\alpha$ in $\left(S_{\frac{\pi}{4},0}\cup S_{\frac{\pi}{4},2}\cup S_{\frac{\pi}{4},4}\right)\times[y_\t{std}, 2y_\t{std}]$ the condition (B4) is satisfied. In addition, the return time is bounded by $\theta_\t{max}$.
\end{lemma}

\begin{proof}
For $y\in\left[y_\t{std},\frac{3}{2}y_\t{std}\right]$, we have $\alpha=\sin(x)\d y+\cos(x)\d z$. Thus the domain and range of the return map can be made as close to $x=k\pi$ as desired for $z_\t{max}$ small enough. In addition, the only $R_\alpha$-chord is $\Gamma_D\times\{0\}$ and the return time is bounded by~$\theta_\t{max}$.
\end{proof}

\subsection{Pre-convex bypasses}\label{subsection_precconvex}
Let $(\mathcal B,\alpha)$ be a bypass as defined in Section \ref{subsubsection_contruction_explicite}. In what follows we will always assume the the smoothing map $l_\t{inf}$ is invariant by mirror symmetry along the plane $z-z_\t{prod}-y=0$ for $z\geq z_\t{prod}$. We assume $y_\text{smooth}<\frac{y_\text{std}}{2}$.
Recall that the Reeb vector field is tangent to the dividing set $\Gamma_\t{smooth}$ of $\partial \mathcal B=S'$ when $S'$ is vertical and points toward $S'_-$ in the concave part of $S'$
(Proposition \ref{Gamma_smooth}). To apply the convexification process we first ``eliminate'' the tangency points between $R_\alpha$ and $\Gamma_\t{smooth}$ by perturbing $\alpha$ to obtain a \emph{pre-convex bypass}. We use the symmetries of the bypass to extend local constructions. In particular, if $(x,y,z)\in\left[\frac{\pi}{4},\frac{3\pi}{4}\right]\times[0,y_\t{std}]\times[z_\t{prod},z_\t{prod}+y_\t{std}]$, let
\begin{equation}\label{eq_sigma}
\sigma(x,y,z)=\left(-x+\frac{3\pi}{2}, z-z_\text{prod},y+z_\text{prod}\right)
\end{equation}
be the rotation of angle $\pi$ with axis $x=\frac{3\pi}{4}$, $y= z-z_\text{prod}$. For  $(x,y,z)\in\left[-\frac{3\pi}{4},\frac{\pi}{4}\right]\times[0,y_\t{std}]\times[0,z_\t{max}]$, let 
\begin{equation}\label{eq_tau}
\tau(x,y,z)=\left(-x,y,-z\right)
\end{equation}
be the rotation of angle $\pi$ with axis $x=z=0$.
Let $\Gamma_A^+=\Gamma_A\times\{z_\text{prod}\}$, $\Gamma_A^-=\Gamma_A\times\{-z_\text{prod}\}$ and $\Gamma_A^\pm=\Gamma_A^-\cup\Gamma_A^+$. We use similar notations for $\Gamma_D$. 

A \emph{pre-convex} perturbation $(k_\t{sup},k_\t{inf})$ of a bypass $(\mathcal B,\alpha)$ is composed of two smooth maps $\mathcal B\to \mathbb R^*_+$ such that (see Figure \ref{V_+_et_V_-}):
\begin{itemize}
  \item there exist a neighbourhood $V_+$ of the restriction of $\left(\Gamma_A^\pm\right)$ to $y\geq \frac{y_\text{std}}{2}$ and a neighbourhood $V_-$ of the restriction of $\left(\Gamma_A^\pm\right)$ to $y\leq \frac{y_\text{std}}{2}$ such that $k_\t{sup}=1$ outside  $V_+$ and $k_\t{inf}=1$ outside $V_-$;
  \item $\frac{\partial k_\t{sup}}{\partial z}>0$ near $\Gamma^+_A$ and $\frac{\partial k_\t{sup}}{\partial z}<0$ near $\Gamma^-_A$;
  \item $k_\t{sup}$ does not depend on $x$ near $\Gamma_A\setminus\Gamma_D$ and on $r$ near $\Gamma_D$;
  \item $k_\t{inf}(x,y,z)=(1-f_\text{inf}(y)\rho z)$ near $\Gamma_D^+\cap\{(x,y,z), x\in \left[\frac{\pi}{4},\frac{3\pi}{4}\right]\}$ where $\rho>0$, $f_\text{inf} : [0,y_\text{std}]\to\mathbb R_+$, $f_\text{inf}=0$ near $0$ and for $y\geq \frac{y_\text{std}}{2}$, $f_\text{inf}$ is increasing on $[0,y_\rho^-]$, $f_\text{inf}=1$ on $[y_\rho^-,y_\rho^+]$ and  is decreasing on $[y_\rho^+,\frac{y_\text{std}}{2}]$;
  \item $k_\t{sup}$ is $\tau$-invariant and ``$\pi$-periodic'' for $y\in[0,y_\text{std}]$ and $k_\t{inf}$ is $\tau$-invariant and ``$\pi$-periodic''.
\end{itemize}
\begin{figure}[here]
\begin{center}
 \includegraphics{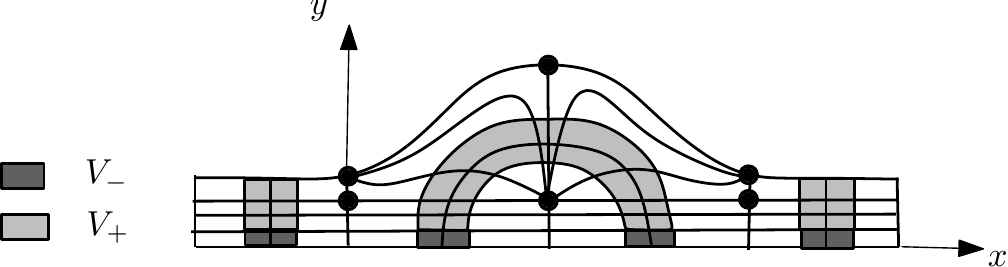}
\end{center}
 \caption{The neighbourhoods $V_+$ and $V_-$}\label{V_+_et_V_-}
\end{figure} 

In what follows, we assume $y_\rho^+<\frac{2y_\text{std}}{5}$, $y_\rho^->y_\text{smooth}$, and that the radius of $V_-$ is smaller than $\frac{\lambda}{2}$.
Let $k'_\t{sup}=k_\t{sup}\circ \sigma$. We extend $k'_\t{sup}$ in a neighbourhood of $\mathcal B$ to obtain a $\pi$-periodic, $\tau$-invariant function. We define $k'_\t{inf}$ similarly. Let
\[\alpha_\text{prec}=(k_\t{sup}k_\t{inf})\sin(x)\d y+(k'_\t{sup}k'_\t{inf})\cos(x)\d z.\]
For $k_\t{sup}$ and $k_\t{inf}$ close to $1$, $\alpha_\text{prec}$ is a contact form. Let $\eta(k_\t{sup})$ and $\eta(k_\t{inf})$ be the radius of the neighbourhoods of $\Gamma_A^\pm$ where $k_\t{sup}$ and $k_\t{inf}$ do not depends on $x$ or $r$.

\begin{proposition}\label{decoupage_preconvexe}
Fix $\eta>0$. Let $(\mathcal B,\alpha)$ be a bypass.
There exist real positive numbers $\epsilon_\t{inf}$ and $\epsilon_\t{sup}$ such that for any pre-convex perturbation $(k_\t{sup},k_\t{inf})$ satisfying $\eta(k_\t{sup})\geq\eta$, $\eta(k_\t{sup})\geq\eta$, $\Vert k_\t{sup}-1\Vert<\epsilon_\t{sup}$ and $\Vert k_\t{inf}-1\Vert<\epsilon_\t{inf}$ the following holds
\begin{itemize}
  \item $S'$ is convex with dividing set $\Gamma_\text{prec}$, the set of tangency points between $R_\text{prec}$ and $S'$;
  \item $R_\text{prec}$ points toward $S_+$ for $y>\frac{y_\t{std}}{2}$ and toward $S_-$ for $y<\frac{y_\t{std}}{2}$.
\end{itemize}
\end{proposition}
Note that $\epsilon_\t{inf}>0$ and $\epsilon_\t{sup}>0$ depend on the smoothing of $\mathcal A\times I_\t{prod}$.

\begin{proof}
The tangency condition is open. By Proposition \ref{Gamma_smooth} and Lemma \ref{lemme_convexite_Reeb} the first condition is satisfied for small enough perturbations of $\alpha$.

The Reeb vector field $R_\alpha$ is transverse to $\Gamma_\t{smooth}$ for $x\neq \frac{k\pi}{2}$ as the tangents to $\Gamma_\t{smooth}$ have a non-vanishing $x$-component. Thus, for small enough perturbations, the transversality still holds for $\left\vert x-\frac{k\pi}{2}\right\vert\geq \eta$.

We study the case $y<\frac{y_\t{std}}{2}$. We prove transversality for $\left\vert x-\frac{\pi}{2}\right\vert< \eta$ and extend the result by symmetry. The Reeb vector field is 
\begin{equation}\label{eq_R_inf}
R_\alpha=\frac{1}{k_\t{inf}(y,z)}\left(\begin{array}{c}
 -\frac{\partial k_\t{inf}}{\partial z}(y,z)\sin(x) \\
\sin(x) \\
k_\t{inf}(y,z)\cos(x)
\end{array}\right)
\end{equation}
and the tangency condition between $R_{\alpha_\text{prec}}$ and $S'$ is \[k_\t{inf}(y,z)\cos(x)-L'_\t{inf}(y)\sin(x)=0\] where $L_\t{inf}=l^{-1}_\t{inf}$ (see Section \ref{subsubsection_contruction_explicite}). Thus a parametrisation of $\Gamma_\text{prec}$ is given by
\[y\mapsto\left(\left(\frac{\cos}{\sin}\right)^{-1}\left(\frac{L'_\t{inf}(y)}{k_\t{inf}(y,L_\t{inf}(y))}\right),y,L_\t{inf}(y)\right)=(x(y),y,z(y)).\]
The tangency condition between $R_{\alpha_\text{prec}}$ and $\Gamma_\text{prec}$ implies \[x'(y)+\frac{\partial k_\t{inf}}{\partial z}\big(y,z(y)\big)=0.\] Yet $\frac{\partial k_\t{inf}}{\partial z}\big(y,z(y)\big)<0$ and $x'(y)$ has the sign of $-L''_\t{inf}(y)$. Thus $x'(y)<0$ for $k_\t{inf}$ close to $1$ and $R_{\alpha_\text{prec}}$ is not tangent to $\Gamma_\text{prec}$.

We now study the case $y>\frac{y_\t{std}}{2}$. The tangency condition between $R_{\alpha_\text{prec}}$ and $S'$ is \[k_\t{sup}(y,z)\cos(x)-L'_\t{sup}(y)\sin(x)=0.\]
The tangency condition between $R_{\alpha_\text{prec}}$ and $\Gamma_\text{prec}$ implies \[x'(y)+\frac{\partial k_\t{sup}}{\partial z}\big(y,z(y)\big)=0.\] Yet $\frac{\partial k_\t{sup}}{\partial z}\big(y,z(y)\big)>0$ and $x'(y)$ has the sign of $L''_\t{sup}(y)$.
\end{proof}

The convexification process takes place near $\Gamma_\t{prec}$. We carefully choose $(\mathcal B,\alpha)$ and a pre-convex perturbation to control the image of the dividing set by the Reeb flow. For $s\in[0,\frac{\pi}{2}]$, let
\[\begin{array}{l}
\gamma_0(s)=(s,0,-z_\text{prod})\\
\gamma_1(s)=\left(\pi-s,0,z_\text{prod}\right)\\
\gamma^0(s)=\left(s+\pi,0,-z_\text{prod}\right)\\
\gamma^1(s)=\left(2\pi-s,0,z_\text{prod}\right).
\end{array} \]
We denote by $p_V$ and $p_R$ the images on $S_Z$ and $S_R$ by the Reeb flow.

A quadruple $(\mathcal B,\alpha,k_\t{sup},k_\t{inf})$ is a \emph{pre-convex bypass} if $(\mathcal B,\alpha)$ is a bypass as defined in Section \ref{subsubsection_contruction_explicite}, $(k_\t{sup},k_\t{inf})$ is a pre-convex perturbation and there exist positive real numbers $\epsilon_R$, $A_R$ and $\epsilon_Z$ ($\epsilon_Z$ is arbitrarily small), four $z$-graphs in $S_R$ denoted by $\delta_0$, $\delta_1$, $\delta^0$ and $\delta^1$ and product neighbourhoods 
\begin{equation}\label{eq_V_A}
V_{\Gamma_A}=V_x\times \left[0,\frac{3}{4}y_\t{std}\right]\times V_z
\end{equation}
of the restriction of $\left(\Gamma_A^\pm\right)$ to $y\in\left[0,\frac{3y_\t{std}}{4}\right]$ with $\vert V_x\vert=\lambda_B<\lambda$ such that
\begin{enumerate}
  \item\label{conv_V_Gamma_A} there is no Reeb chord between $V_{\Gamma_A}$ and $S_{y_\text{std}}$; 
  \item\label{cond_Gamma_S_Z} if $\Gamma_B$ is the restriction of $\Gamma_\t{prec}$ for $x\in\left[\frac{\pi}{2}+\frac{\lambda_B}{2},\pi-\frac{\lambda_B}{2}\right]$ then  $\left\vert\int_{\Gamma_{B}}\alpha\right\vert <\tau$ and $d_{\mathcal C^1}(p_Z(\Gamma_B),\gamma_1)<\frac{\epsilon_Z}{2}$;
  \item\label{cond_delta_i} $\delta_i$ is increasing and $\delta^i$ decreasing, the graphs intersect the segment $r=\frac{3\pi}{2}$;
  \item\label{cond_delta}\label{cond_verticale} $\mathcal C(\delta_i,\epsilon_R)\cap\mathcal C(H,\epsilon_R)=\{0\}$ and $\mathcal C(\delta^i,\epsilon_R)\subset\mathcal C(\delta_j^\perp,A_R)$ and $\delta^i$ satisfies symmetric conditions; 
  \item\label{cond_delta'} if $d_{\mathcal C^1}(\delta,\delta_i)<\epsilon_R$ and $d_{\mathcal C^1}(\delta',\delta^j)<\epsilon_R$ then $\delta$ and $\delta'$ intersect transversely in one point ;
  \item\label{cond_S_Z_S_R} if $\gamma$ is a curve in $S_Z$ such that $d_{\mathcal C^1}(\gamma,\gamma_i)<\epsilon_Z$ or $d_{\mathcal C^1}(\gamma,\gamma^i)<\epsilon_Z$ then either $d_{\mathcal C^1}(p_R(\gamma),\delta_i)<\epsilon_R$ or $d_{\mathcal C^1}(p_R(\gamma),\delta^i)<\epsilon_R$;
  \item\label{cond_tau} the return time between $S_Z$ and $S_R$ is bounded by $\tau$.
\end{enumerate}

\begin{proposition}\label{proposition_existence_pre_convexe}
Let $(M,\xi=\ker(\alpha))$ be a contact manifold with convex boundary $(S,\Gamma)$ and $\gamma_0$ be an attaching arc satisfying condition (C1), (C2) and (C3). Let $(\mathcal A,\beta)$ be a bypass foliation (see Section \ref{subsubsection_contruction_explicite}).
For $z_\t{max}$ small enough, there exists a pre-convex bypass $(\mathcal B,\alpha,k_\t{sup},k_\t{inf})$ with $k_\t{sup}$ and $k_\t{inf}$ arbitrarily close to $1$.
\end{proposition}

The rest of this section is devoted to the proof of Proposition \ref{proposition_existence_pre_convexe}. Let $\eta<\frac{\lambda}{4}$. Choose $l_\t{sup}$ and $k_\t{sup}$ such that $\eta(k_\t{sup})\geq\eta$ and $\Vert k_\t{sup}-1\Vert<\epsilon_\t{sup}$.
\begin{lemma}
If $k_\t{sup}$ is close to $1$ and $z_\t{max}$ small enough, then $p_R(\gamma_i)$ and $p_R(\gamma^i)$ satisfy condition (\ref{cond_delta_i}).
\end{lemma}

\begin{proof}
For $z_\t{max}$ small, the Reeb chords of $\alpha$ that contribute to the map between $S_Z$ and $S_R$ are contained in the neighbourhood where $\alpha=\sin(r)\d\theta+\cos(r)\d z$ and $k_\t{sup}$ does not depend on $r$. In coordinates $(r,\theta,z)$, the image of $\gamma^1$ on $S_\theta$ is 
\[s\longmapsto\left(2\pi-s,\theta,-\frac{\cos(s)}{\sin(s)}\theta+z_\t{prod}\right).\]
For $z_\t{max}$ small enough, this curve in a decreasing graph in $z$ and contains $(\frac{3\pi}{2},\theta,z_\t{prod})$. 
Consider the perturbation $\alpha_\t{sup}$ of $\alpha$ associated to $k_\t{sup}$. As
\[R_\alpha=\frac{1}{k_\t{sup}(y,z)}\left(\begin{array}{c}
 -\frac{\partial k_\t{sup}}{\partial z}(y,z)\sin(x) \\
\sin(x) \\
k_\t{sup}(y,z)\cos(x)
\end{array}\right),\] 
we have $\frac{\partial k_\t{sup}}{\partial z} =0$ for $y\leq \frac{y_\t{std}}{2}$ and $\frac{\partial k_\t{sup}}{\partial z}>0$ for $z>\frac{y_\t{std}}{2}$ near $\Gamma_D^+$. Therefore, the curve $p_R(\gamma^1)$ intersects the segment $r=\frac{3\pi}{2}$. The proof is similar in the other cases.
\end{proof}
Let $p_R(\gamma_i)=\delta_i$ and $p_R(\gamma^i)=\delta^i$. The curves $\delta_i$ and $\delta^j$ intersect transversely in exactly one point and there exist $\epsilon_R$, $A_R$ and $\epsilon_Z$ satisfying conditions (\ref{cond_delta}) and~(\ref{cond_delta'}) for any small perturbation of $\alpha_\t{sup}$. Additionally, all $R_{\alpha_\t{sup}}$-orbits intersecting $\left\{\frac{\pi}{2}\right\}\times\left[0,\frac{3}{4}y_\t{std}\right]\times\{z_\t{prod}\}$ go out of the bypass. Then, there exists $V_{\Gamma_A}$ such that the Reeb vector field of all small perturbation of $\alpha_\t{sup}$ satisfies condition (\ref{conv_V_Gamma_A}). 
We now carefully choose $l_\t{inf}$. A parametrisation of $\Gamma_B$ is 
\[\begin{array}{ccc}
\left[\frac{\pi}{2}+\frac{\lambda_B}{2},\pi-\frac{\lambda_B}{2}\right]&\longrightarrow & \mathbb R^3\\
x &\longmapsto & \left(x,l_\t{inf}\left(\big(l'_\t{inf}\big)^{-1}\big(\tan(x)\big)\right),\big(l'_\t{inf}\big)^{-1}\big(\tan(x)\big)\right)
\end{array}\]
thus $p_Z(\Gamma_B)$ is parametrised by
\[x\longmapsto \left(x,0,\big(l'_\t{inf}\big)^{-1}\big(\tan(x)\big)-\cotan(x)\times l_\t{inf}\left(\big(l'_\t{inf}\big)^{-1}\big(\tan(x)\big)\right)\right) \]
and its derivative is
\[x\longmapsto \left(1,0,-\frac{1}{\sin^2(x)}\times l_\t{inf}\left(\big(l'_\t{inf}\big)^{-1}\big(\tan(x)\big)\right)\right).\]
Thus, the second half of condition (\ref{cond_Gamma_S_Z}) is satisfied for $l_\t{inf}$ small.
In addition, we have
\[\int_{\Gamma}\alpha=\int_{\frac{\pi}{2}+\frac{\lambda_B}{2}}^{\pi-\frac{\lambda_B}{2}} \frac{1}{\sin^3(x)}\times\frac{1}{l''_\t{inf}\big((l'_\t{inf})^{-1}(\tan(x))\big)}\d x.\]
Fix $C>0$ and $c>0$ we chose $l_\t{inf}$ such that $l_\t{inf}<c$ and $l''_\t{inf}\big((l'_\t{inf})^{-1}(y)\big)>C$ for all
\[y\in\left[\tan\left(\frac{\pi}{2}+\frac{\lambda_B}{2}\right),\tan\left(\pi-\frac{\lambda_B}{2}\right)\right].\]
For $C$ big enough and $c$ small enough, condition (\ref{cond_Gamma_S_Z}) is satisfied.
Additionally, such a function exists: it is the anti-derivative of a function $f :(0,1]\longrightarrow(-\infty,0]$ such that 
\begin{itemize}
 \item $f$ is increasing and $\left\vert\int_0^1 f\right\vert=c$;
 \item for all $k\in\bb N$, $f^{(k)}(1)=0$ and $\lim_{x\to 0}\vert f^{(k)}(x)\vert=\infty$;
 \item $f'\big((f)^{-1}(y)\big)>C$ for all $y\in\left[\tan\left(\frac{\pi}{2}+\frac{\lambda_B}{2}\right),\tan\left(\pi-\frac{\lambda_B}{2}\right)\right]$.
\end{itemize}
Finally, we choose $k_\t{inf}$ small enough. Condition \ref{cond_tau} derives from Lemmas \ref{epsilon_stab_retour_S_S_y} and \ref{lemme_B4} and Remark \ref{remark_tau}. This concludes the proof of Proposition \ref{proposition_existence_pre_convexe}.

\subsection{Convexification coordinates}\label{subsection_convexification_coordinates}
In the two following sections, we describe the actual construction of an hyperbolic bypass.
\begin{construction}
There exists a pre-convex bypass $(\mathcal B,\alpha, k_\t{sup}, k_\t{inf})$ adapted to $\tau$ and $\lambda$ (Proposition \ref{proposition_existence_pre_convexe}). For technical reasons, we also consider a second pre-convex bypass $(\mathcal B',\alpha,k_\t{sup},k_\t{inf})$ extending $\mathcal B$ and such that, $d_{\mathcal C^0}(\partial \mathcal B',\mathcal B)>0 $ for all $\vert z\vert \leq z_\text{prod}+y_\text{std}$ and $0\leq y \leq y_\text{std}$.  We denote by $p'_Z$ and $p'_R$ the projections in $\mathcal B'$. Without loss of generality $\epsilon_{Z}<\nu$ and
$p_Z(\dom({p_R}_{\vert\Gamma_B})\subset\dom({p'_R}_{\vert S_Z}) $.
\end{construction}
We are now in position to apply the convexification process described in \cite{CGHH10}. 
Recall that $S'$  is the boundary  of $M'=M\cup\mathcal B$, $k_\t{inf}=1-\rho z$ on $[y^-_\rho, y^+_\rho]$,
\[y_\t{smooth}<y^-_\rho< y^+_\rho<\frac{2}{5}y_\t{std},\]
 and the upper boundary of $V_{\Gamma_A}$ is contained in $y=\frac{3}{4}y_\t{std}$. Choose 
\[ y^-_\rho<y^-< y^+< y^+_\rho.\]
Our first step is to obtain nice coordinates near $\Gamma_\t{prec}\cap \mathcal B^{\leq y_\t{std}}$. We construct these coordinates near the connected component $\Gamma_0$ contained in $\left[\frac{\pi}{2},\pi\right]\times [0,y_\text{std}]\times [0,z_\text{std}]$ and extend them using the symmetries of $\mathcal B$.

\begin{fact}\label{epsilon'_stab}
There exists $\epsilon_\t{prec}$ such that for all $\epsilon_\t{prec}$-perturbation of $S'$, the contact form $\alpha_\t{prec}$ is adapted to the new boundary for $y\geq \frac{2}{3}y_\t{std}$.
\end{fact}

\begin{proposition}\label{coordonnees_t_u_tau}
There exists a surface $\Sigma$ with coordinates $(u,v)\in[u_\t{min},u_\t{max}]\times[-v_\t{max},v_\t{max}]$ such that
\begin{enumerate}
  \item\label{c_1} $\Sigma$ is transverse to $R_{\alpha_\t{prec}}$, $\Sigma\cap S'=\Gamma_0$ and the intersection is transverse;
  \item\label{c_2} $\sigma(i_\Sigma(u,v))=i_\Sigma(u_\t{max}+u_\t{min}-u,v)$  where $i_\Sigma : \Sigma\to\mathbb R^3$ is the inclusion and  $\sigma$ is defined by equation (\ref{eq_sigma});
  \item\label{c_3} $\frac{\partial }{\partial u}=\frac{\partial }{\partial y}$ and $\frac{\partial }{\partial v}=\frac{\partial }{\partial z}$ for $y$ close to $[y^-,y^+]$;
  \item\label{c_4} $\alpha_\t{prec}=\d t+(1-\rho z_\t{prod}-\rho v)\d u$ in the flow-box coordinates $(t,u,v)$ associated to $\Sigma$. 
\end{enumerate} 
\end{proposition}

We denote by $y_\Sigma $ the coordinate such that  $\left(\frac{\pi}{2},y_\Sigma\right)=i_\Sigma(u_\t{max},0)$.
\begin{proof}
Choose a surface $\Sigma$ satisfying conditions (\ref{c_1}), (\ref{c_2}) and (\ref{c_3}) and such that $i_\Sigma^ *\alpha_\t{prec}=g(u,v)\d v$ and $g(u,0)=1-\rho z_\t{prod}$. Then $g(u,v)=1-\rho z_\t{prod}-\rho v$ for $y$ close to $[y^-_\rho,y^+_\rho]$ and $\alpha_\t{prec}=g(u,v)\d u+\d t$. Moser's trick provides us with a diffeomorphism $\phi_1$ such that $\phi_1=\Id$ along $\Gamma_\t{prec}$ and for $y$ close to $[y^-,y^+]$, $\phi_1$ preserves $\Sigma$ and $\phi_1^ *\alpha_\t{prec}=\d t+(1-\rho z_\t{prod}-\rho v)\d u$ as the two Reeb vector fields coincide. In addition $\phi_1\circ\sigma=\sigma\circ\phi_1$. Thus ${\phi_1}_{\vert\Sigma}$  has the desired properties.
\end{proof}

Let $u_\rho^\pm$ and $u^\pm$ be the $u$-coordinates associated to the intersection points between $\Gamma_0$ and $S_{y_\rho^\pm}$ or $S_{y^\pm}$.
Let $\psi=(\psi_x,\psi_y,\psi_z)$ be the diffeomorphism associated to the change of coordinates. Let $S_\mathcal{B}=\psi^{-1}(S')$. Without loss of generality $\psi:\mathcal{U}\to\mathcal V$\label{mathcal_V} and\footnote{See Lemma \ref{lemme_cordes_Reeb} for the definition of $ V_\t{smooth}$.}
\begin{gather} 
 \mathcal U=I_t\times I_u\times I_v=[-t_\t{max},t_\t{max}]\times[u_\t{min},u_\t{max}]\times[-v_\t{max},v_\t{max}],\\\label{eq_V}
\mathcal V\subset\Big(\left[\frac{\pi}{2}-\frac{\lambda}{8},\pi+\frac{\lambda}{8}\right]\times\left[0,\frac{3}{4}y_\t{std}\right]\times I_\t{max}\Big)\cap V_\t{smooth}\cap\mathcal B',\\
  \alpha=(1-\rho z)\sin(x)\d y+\cos(x)\d z \t{ on } \psi\left(I_t\times[u^-,u^+]\times I_v\right).
\end{gather}  

\begin{fact}\label{lemme_phi_u} 
For all $(t,u,v)\in I_t\times [u^-,u^+]\times I_v$, it holds that
\begin{align*}
\sigma\big(\psi(t,u,v)\big)&=\psi\big(-t,u_\t{max}+u_\t{min}-u,v\big),\\
 \psi(0,u,v)&=\left(\frac{\pi}{2},y_\Sigma-u_\t{max}+u,z_\t{prod}+v\right),\\
 \psi(t,u,v)&=\psi(t,0,v)+(0,y_\Sigma-u_\t{max}+u,0),\\
 \psi(-t,u,v)&=\big(\pi-\psi_x(t,u,v),2\psi_y(0,u,v)-\psi_y(t,u,v), \psi_z(t,u,v)\big).
\end{align*} 
\end{fact}
We rewrite our objects and conditions in coordinates $(t,u,v)$.
We subdivide $[u_\t{min},u_\t{max}]$ into $u_\t{min}<u_1<u_2<u_3<u_4<u_5<u_\t{max}$ (see Figure \ref{figure_coordonnees_t_u_tau}) such that 
\begin{gather}
 \label{eq1}\psi(I_t\times[u_5,u_\t{max}]\times I_v)\subset V_{\Gamma_A}\\
 \label{eq2}\left[\frac{\pi}{2},\frac{\pi}{2}+\frac{\lambda}{2}\right]\subset \psi_x(I_t\times[u_4,u_\t{max}]\times I_v),\\
 \label{eq3}\left[\frac{\pi}{2}+\frac{3\lambda}{4},\pi-\frac{3\lambda}{4}\right]\subset\psi_x(I_t\times[u_2,u_3]\times I_v),\\
 \left[\pi-\frac{\lambda}{2},\pi\right]\subset \psi_x(I_t\times[u_\t{min},u_1]\times I_v).
\end{gather}
\begin{figure}[here]
\begin{center}
 \includegraphics{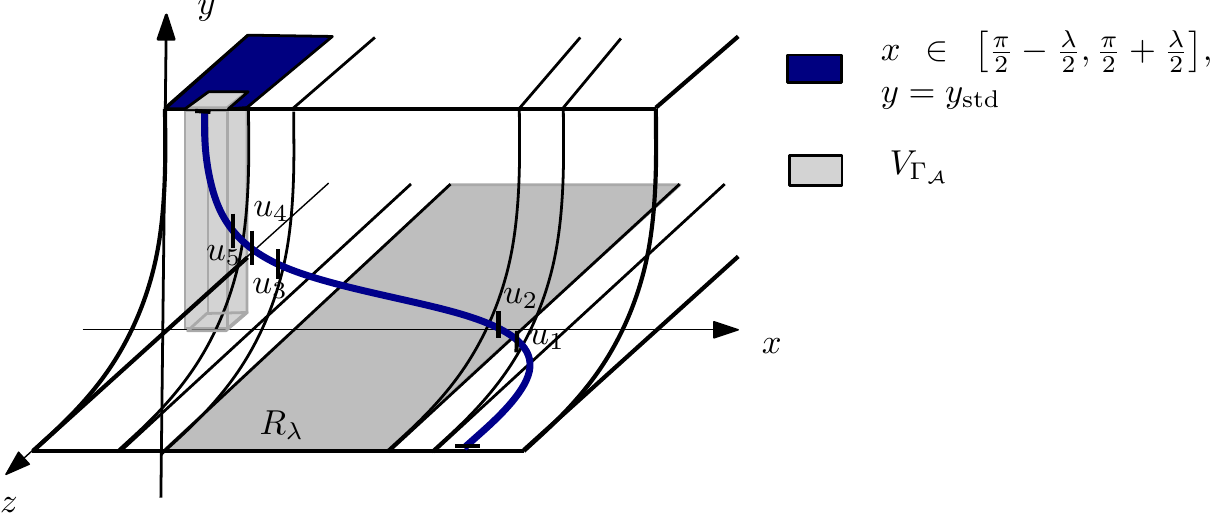}
\end{center}
 \caption{The subdivision $u_i$}\label{figure_coordonnees_t_u_tau}
\end{figure}

\begin{lemma}\label{lemme_S_B}
Without loss of generality we may assume that (see Figure \ref{figure_S_h})
\begin{enumerate}
  \item $S_\mathcal B$ is a smooth surface contained in $\mathring{I_t}\times I_u\times [0,v_\t{max}]$ and
  \begin{enumerate}
    \item its restriction to the plane $u=\t{cst}$ is a smooth curve composed of two graphs containing $(0,0)$, one positive and increasing and the other negative and decreasing on $(0,v_\t{max}]$;
    \item $S_\mathcal B$ is $u$-invariant and invariant by the mirror symmetry along the plane $t=0$ for $u\in[u^-,u^+]$;
  \end{enumerate}
  \item\label{c_Reeb_chord} there is no Reeb chord of $\partial \mathcal V$ outside $\mathcal V$;
  \item $C_\psi=\Vert \psi(t,u,v)-\psi(0,u,0) \Vert_{\infty} +\Vert \ d\psi(t,u,v)-\ d\psi(0,u,0) \Vert_{\infty}\ll 1$:
 \end{enumerate}
\end{lemma}

\begin{proof}
$S_\mathcal B$ is a smooth surface containing $\{0\}\times [u_\t{min},u_\t{max}]\times\{0\}$, tangent to $\Vect\left(\frac{\partial}{\partial u},\frac{\partial}{\partial t}\right)$ along this curve and transverse to $\frac{\partial}{\partial t}=R_{\alpha_\t{prec}}$ elsewhere. In addition, $\frac{\partial}{\partial t}$  is positively transverse to $S_\mathcal B$ for $t<0$ and negatively transverse for $t>0$ (see Figure \ref{graphe_negatif_positif}).
\begin{figure}[here]
\begin{center}
 \includegraphics{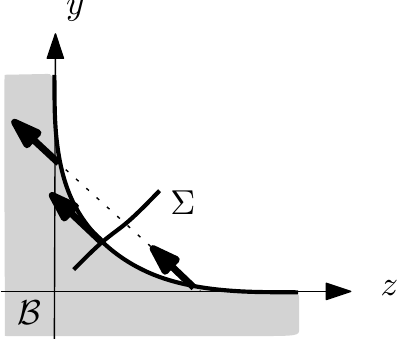}
\end{center}
 \caption{The Reeb vector field and $S_\mathcal{B}$}\label{graphe_negatif_positif}
\end{figure}

To prove condition (\ref{c_Reeb_chord}), we first note that there is no Reeb chord of $\Gamma_\t{prec}$ in $\mathcal B^{<\frac{3}{4}y_\t{std}}$. Indeed, in the set where $\alpha_\t{prec}=\sin(x)\d y+\cos(x)\d z$, we have $R_x=0$ and $\Gamma_\t{prec}$ intersects the planes $x=\t{cst}$ in one point. By symmetry, it remains to prove the result for $x$ close to $\frac{\pi}{2}$ and $z>0$. In this set, we have $R_x>0$, $R_y>0$. The projection of $\Gamma_\t{prec}$ on the plane $(y,x)$ is decreasing and there is no Reeb chord of $\Gamma_\t{prec}$ in $\mathcal B^{<y_\t{prec}}$.

If there exists a sequence $\big(\gamma_n\big)_{n\in\mathbb N^*}$ of Reeb chords of $\partial \mathcal V_n$ where the radius of $\mathcal V_n$ is smaller that $\frac{1}{n}$, then the endpoints of $\gamma_n$ converge to $\Gamma_\t{prec}$. In addition, the period of these chords is bounded (Lemma \ref{epsilon_stab_retour_S_S_y}) and bounded below by $t_\t{max}$ (associated to the maximal $t$-coordinate in $\mathcal V_1$). Thus $\gamma_n$ converges to a Reeb chord of $\Gamma_\t{prec}$. This leads to a contradiction and condition \ref{c_Reeb_chord} is proved.
\end{proof}

\begin{lemma}\label{lemme_S_B_1}
\label{c_Sigma} Without loss of generality we may assume that there exist real positive numbers $\epsilon_B$ and $B$ such that for all $t_\Sigma\in I_t$, the maps $\phi_-$ and $\phi_+$ induced by the Reeb flow between $\Sigma=\{(t,u,v),t=t_\Sigma\}$ and $S_Z$ and between $\Sigma$ and $S_R$ satisfy:
  \begin{enumerate}
    \item $\mathcal C(H,\epsilon_B)\cap \mathcal C(V,B)=\{0\}$;
    \item\label{phi_-}\label{phi_-_-1} $(\phi_-)_*\left(\mathcal C_p(H,\epsilon_B)\right)\subset\mathcal C_{\phi_-(p)}(H,\nu)$ and $(\phi_-^{-1})_*\left(\mathcal C_{\phi_-(p)}(V,A)\right)\subset\mathcal C_p(V,B)$ for all $p\in[u_2,u_3]\times I_v$;
    \item\label{phi_+}\label{phi_+_-1} $(\phi_+)_*\left(\mathcal C_p(H,\epsilon_B)\right)\subset \mathcal C_{\phi_+(p)}(\delta_1,\epsilon_R)$ and $(\phi_+^{-1})_*\left(\mathcal C_{\phi_+(p)}(\delta_1^\perp,A_R)\right)\subset\mathcal C_p(V,B)$ for all $p\in[u_4,u_5]\times I_v$;
    \item\label{phi_+'} $\phi_+([u_4,u_5]\times I_v)\subset \{(x,z),\vert x-\delta_1(z)\vert<\epsilon_R\}$;
    \item\label{retour_phi_+_-} the return time between $S_Z$ and $\Sigma$ is bounded by $\tau$ and  by $2\tau$ between $S_R$ and $\Sigma$.
  \end{enumerate}

\end{lemma}
Let $L$ be a bound of $\Vert\d\phi_\pm\Vert$ and $\Vert\d{\phi_\pm}^{-1}\Vert$ where $\Vert \cdot\Vert$ is defined in the coordinates $(x,y,z)$ in $S_R$ and $S_Z$ and in the coordinates $(t,u,v)$ in $\Sigma$.

\begin{proof}
Let $\Gamma_u=\Gamma_\t{prec}\cap\psi\left(I_t\times[u_2,u_3]\times I_v\right)$. Then $\Gamma_u\subset\Gamma_B$ and all the Reeb chords between $\Gamma_B$ and $S_Z$ have an endpoint in $\Gamma_u$. There exists $\epsilon_B>0$ such that if $d_{\mathcal C^1}(\gamma,\Gamma_u)<\epsilon_B$ then  
$d_{\mathcal C^1}(p_Z(\gamma),p_Z(\Gamma_u))<\frac{\epsilon_Z}{2}$. Thus the fist half of condition~(\ref{phi_-}) derive from conditions (\ref{cond_Gamma_S_Z}) and (\ref{cond_S_Z_S_R}) in the definition of pre-convex bypasses for $\mathcal U$ small enough. As $p_Z(\dom({p_R}_{\vert\Gamma_B})\subset\dom({p'_R}_{\vert S_Z}) $, condition (\ref{phi_+'}) and the first part of condition (\ref{phi_+}) derive from conditions (\ref{cond_Gamma_S_Z}) and (\ref{cond_S_Z_S_R}) in the definition of pre-convex bypasses. As 
$(\phi_-^{-1})_*\left(\mathcal C(V,A)\right)\cap\mathcal C(H,\epsilon_B)=\emptyset$ there exists $B$ satisfying the second part of condition (\ref{phi_-}). The proof of the second part of condition (\ref{phi_+}) is similar.
\end{proof}
We now perturb $S'$ to obtain a $u$-invariant surface.
\begin{construction}
We perturb $S'$ (see Figure \ref{figure_perturbation_S_B}) so that there exists $v_0\in(0,v_\t{max})$ satisfying
\begin{itemize}
  \item ${S_\mathcal B}_{\vert u}={S_\mathcal B}_{\vert u_\t{max}}$ on $[0,v_0]$;
  \item $S_\mathcal B$ contains $\{0\}\times [u_\t{min},u_\t{max}]\times\{0\}$ and is tangent to $\Vect\left(\frac{\partial}{\partial u},\frac{\partial}{\partial t}\right)$ along this curve;
  \item the Reeb vector field is positively transverse to $S_\mathcal B$ for $t<0$, negatively transverse for $t>0$.
\end{itemize}
\begin{figure}[here]
\begin{center}
 \includegraphics{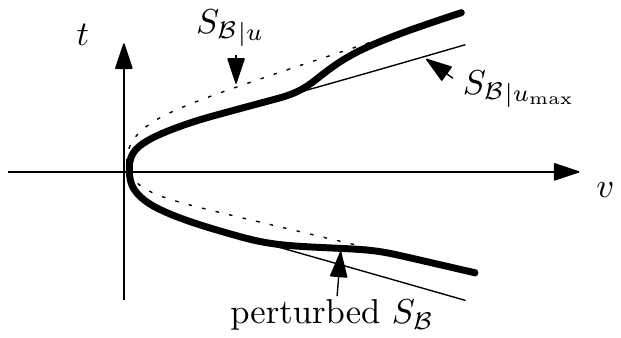}
\end{center}
 \caption{The perturbation of $S'$}\label{figure_perturbation_S_B}
\end{figure}
These condition are automatically satisfied for $u\in[u_\t{min},u_1]$ and $u\in[u_5,u_\t{max}]$. Let $\mathcal B$ denote the new bypass. For $\mathcal U$ small enough, this perturbation is $\epsilon_\t{prec}$-small. In what follows, let $I_v=[-v_0,v_0]$ and $v_\t{max}=v_0$.  
\end{construction}

\subsection{Convexification}\label{subsection_convexification}
The convexification process consists in adding a ``bump'' with prescribed contact structure in a neighbourhood of $\Gamma_\t{prec}$. We first describe the new boundary in Section \ref{subsubsection_boundary}. In Section \ref{subsubsection_model}, we present the contact structure in the convexification and in Section \ref{subsubsection_perturbation} we modify this model to obtain the desired cone-preserving properties. Recall that we construct the convexification near the connected component $\Gamma_0$ contained in $\left[\frac{\pi}{2},\pi\right]\times [0,y_\text{std}]\times [0,z_\text{std}]$.

\subsubsection{Perturbed boundary}\label{subsubsection_boundary}
We perturb the boundary for $y\leq\frac{3}{4}y_\t{std}$. 
For $[y^-,\frac{3}{4}y_\t{std}]$, the new boundary is the graph of a function. Let $\mathcal{H}$ be the set of smooth functions 
\[h: \left[\frac{\pi}{2}-\frac{\lambda}{3},\frac{\pi}{2}+\frac{\lambda}{3}\right]\times\left[y^-,\frac{3}{4}y_\t{std}\right]\to\mathbb R\] 
such that
\begin{enumerate}
 \item $\Vert h\Vert_{\mathcal C^\infty}\ll 1$ (in particular $\Vert h\Vert_{\mathcal C^\infty}<\epsilon_\t{prec}$);
 \item $h=0$ near $x=\frac{\pi}{2}\pm \frac{\lambda}{3}$ and $y=\frac{3}{4}y_\t{std}$;
 \item $h(-x,y)=h(x,y)$;
 \item for $y\in \left[y^-,\frac{2}{3}y_\t{std}\right]$, the map $h$ does not depend on $y$ and there exists $x^h_\t{flat}$ such that $h$ is increasing for $x<-x^h_\t{flat}$, constant on  $\left[\frac{\pi}{2}-x^h_\t{flat},\frac{\pi}{2}+x^h_\t{flat}\right]$ and decreasing for $x>x^h_\t{flat}$;
 \item $\tilde{S}_h=\left\{(x,y,z_\text{prod}+h(x)),x\in\left[\frac{\pi}{2}-\frac{\lambda}{3},\frac{\pi}{2}+\frac{\lambda}{3}\right],y\in[y^-,y^+] \right\}$ is contained in $\mathcal V$.
\end{enumerate}

Let $S_h=\psi^{-1}\left(\tilde{S}_h \right)$. We denote by $v_h$ the maximum of h and $z_\text{prod}+v_h$ by $z_h$. The minimum of $S_h$ on the $v$-axis corresponds to $t=0$ and $v=v_h$ (see Figure \ref{figure_S_h}). In addition $\mathcal H\neq\emptyset$. Let $v_{\mathcal H} =\sup\{\max(h),h\in\mathcal H\}$ then $v_{\mathcal H}>0$ and for all $0<v<v_{\mathcal H}$ there exists $h\in\mathcal H$ such that $v_h=v$. Given $v_h$  there exists $h\in\mathcal H$ with $x^h_\t{flat}$ arbitrarily small.

\begin{figure}[here]
\begin{center}
 \includegraphics{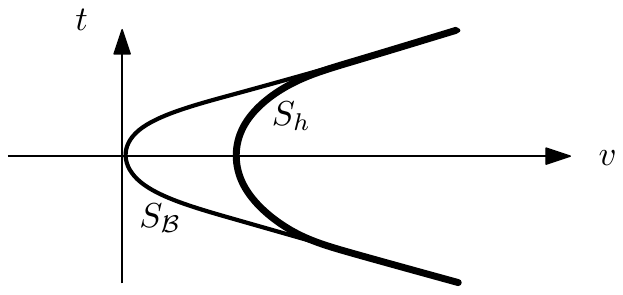}
\end{center}
 \caption{The surfaces $S_{\mathcal B}$ and $S_h$}\label{figure_S_h}
\end{figure}

\begin{lemma}\label{S_h}
The surface $S_h$ is $u$-invariant and its restriction to the plan $u=\t{cst}$ is a smooth curve composed of two symmetric graphs containing $(v_h,0)$, one positive and increasing and the other negative and decreasing on $(v_h,v_\t{max}]$;
\end{lemma}
\begin{proof}
For $h$ small enough, the tangency points between the Reeb vector field and $\tilde{S}_h$ are the segment $\left\{\frac{\pi}{2}\right\}\times[y^-,y^+]\times\{z_h\}$. In addition, $\frac{\partial}{\partial t}$ is positively transverse to $S_h$ for $t<0$ and negatively transverse for $t>0$. The proof is similar to the proof of Lemma \ref{lemme_S_B}. The symmetry derives from Fact \ref{lemme_phi_u}.
\end{proof}

We extend the surface $S_h$ by translation in the coordinates $(t,u,v)$ and still denote by $S_h$ the extension. The surface $S_h$ also extends $\sigma\left(\tilde{S}_h\right)$ as $S_h$ is parametrised by $(\pm l (v),u,v)$ and
$\sigma(\psi(\pm l (v),u,v)=\psi(\mp l(v),u_\t{max}+u_\t{min}-u,v)$.
The following lemma is a powerful tool to study the Reeb chords in $\mathcal B^{\leq y_\t{std}}$ (conditions (B6), (B7) and (B8)). We will use its corollary (Corollary \ref{p_e_p_s}) in Section \ref{subsection_conditions}.

\begin{lemma}\label{lemme_restriction_psi}
There exist positive numbers $t'_\t{max}$, $v'_\t{max}$, $u_\lambda$, $\Delta$ and $v_\Delta$ such that
\begin{itemize}
 \item $t'_\t{max}<t_\t{max}$, $v_\Delta<v'_\t{max}<v_\t{max}$ and  $v_\Delta<v_{\mathcal H}$;
 \item\label{cond_lambda} $\psi_x\left(I'_t\times[u_\t{min},u_\lambda]\times I'_v\right)\subset\left[\frac{\pi}{2}+\frac{\lambda}{2},\pi+\lambda\right]$ where $I'_t=[-t'_\t{max},t'_\t{max}]$ and $I'_v=[-v'_\t{max},v'_\t{max}]$;
 \item\label{cond_pe_ps} $\psi_y(p)<\psi_y(p')$ and $\psi_x(p')-\psi_x(p)\leq\frac{\lambda}{12}$ for all $p=(t,u,v)$ and $p'=(t',u',v')$ in $I'_t\times[u_\lambda,u_\t{max}]\times I'_v$ such that $t'-t>\Delta$ and $u'\geq u$;
 \item\label{cond_v_Delta} the planes $t=\pm\Delta$ intersect $S_{\mathcal B}$ for $v\leq v_{\Delta}$.
\end{itemize}
\end{lemma}
\begin{proof}
Without loss of generality, there exists $\eta$ such that $\frac{1}{\eta}<k_\t{inf}<\eta<2$ and $\eta^2-\sin\left(\frac{\pi}{2}+\lambda\right)<\frac{\eta}{32}$. We start with $t'_\t{max}=t_\t{max}$ and $v'_\t{max}=v_\t{max}$ and progressively reduce them. There exists $u_\lambda$ such that for $t'_\t{max}$ and $v'_\t{max}$ small enough
\begin{gather*}
\psi_x\left(I'_t\times[u_\lambda,u_\t{max}]\times I'_v\right)\subset\left[\frac{\pi}{2}-\lambda,\frac{\pi}{2}+\lambda\right],\\
\psi_x\left(I'_t\times[u_\t{min},u_\lambda]\times I'_v\right)\subset\left[\frac{\pi}{2}+\frac{\lambda}{2},\pi+\lambda\right].
\end{gather*}
Let $M=\left\Vert\frac{\partial \psi_y}{\partial v}\right\Vert_\infty+\left\Vert\frac{\partial \psi_x}{\partial v}\right\Vert_\infty+\left\Vert\frac{\partial k_\t{inf}}{\partial z}\right\Vert_\infty$. Choose $\Delta$ such that, upon reducing $t'_\t{max}$ and $v'_\t{max}$
\begin{itemize}
 \item $t'_\t{max}=4\Delta<\frac{\lambda}{96 M}$;
 \item $v'_\t{max}<\min\left(\frac{\Delta}{8M},\frac{\lambda}{48 M},v_{\mathcal H}\right)$; 
 \item the planes $t=\pm\Delta$ intersect $S_\mathcal{B}$ for $v\leq v'_\t{max}$. 
\end{itemize}
For all $(t,u,v)\in [0,t'_\t{max}]\times[u_\lambda,u_\t{max}]\times I'_v$ we have\footnote{See equation(\ref{eq_R_inf}) for the explicit form of the Reeb vector field.}
\begin{align*}
\psi_y(0,u,v)+\frac{t}{\eta}\sin\left(\frac{\pi}{2}+\lambda\right)&\leq\psi_y(t,u,v)\leq\psi_y(0,u,v)+\eta t,\\
\psi_y(0,u,0)-v'_\t{max}\left\Vert\frac{\partial \psi_y}{\partial v}\right\Vert_\infty &\leq\psi_y(0,u,v)\leq\psi_y(0,u,0)+ v'_\t{max}\left\Vert\frac{\partial \psi_y}{\partial v}\right\Vert_\infty.
\end{align*}
Thus $\psi_y(p')-\psi_y(p)\geq\frac{\Delta}{4}>0$. Similarly, it holds that
\begin{gather*}
\psi_x(0,u,0)-v'_\t{max}\left\Vert\frac{\partial \psi_x}{\partial v}\right\Vert_\infty-2t'_\t{max}\left\Vert\frac{\partial k_\t{inf}}{\partial z}\right\Vert_\infty\leq\psi_x(t,u,v),\\
\psi_x(t,u,v)\leq\psi_x(0,u,0)+v'_\t{max}\left\Vert\frac{\partial \psi_x}{\partial v}\right\Vert_\infty+2t'_\t{max}\left\Vert\frac{\partial k_\t{inf}}{\partial z}\right\Vert_\infty,
\end{gather*}
and $\psi_x(p')-\psi_x(p)\leq2v'_\t{max}\left\Vert\frac{\partial \psi_x}{\partial v}\right\Vert_\infty+4t'_\t{max}\left\Vert\frac{\partial k_\t{inf}}{\partial z}\right\Vert_\infty$.
\end{proof}

\begin{construction}\label{construction_h}
We apply Lemma \ref{lemme_restriction_psi} and choose a map $h\in\mathcal H$ such that
$v_\Delta<\max(h)<v_\t{max}$ and such that the planes $v=v'_\t{max}$ intersect $S_{h}$ outside the flat part of $\tilde{S}_{h}$.
\end{construction}

\begin{fact}\label{epsilon_H}
There exists $\epsilon^h_\t{pert}$ such that for any $\epsilon^h_\t{pert}$-perturbation of $\alpha_\t{prec}$ in $\mathcal C^1$-norm, the Reeb vector field is transverse to $\tilde{S}_h$ for all $x$ satisfying $x^h_\t{flat}\leq\vert x-\frac{\pi}{2}\vert\leq\frac{\lambda}{3}$.
\end{fact}

\subsubsection{Convexification model}\label{subsubsection_model}
We now construct a convexification model for $y\in\left[y_-,\frac{2}{5}y_\t{std}\right]$ and interpolate with the adapted model for $y\in\left[\frac{2}{5}y_\t{std},\frac{2}{3}y_\t{std}\right]$. The new boundary is smoothed for $y\in\left[\frac{2}{3}y_\t{std},\frac{3}{4}y_\t{std}\right]$. Let $z>0$ and 
\[\left\{\begin{array}{llll}
 \Phi_{z_0}(z)&=&\exp(-\frac{1}{z-z_0}), &\text{ if } z>z_0,\\
\Phi_{z_0}(z)&=&0, &\text{ otherwise.}
\end{array}\right.\]
For $y\in[y^-_\rho,y^+_\rho]$ and near $\Gamma_0$, \[\alpha_\t{prec}=k_\t{inf}(z)\sin(x)\d y+\cos(x)\d z\] and $k_\t{inf}(z)=1-\rho z$. Let $k_\t{conv}(z)=k_\t{inf}(z)+a\Phi_{z_0}(z)$ and
\begin{equation}\label{equation_alpha_conv}
\alpha_\t{conv}=k_\t{conv}(z)\sin(x)\d y+\cos(x)\d z
\end{equation}
where $a>0$. 
\begin{fact}\label{fact_adapted}
The contact form $\alpha_\t{conv}$ is adapted to $\tilde{S}_h$ for $y\in\left[y_-,\frac{2}{3}y_\t{std}\right]$.
\end{fact}

In the coordinates $(t,u,v)$, the associated contact form is not a convexification model in the sense of \cite{CGHH10}. We use a weakened version of convexification. Let $J_u$ be such that $[u_1,u_5]\subset J_u\subset I_u$. Two functions $f$ and $g$ from $I_t\times J_u\times I_v$ to $\mathbb R_+^*$ form a \emph{convexification pair} if  
\begin{enumerate}
 \item $f=1$ and $g(t,u,v)=1-\rho z_\t{prod}-\rho v$ near $S'$ and for $v\geq v'_\t{max}$;
 \item $f$ and $g$ do not depend on $u$ for $u\in[u_1,u_5]$;
 \item $\frac{\partial f}{\partial v}\geq 0$ and $\frac{\partial f}{\partial v}> 0$ near $(0,v_h)$;
 \item in a neighbourhood of $[u_1,u_5]$, in the planes $u=\t{cst}$, the vector field $X_g=\left(-\frac{\partial g}{\partial v},\frac{\partial g}{\partial t}\right)$ is negatively transverse to $S_h$ for $t>0$, positively transverse for $t<0$ and points toward the half-space $t<0$ for $t=0$ (see Figure \ref{figure_Reeb_S_h}). 
\end{enumerate}
\begin{figure}[here]
\begin{center}
 \includegraphics{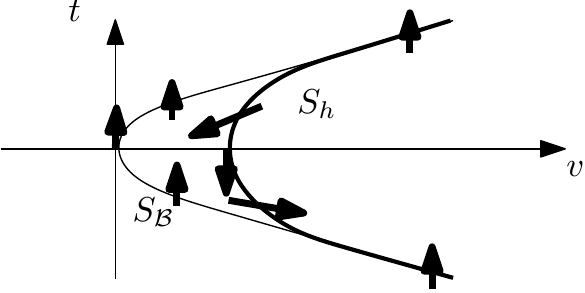}
\end{center}
 \caption{The surface $S_h$ and the Reeb vector field}\label{figure_Reeb_S_h}
\end{figure}

\begin{proposition}\label{existence_f_g} For $\rho$ and $C_\psi$ small enough, there exists $\epsilon_{z_0}>0$ such that for any small $\rho_0$ and any $z_h-\epsilon_{z_0}<z_0<z_h$, there exists a contact form $\alpha$ and a pair of convexification $(f,g)$ with $J_u=[u_\t{max}+u_\t{min}-u_+,u_+]$ satisfying
\begin{enumerate} 
 \item\label{cond_transv} $R_u\geq 0$, $R$  is positively transverse to $S_h$ for $t<0$, negatively transverse for $t>0$ and points toward the half-space $t<0$ for $t=0$; 
 \item\label{cond_sym_sigma} $\sigma^*\alpha=-\alpha$;
 \item\label{cond_Delta} $\alpha=\alpha_\t{prec}$ for $v\geq v'_\t{max}$ and $v\leq v_\Delta$;
 \item\label{cond_alpha} $\alpha=f(t,u,v)\d t+g(t,u,v)\d u $ on $I_t\times J_u\times I_v$;
 \item ${\psi^{-1}}^*\alpha=\alpha_\t{conv}$ and $k'_\t{conv}(z_h)=\rho_0$ in a neighbourhood of $\{\frac{\pi}{2}\}\times\left[y^+,\frac{2}{5}y_\t{std}\right]\times\{z_\text{prod}\}$.
\end{enumerate}
\end{proposition}

By definition $\Vert k_\t{conv} \Vert_{\mathcal C^1}\leq \rho_0$. The end of Section \ref{subsubsection_model} is devoted to the proof of Proposition \ref{existence_f_g}. Before extending $\alpha_\t{conv}$ for all $u$, we study some properties of $\alpha_\t{conv}$ in the $(t,u,v)$-coordinates.

\begin{fact}\label{lemme_d_psi}
For all $(t,u,v)\in I_t\times [u^-,u^+]\times I_v$, we have
\begin{align*}
\d\psi(t,u,v)&=\left(\begin{array}{ccc}
R_x(\psi(t,u,v))&0& \frac{\partial \psi_x}{\partial v}(t,u,v)\\
R_y(\psi(t,u,v))&1&\frac{\partial \psi_y}{\partial v}(t,u,v)\\
R_z(\psi(t,u,v))&0&\frac{\partial \psi_z}{\partial v}(t,u,v)
\end{array}\right),\\
\d\psi(0,u,v)&=\left(\begin{array}{ccc}
R_x(\frac{\pi}{2},u,v)&0& 0\\
R_y(\frac{\pi}{2},u,v)&1&0\\
R_z(\frac{\pi}{2},u,v)&0&1
\end{array}\right).
\end{align*}
\end{fact}

\begin{fact}\label{lemme_psi_alpha}
For all $(t,u,v)\in I_t\times [u^-,u^+]\times I_v$, we have
\[\psi^*\alpha_\t{prec}(t,u,v)=\d t+k_0\d u \t{ and } 
\left\{\begin{array}{l}
k_\t{inf}(\psi_z)\sin(\psi_x)=k_0 \\
k_\t{inf}(\psi_z)\sin(\psi_x)\frac{\partial \psi_y}{\partial v}+\cos(\psi_x)\frac{\partial \psi_z}{\partial v}=0
\end{array}\right.\]
where $k_0(v)=k_\t{inf}(v+z_\t{prod})$.
\end{fact}

\begin{fact}\label{lemme_psi_alpha_2}
For all $(t,u,v)\in I_t\times [u^-,u^+]\times I_v$, we have
\begin{gather*}
  \begin{split}
  \psi^*\alpha_\t{conv}=\Bigg(\frac{k_\t{conv}(\psi_z)}{k_\t{inf}(\psi_z)}\sin^2(\psi_x)+\cos^2(\psi_x) \Bigg)\d t+k_\t{conv}(\psi_z)\sin(\psi_x)\d u+\\ \Bigg(k_\t{conv}(\psi_z)\sin(\psi_x)\frac{\partial \psi_y}{\partial v}+\cos(\psi_x)\frac{\partial \psi_z}{\partial v} \Bigg)\d\tau,
  \end{split}\\
   \begin{split}
  \psi^*\alpha_\t{conv}=\Bigg(\bigg(\frac{k_\t{conv}(\psi_z)}{k_\t{inf}(\psi_z)}-1\bigg)\sin^2(\psi_x)+1 \Bigg)\d t+\frac{k_\t{conv}(\psi_z)}{k_\t{inf}(\psi_z)}k_0\d u+\\ \Bigg(\big(k_\t{conv}(\psi_z)-k_\t{inf}(\psi_z)\big)\sin(\psi_x)\frac{\partial \psi_y}{\partial v}\Bigg)\d\tau.
 \end{split}
\end{gather*}
\end{fact}

In coordinates $(t,u,v)$,
$\alpha_\t{conv}=f_1(t,v)\d t+g_1(t,v)\d u+h_1(t,v)\d v$
where $f_1$ $g_1$ and $h_1$ do not depend on $u$. Note that $f_1(-t,v)=f_1(t,v)$, $g_1(-t,v)=g_1(t,v)$ and $h_1(-t,v)=-h_1(t,v)$ (Fact \ref{lemme_phi_u}).
Fix $u^-<u'_0<u'_1<u^+$ and $p : [u'_0,u'_1]\to\mathbb R$ such that $p=0$ near $u'_0$ and $p=1$ near $u'_1$. Let \[\alpha=f_1(t,v)\d t+g_1(t,v)\d u+p(u)h_1(t,v)\d v.\]

\begin{lemma}\label{lemme_alpha_contact}
For $\rho$ and $C_\psi$ small and for $z_0$ close to $z_h$, $\alpha$ is a contact form, $\frac{\partial f_1}{\partial v}\geq 0$ and $\frac{\partial f_1}{\partial v}> 0$ near $(0,v_h)$;
\end{lemma}

\begin{proof}
The differential of $\alpha$ is
\[\d \alpha=\frac{\partial f_1}{\partial v}\d v\wedge\d t+\frac{\partial g_1}{\partial v}\d v\wedge\d u+\frac{\partial g_1}{\partial t}\d t\wedge\d u+p'h_1\d u\wedge\d v+p \frac{\partial h_1}{\partial t}\d t\wedge\d v.\]
The contact condition is
\[\frac{\partial f_1}{\partial v}g_1-\frac{\partial g_1}{\partial v}f_1+ph_1\frac{\partial g_1}{\partial t}\d t+p'h_1f_1-p \frac{\partial h_1}{\partial t}g_1>0.\]
Without loss of generality we may assume that the ranges of $f_1$, $g_1$ and $h_1$ are in $\left[\frac{1}{2},2\right]$. In what follows, the bounds associated to the notation $O$ are uniform for all convexification models. By definition, it holds that
\begin{align*}
k_\t{conv}(z)-k_\t{inf}(z)&=A\Phi_{z_0}(z)=(z-z_0)^2A\Phi'_{z_0}(z),\\
\frac{k_\t{conv}}{k_\t{inf}}(z)-1&=(z-z_0)^2A\Phi'_{z_0}(z)O(1),\\
\left(\frac{k_\t{conv}}{k_\t{inf}}\right)'(z)&=A\Phi'_{z_0}(z)\left(\frac{1}{k_\t{inf}(z)}+\frac{\rho(z-z_0)^2}{\big( k_\t{inf}(z)\big)^2}\right).
\end{align*}
Thus, we obtain
\begin{align*}
f_1(t,v)&=1+A\Phi'_{z_0}(\psi_z)(\psi_z-z_0)^2 O(1),\\
\frac{\partial f_1}{\partial v}&=A\Phi'_{z_0}(\psi_z)\left(1+O(\rho)+O(C_\psi)+(\psi_z-z_0)^2O(1)\right).
\end{align*}
Therefore for $\rho$, $C_\psi$, $z_h-z_0$ small enough, we have $\frac{\partial f_1}{\partial v}\geq 0$ and $\frac{\partial f_1}{\partial v}(0,v_h)>0$.
Similarly, we obtain
\begin{align*}
g_1(t,v)&=1+O(\rho)+(z-z_0)^2 O(1)+O(C_\psi),\\
\frac{\partial g_1}{\partial v}&=-\rho+A\Phi'_{z_0}(\psi_z)\left(1+(\psi_z-z_0)^2 O(1)+O(\rho)+O(C_\psi)\right),\\
\frac{\partial g_1}{\partial t}(t,v)&=A\Phi'_{z_0}(\psi_z)O(C_\psi),\\
h_1(t,v)&=A\Phi'_{z_0}(\psi_z)(\psi_z-z_0)^2 O(1),\\
\frac{\partial h_1}{\partial t}(t,v)&=A\Phi'_{z_0}(\psi_z)\left((\psi_z-z_0)^2 O(1)+O(C_\psi)\right),
\end{align*}
and the contact condition is
\begin{equation}
\rho+A\Phi'_{z_0}(\psi_z)\left(O(\rho)+(\psi_z-z_0)^2 O(1)+O(C_\psi)\right)>0.
\end{equation}
Yet $A\Phi'_{z_0}(z)<2\rho$ as $k'_\t{conv}(z_h)=\rho_0<\rho$. Thus the contact condition is satisfied for $\rho$, $C_\psi$ and $z_h-z_0$ small enough.
\end{proof}

\begin{lemma}\label{lemme_alpha_transverse}
For $\rho$, $\rho_0$, $C_\psi$, $z_h-z_0$ small, $R_u\geq 0$ and $R$ is positively transverse to $S_h$ for $t<0$, negatively for $t>0$ and points toward the half-plane $t<0$ for $t=0$.
\end{lemma}

\begin{proof}
The component $R_u$ is positively collinear to 
\[\frac{\partial f_1}{\partial v}-p\frac{\partial h_1}{\partial t}=A\Phi'_{z_0}(\psi_z)\left(1+O(\rho)+O(C_\psi)+(\psi_z-z_0)^2O(1)\right).\]
Thus $R_u\geq 0$.
By $u$-invariance, we study the transversality properties in the planes $u=\t{cst}$. The Reeb vector field is positively collinear to
\[Y=\left(\begin{array}{c}
-\frac{\partial g_1}{\partial v}+p'(u)h_1\\
\frac{\partial g_1}{\partial t}
\end{array}\right) \] 
in the coordinates $(t,v)$.
The tangency condition for $t=0$ is automatically satisfied as $h_1(0,v)=0$. On the non flat part of $h$, the transversality conditions are satisfied for $\rho_0<\epsilon_h^\t{pert}$ as $\Vert k_\t{conv}-k_\t{inf}\Vert_{\mathcal C^1}\leq\rho_0$ (Fact \ref{epsilon_H}). We now prove the result in the flat part of $h$. The transversality condition is
\begin{equation}
-\frac{\partial g_1}{\partial v}+p'(u)h_1-l'_\pm\frac{\partial g_1}{\partial t}>0
\end{equation}
for $t\neq 0$ where $l_+$ and $l_-$ parametrise $S_h$.
Let $\rho_0=s_0\rho$. Then $A\Phi_{z_0}'(\psi_z)=(1+s_0)\rho$. 
For $p=1$, $Y$ satisfies the desired transversality conditions (Fact \ref{fact_adapted}). For $p'=0$ the transversality condition is
\begin{equation}\label{condition_transversalite}
-s_0\rho+\rho(1+s_0)a(t,u,v)>0
\end{equation}
where $a$ does not depend on $s_0$. There exists $s_\t{max}$ such that  (\ref{condition_transversalite}) is satisfied for $s_0\in(0,s_\t{max}]$. Thus $a> \frac{s_\t{max}}{1+s_\t{max}}$. For $z_0$ close to $z_h$, 
\begin{equation*}
 \left\vert p'(u)(\psi_z-z_0)^2\sin(\psi_x)\frac{\partial \psi_y}{\partial v}\right\vert<\frac{s_\t{max}}{2(1+s_\t{max})}.
\end{equation*}
The general transversality condition is
\begin{equation}\label{condition_transversalite'}
a(t,u,v)+p'(u)(\psi_z-z_0)^2\sin(\psi_x)\frac{\partial \psi_y}{\partial v}+s(a(t,u,v)-1)>0.
\end{equation}
For $s\leq\frac{s_\t{max}}{2}$, we obtain $a(t,u,v)+s(a(t,u,v)-1)>\frac{s_\t{max}}{2(1+s_\t{max})}$ and (\ref{condition_transversalite'}) is satisfied.
\end{proof}

\begin{proof}[Proof of Proposition \ref{existence_f_g}]
We choose $\rho$, $\rho_0$, $C_\psi$, $z_h-z_0$ small enough to apply Lemma \ref{lemme_alpha_contact} and Lemma \ref{lemme_alpha_transverse}. We extend $\alpha$ to $\mathcal U$ by $\alpha=f_1\d t+g_1\d u$ for $u\in J_u$ and $-\sigma^*\alpha$ for $u\in [u_\t{min},u_\t{min}+u_\t{max}-u^+]$.
It remains to prove that $\alpha=\alpha_\t{prec}$ for $v\geq v'_\t{max}$ and $v\leq v_\Delta$. The set where \[(f_1(t,v),g_1(t,v))\neq(1,1-\rho z_\t{prod}-\rho v)\] is contained between the surface
$S_{z_0}$ associated to the equation $z=z_0$ and $S_h$. The surface $S_{z_0}$ has properties similar to $S_h$. In particular, its $v$-coordinates are greater than $z_0-z_\t{prod}$. As $z_h-z_\t{prod}>v_\Delta$ (Construction \ref{construction_h}) $z_0-z_\t{prod}>v_\Delta$ for $z_h-z_0$ small enough.
Additionally, $S_{z_0}$ intersects $S_h$ in its non-flat part. Yet for $z_h-z_0$ small enough the intersection points are arbitrarily close to the endpoints of the flat part and the $v$-coordinates of the intersection points are smaller than $v'_\t{max}$ (Construction \ref{construction_h}).
\end{proof}

\subsubsection{Perturbed convexification}\label{subsubsection_perturbation}
The contact form $\alpha$ described in Proposition \ref{existence_f_g} is adapted to the boundary but does not give us the desired control on the Reeb flow. Let $\Sigma_\pm=\{(\pm t_\t{max}, u, v), u\in I_u, v \in I_v\}$ and $\Phi$ be the map induced by the Reeb flow of $\alpha_\t{conv}$ between $\Sigma_-$ and $\Sigma_+$.
\begin{proposition}\label{lemme_perturbation_2}
Let $(f,g)$ be a convexification pair given in Proposition \ref{existence_f_g}. Upon perturbing $\alpha$ near $I_t\times[u_1,u_5]\times I_v$ and we may also assume that
\begin{enumerate}
 \item\label{cond_Phi_1}\label{cond_Phi_2}  $\Phi_*\left(\mathcal C_p (V,B)\right)\subset \mathcal C_{\Phi(p)} (H,\epsilon_B)$ and $\Phi^{-1}_*\left(\mathcal C_{\Phi(p)} (V,B)\right)\subset \mathcal C_{p} (H,\epsilon_B)$  for all $p\in[u_2,u_3]\times I_v\cap\Phi^{-1}([u_4,u_5]\times I_v)$;
 \item $\Vert \d\Phi(p,v)\Vert>\frac{L^2}{\sqrt\eta}\Vert v\Vert$ and $\Vert \d\Phi^{-1}(\Phi(p),w)\Vert>\frac{L^2}{\sqrt\eta}\Vert w\Vert$ for all 
 $p\in[u_2,u_3]\times I_v\cap\Phi^{-1}([u_4,u_5]\times I_v)$ and $v\in\mathcal C_p(V,B)$;
 \item the return time in $[u_2,u_3]\times I_v$ is bounded by $\tau$.
\end{enumerate}
\end{proposition}

\begin{construction}\label{construction_f_g}
Apply Proposition \ref{existence_f_g} with $\rho_0\leq\epsilon^h_\t{pert}$ and Proposition \ref{lemme_perturbation_2}. 
\end{construction}

The end of Section \ref{subsubsection_perturbation} is devoted to the proof of Proposition \ref{lemme_perturbation_2}.  The contact form from Proposition \ref{existence_f_g} is $\alpha=f(t,u,v)\d t+g(t,u,v)\d u$. Thus
\begin{equation}
R_\alpha=\frac{1}{\frac{\partial f}{\partial v}g-\frac{\partial g}{\partial v}f}\left(\begin{array}{c}
-\frac{\partial g}{\partial v}\\
\frac{\partial f}{\partial v}\\
\frac{\partial g}{\partial t}
\end{array}\right).
\end{equation}
We progressively modify $f$ and $g$ so that the difference between the $u$-coordinates of two Reeb orbits which contribute to $\Phi$ widens when the Reeb orbits cross the convexification area. 

\begin{remark}\label{remarque_f_g}
If $(f_1,g_1)$ and $(f_1,g_2)$ are two convexification pairs satisfying conditions (\ref{cond_transv}), (\ref{cond_Delta}) and (\ref{cond_alpha}) of Proposition~\ref{existence_f_g} on $I_t\times J_u\times I_v$ then there exists a convexification pair $(f_1,g)$ satisfying the same conditions such that $g=g_1$ outside a neighbourhood of $I_t\times[u_1,u_5]\times I_v$ and $g=g_2$ in a neighbourhood of $I_t\times[u_1,u_5]\times I_v$. There exists a analogous statement to interpolate between $f$-coordinates in a convexification pair if $f_1=f_2$ near $S_h$.
\end{remark}

If $(f,g)$ is a convexification pair given in Proposition \ref{existence_f_g}, without loss of generality, we have $1\leq f\leq\frac{3}{2}$. Let $\Gamma_v$ denote the $g$-level intersecting  $(0,v)$ in a plane $u=\t{cst}$ for $u\in[u_1,u_5]$ and $\Gamma_{[v,v']}$ be the set between $\Gamma_{v}$ and $\Gamma_{v'}$ if $v<v'$.  
Let $m_f>1024\frac{u_5-u_1}{u_4-u_3}$ and $M_g=\frac{m_f t_1}{8(u_5-u_1)}$ for $t_1>0$. Without loss of generality $\rho<\frac{1}{4}$.

 \begin{figure}[here]
\begin{center}
 \includegraphics{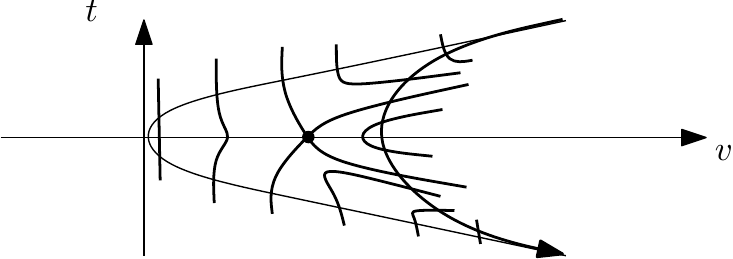}
\end{center}
 \caption{The $g$-levels}\label{niveaux_g}
\end{figure}

\begin{figure}[here]
\includegraphics{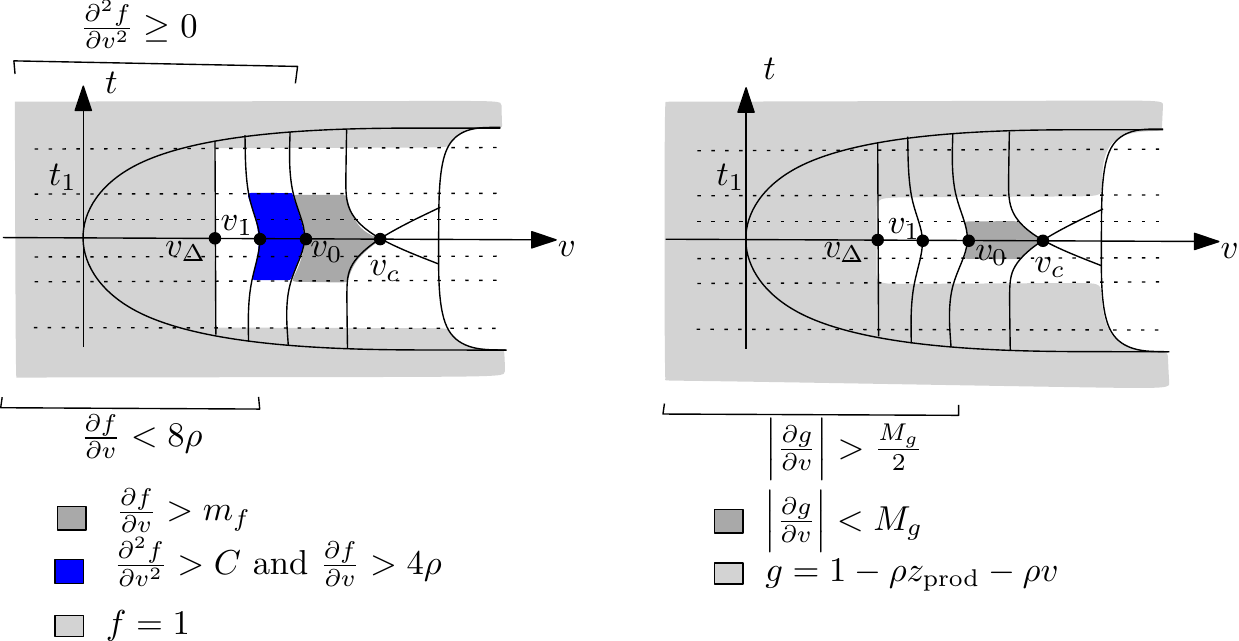}
\caption{Conditions on $f$ and $g$}\label{conditions_f_g}
\end{figure}

\begin{lemma}
Let $C>0$.
We may assume that there exist $t_1>0$, $v_c$, $v_0$ and $v_1$ such that $v_\Delta<v_1<v_0<v_c<v_h$ and for all $u\in[u_1,u_5]$ (see Figure \ref{conditions_f_g})
\begin{itemize}
  \item $\frac{1}{2}\leq g\leq 1$ and $\left\vert\frac{\partial g}{\partial v}\right\vert\leq\rho $; 
  \item $g(-t,v)=g(t,v)$ and $f(-t,v)=f(t,v)$;
  \item \label{cond_point_selle} $(0,v_c)$ is a saddle for $g$ and the $g$-level intersecting $\{0\}\times[v_c,v_h]$ do not intersect $\Sigma_+$ or $\Sigma_-$ (see Figure \ref{niveaux_g});
  \item $g(t,u,v)=1-\rho z_\t{prod}-\rho v$ for all $(t,u,v)$ such that $\vert t\vert\geq t_1$ or $(t,v)$ on $\Gamma_{[0,v_\Delta]}$;
  \item $f(t,u,v)=1$ for all $(t,u,v)$ such that $\vert t\vert \geq 2t_1$ or $(t,v)$ on $\Gamma_{[0,v_\Delta]}$;
  \item $\frac{\partial^2 g}{\partial v^2}\geq 0$ and $\frac{\partial g}{\partial v}< 0$ on $\Gamma_{[0,v_c]}$;
  \item  $\left\vert\frac{\partial g}{\partial v}(t,v)\right\vert>\frac{M_g}{2}$ and $\frac{\partial^2 f}{\partial v^2}(t,u,v)\geq 0$ on $\Gamma_{[0,v_0]}$;
  \item $\frac{\partial f}{\partial v}(t,u,v)< 8\rho$ on $\Gamma_{[0,v_1]}$;
  \item $\frac{\partial f}{\partial v}(t,u,v)>m_f$ for all $(t,u,v)$ such that $\vert t\vert\leq t_1$ and $(t,v)\in\Gamma_{[v_0,v_c]}$; 
  \item $\left\vert\frac{\partial g}{\partial v}(t,v)\right\vert<M_g$ for all $(t,u,v)$ such that $\vert t\vert\leq \frac{t_1}{2}$ and $(t,v)\in\Gamma_{[v_0,v_c]}$; 
  \item $\frac{\partial^2 f}{\partial v^2}(t,u,v)>C$ and $\frac{\partial f}{\partial v}(t,u,v)>4\rho$ for all $(t,u,v)$ such that $\vert t\vert\leq t_1$ and $(t,v)\in\Gamma_{[v_1,v_0]}$;  
\end{itemize} 
\end{lemma}

\begin{proof}
We extensively use Remark \ref{remarque_f_g} to modify $f$ and $g$.
We first modify $f$ and choose $v_c$ so that $\frac{\partial f}{\partial v}>2m_f$ in a neighbourhood of $\{0\}\times[u_1,u_5]\times \{v_c\}$. To achieve this condition, we modify $f$ near $(0,v_h)$ in a neighbourhood that does not intersect~$S_h$ and such that $\frac{\partial f}{\partial v}>0$. The only non-trivial condition on the perturbed $f$ is the contact condition $\frac{\partial f}{\partial v}g-\frac{\partial g}{\partial v}f>0$. There exists $\epsilon>0$ such that for any $f_1$ with $\vert f-f_1\vert<\epsilon$ and $\frac{\partial f_1}{\partial v}>\frac{\partial f}{\partial v}-\epsilon$, then $\frac{\partial f_1}{\partial v}g-\frac{\partial g}{\partial v}f_1>0$. We choose $f_1$ such that $\frac{\partial f_1}{\partial v}>2m_f$ near $(0,v_c)$, $\vert f-f_1\vert<\epsilon$ and $\frac{\partial f_1}{\partial v}>\frac{\partial f}{\partial v}-\epsilon$.

We now choose $t_1$ and $v_0$ such that $\rho>\frac{M_g}{2}$ and $\left\vert\frac{\partial f}{\partial v}(t,v)\right\vert>2m_f$ for all $(t,u,v)$ such that $\vert t\vert \leq t_1$ and $(t,v)\in\Gamma_{[v_0,v_c]}$. We modify $g$ (and change $v_0$ if necessary) so that $g$ satisfies the desired conditions. To obtain the contact condition we choose $g$ so that  $\frac{\partial g}{\partial v}>0 $ implies $\frac{\partial f}{\partial v}>0$ and $\max\left(\frac{\partial g}{\partial v}\right)\ll 1$. Finally we modify~$f$.
\end{proof}

\begin{lemma}\label{lemme_projection}
The projection in the $(t,v)$-plane of any Reeb orbit which contributes to $\Phi$ is contained in $\Gamma_{[v_1,v_0]}$.
\end{lemma}

\begin{proof}
The projection in the $(t,v)$-plane of a Reeb orbit is contained in a $g$-level. Thus no Reeb orbit contributes to $\Phi$ and intersects $\{0\}\times[u_1,u_5]\times[v_c,v_h]$.

If a Reeb orbit intersects $\{0\}\times[u_1,u_5]\times[v_0,v_c]$ then this orbit crosses the strip $\vert t\vert\leq t_1$ and is contained in $\Gamma_{[v_0,v_c]}$. In this strip 
\[\frac{m_f}{2}\leq \frac{\partial f}{\partial v}g\leq\frac{\partial f}{\partial v}g-\frac{\partial g}{\partial v}f\leq \frac{5}{2}\left\vert\frac{\partial f}{\partial v} \right\vert,\]
$\vert R_t\vert\leq\frac{2M_g}{m_f}$ and $\vert R_u\vert\geq \frac{2}{5}$. Therefore, the time spent in the strip is bounded below by \[\frac{2t_1}{\max(R_t)}\geq\frac{t_1m_f}{M_g}=8(u_5-u_1)\] and the $u$-interval swept out by the orbit is bounded below by
$\min\vert R_u\vert\times 8(u_5-u_1)>u_5-u_1$. The orbit does not contribute to $\Phi$.

We now consider a Reeb orbit which intersects $\{0\}\times[u_1,u_5]\times[-v_\t{max},v_1]$ and crosses the strip $\vert t\vert\leq 2t_1$. In this strip $\frac{1}{2}\left\vert\frac{\partial f}{\partial v} \right\vert\leq\frac{\partial f}{\partial v}g-\frac{\partial g}{\partial v}f\leq 10\rho\leq 4$, $\vert R_t\vert\geq\frac{M_g}{8}$ and $\vert R_u\vert\leq 2$. The return time between $-2t_1$ and $2t_1$ is bounded by $\frac{4t_1}{\min(R_t)}\leq\frac{32t_1}{M_g}$ and the $u$-interval swept out by the orbit is bounded by $\frac{64t_1}{M_g}<u_4-u_3$. The orbit does not contribute to $\Phi$ as $R_u=0$ for $\vert t\vert\geq 2 t_1$.
\end{proof}

\begin{proof}[Proof of Proposition \ref{lemme_perturbation_2}]
We prove that Proposition \ref{lemme_perturbation_2} is satisfied for $C$ big enough. 
We first study the difference between the $u$-coordinates of two Reeb orbits which contribute to $\Phi$.
Let $(-t_\t{max},u,v)$ and $(-t_\t{max},\tilde{u},\tilde{v})$ be the endpoints of two Reeb chord which contribute to $\Phi$. Without loss of generality $\tilde{v}>v$. Their projections on the $(t,v)$-plane are contained in  $\Gamma_{[v_1,v_0]}$ (Lemma \ref{lemme_projection}). Let 
\[Y=\left(
1,-\dfrac{\frac{\partial f}{\partial v}}{\frac{\partial g}{\partial v}},
-\dfrac{\frac{\partial g}{\partial t}}{\frac{\partial g}{\partial v}}
\right)\]
be a renormalisation of the Reeb vector field and $t\mapsto\big(-t_\t{max}+t,u(t),v(t)\big)$ and $t\mapsto\big(-t_\t{max}+t,\tilde{u}(t),\tilde{v}(t)\big)$ be the $Y$-orbits with endpoints $(-t_\t{max},u,v)$ and $(-t_\t{max},\tilde{u},\tilde{v})$. 
Then
\[\frac{\partial Y_u}{\partial v}=-\dfrac{\frac{\partial^2 f}{\partial v^2}\frac{\partial g}{\partial v}-\frac{\partial^2 g}{\partial v^2}\frac{\partial f}{\partial v}}{\left(\frac{\partial g}{\partial v}\right)^2}\geq-\frac{\frac{\partial^2 f}{\partial v^2}}{\frac{\partial g}{\partial v}}\geq 0\]
and $\tilde{u}-u$ is non-decreasing. 
In addition, $\frac{\partial Y_u}{\partial v}\geq\frac{C}{\rho} $ for $\vert t\vert \leq \frac{t_1}{2}$.
Thus, we have
\[\tilde{u}(t_\t{max})-u(t_\t{max})\geq\tilde{u}-u+\frac{C}{\rho}\min_{\vert t\vert \leq \frac{t_1}{2}}(\tilde{v}(t)-v(t)).\]
Our orbits are contained in $g$-levels, therefore it holds that
\[(\tilde{v}(t)-v(t))\min\left(\left\vert\frac{\partial g}{\partial v}\right\vert\right)\leq g(-t_\t{max},v)-g(-t_\t{max},\tilde{v})\leq(\tilde{v}(t)-v(t))\max\left(\left\vert\frac{\partial g}{\partial v}\right\vert\right)\]
and $\tilde{v}(t)-v(t)\geq\frac{M_g(\tilde{v}-v)}{2\rho}$ as $\frac{M_g}{2}\leq \left\vert\frac{\partial g}{\partial v}\right\vert\leq \rho$.
Thus we obtain
\begin{equation}\label{Delta_u}
\tilde{u}(t_\t{max})-u(t_\t{max})\geq\tilde{u}-u+\frac{CM_gt_1(\tilde{v}-v)}{2\rho^2}. 
\end{equation}

Let $\gamma$ be a curve in $\{-t_\t{max}\}\times[u_2,u_3]\times I_v $ such that $\vert\gamma'(v)\vert\leq B$ (the $v$-coordinate is the vertical coordinate). 
Let $\delta$ be its image on $\{t_\t{max}\}\times[u_4,u_5]\times I_v $. By symmetry, if it is well-defined, the image of $(-t_\t{max},\gamma(v),v)$ is $(t_\t{max},\delta(v),v)$. Using equation (\ref{Delta_u}), we obtain 
\begin{equation}\label{equation_delta'}
\delta'(v)\geq\gamma'(v)+\frac{CM_g}{2\rho^2}\geq -B+\frac{CM_g}{2\rho^2}=D.
\end{equation}
Therefore, we have $D\geq\frac{1}{\epsilon_B}$ for $C$ big enough. A similar proof shows the symmetric result. Additionally, if $w=(1,\gamma'(v))$ then $\Vert w\Vert\leq \sqrt{1+B^2}$ and $\Vert\d\Phi(p, w)\Vert\geq \sqrt{1+D^2}$. Thus the dilatation condition is satisfied for $C$ big .

Finally the return time between $\{-t_\t{max}\}\times[u_2,u_3]\times I_v $ and $\{t_\t{max}\}\times[u_4,u_5]\times I_v $ is bounded by $2(u_5-u_1)+2t_\t{max}$. Indeed, for $\vert t\vert \geq t_1$, we have $\vert R_t\vert\leq 1$ and the return time is bounded by $2t_\t{max}$. Additionally, for $\vert t\vert \leq t_1$,  we have
$\vert R_u\vert=\geq \frac{8}{11}$
and the return time is bounded by $\frac{11}{8}(u_5-u_1)$ as the $u$-interval is bounded by $u_5-u_1$.
As $u_5-u_1\leq\left\vert\int{\Gamma_{\lambda_D}} \alpha\right\vert<\tau$ by definition of pre-convex bypass, we obtain the desired condition on the return time.
\end{proof}

\subsubsection{Convexification smoothing}
In this section we interpolate between $\alpha_\t{conv}$ and $\alpha_\t{prec}$ for $y\geq \frac{2}{5} y_\t{std}$.

\begin{construction}\label{construction_prolongement}
For $y\in\left[\frac{2}{5} y_\t{std},y_\text{std}\right]$, let $\alpha=\alpha_\t{prec}+al(y)\Phi_{z_0}(z)\sin(x)\d y$ where $l$ is non increasing, $l=1$ in $\left[\frac{2}{5} y_\t{std},\frac{1}{2} y_\t{std}\right] $ and $l=0$ for $y\geq \frac{2}{3} y_\t{std}$. We extend this construction to the other non-convex areas by symmetry. 
\end{construction} 

The $1$-form $\alpha$ is a contact form as $al(y)\Phi_{z_0}(z)\sin(x)$ is $\mathcal C^1$-close to $0$.
Let $\mathcal B_\t{conv}$ be the convexified bypass and $\alpha_\t{conv}$ the associated contact structure. We call $\mathcal C=\mathcal B_\t{conv}\setminus\mathring{\mathcal B}$ the \emph{convexification area} and we denote by $\mathcal P$ the set where $\alpha_\t{conv}\neq\alpha_\t{prec}$.
In coordinates $(x,y,z)$, the connected component of $\mathcal P$ containing $\Gamma_0$ is the set $z\geq z_0$. In coordinates $(t,u,v)$, it is contained between $S_{\mathcal B}$ and $S_h$ and its $v$-coordinates are in $(v_\Delta,v'_\t{max})$ (Proposition \ref{existence_f_g}).

\begin{corollary}[Corollary of Lemma \ref{lemme_restriction_psi}]\label{p_e_p_s}
Let $\gamma$ be a Reeb orbit intersecting $\mathcal P$. If $\gamma$ enters $\mathcal C$ in $p_\t{in}=(x_\t{in},y_\t{in},z_\t{in})$ such that $x_\t{in}\in\left[\frac{\pi}{2}-\frac{\lambda}{4},\frac{\pi}{2}+\frac{\lambda}{4}\right]$ and $z_\t{in}>0$ then the exiting point $p_\t{out}=(x_\t{out},y_\t{out},z_\t{out})$ satisfies $x_\t{out}\in\left[\frac{\pi}{2}-\frac{\lambda}{3},\frac{\pi}{2}+\frac{\lambda}{3}\right]$ and $y_\t{out}>y_\t{in}$.
\end{corollary}

\begin{proof}
If $\gamma$ intersects the set $y\geq \frac{2}{5} y_\t{std}$, we obtain the desired result as $R_y>0$ and $h$  is defined for $x\in\left[\frac{\pi}{2}-\frac{\lambda}{3},\frac{\pi}{2}+\frac{\lambda}{3}\right]$. We now assume that $\gamma$ is contained in the set $y\leq \frac{2}{5} y_\t{std}$. As $\gamma$ intersects $\mathcal P$, we have $t_\t{out}-t_\t{in}\geq\Delta$, and $p_\t{in},p_\t{out}\in I'_t\times I_u\times I'_v$ (Lemma \ref{lemme_restriction_psi} and Construction \ref{construction_h}). In addition $u_\t{out}-u_\t{in}\geq0$ as $R_u\geq 0$. Lemma \ref{lemme_restriction_psi} gives the desired result.
\end{proof}
By symmetry, there exist analogous statements near any endpoint of $\Gamma_\t{prec}^{\leq y_\t{std}}$.

\subsection{Conditions (B1) to (B8)}\label{subsection_conditions}
We now prove that our construction satisfies conditions (B1) to (B8) and is adapted to the boundary.

\emph{The contact form is adapted to the boundary.} By the definition of pre-convex bypass, the contact form is adapted to the boundary outside $\mathcal C$. The contact form is adapted for $y\geq\frac{2}{3}y_\t{std}$ as $\alpha_\t{conv}=\alpha_\t{prec}$ and $\Vert h \Vert_{\mathcal C^\infty}<\epsilon'_\t{stab}$ (Lemma \ref{epsilon'_stab}). 
For $y\leq \frac{2}{5}y_\t{std}$ and $z\leq \frac{2}{5}y_\t{std}+z_\t{prod}$, Lemma \ref{existence_f_g} gives the desired result. 
For $y\in\left[\frac{2}{5}y_\t{std},\frac{2}{3}y_\t{std}\right]$ and $x^h_\t{flat}\leq\vert x-\frac{\pi}{2}\vert$, we apply Lemma \ref{epsilon_H}. 
Finally, for $y\in\left[\frac{2}{5}y_\t{std},\frac{2}{3}y_\t{std}\right]$ and $\vert x-\frac{\pi}{2}\vert< x^h_\t{flat}$, we have $h=h(0)$ and $R_z$ is positively collinear to
\[\big(k_\t{inf}(z)k_\t{sup}(z)+al(y)\Phi_{z_0}(z)\big)\cos(x).\]
Thus 
\begin{itemize}
  \item $R_z=0$ for $x=\frac{\pi}{2}$;
  \item $R_z>0$ for $x<\frac{\pi}{2}$;
  \item $R_z<0$ for $x>\frac{\pi}{2}$. 
\end{itemize}
The tangency points between $\tilde{S}_h$ and $R_{\alpha_\t{conv}}$ are the segment $x=\frac{\pi}{2}$. Along this segment and for $y\leq \frac{1}{2}y_\t{std}$, $R_x$ is positively collinear to $f_\t{inf}(y)\rho -a\Phi_{z_0}'(z_h) $ and

\[f_\t{inf}(y)\rho -a\Phi_{z_0}'(z_h)\leq  \rho-a\Phi_{z_0}'(z_h)<0.\]  For $y> \frac{1}{2}y_\t{std}$, $R_x$ is positively collinear to \[-k'_\t{sup}(z_h)-cg(y)\Phi_{z_0}'(z_h)<0.\] By symmetry we obtain the desired result in the other convexified areas.

\emph{Condition (B6).} Let $\gamma$ be a Reeb chord of $S_Z$ in $\mathcal B_\t{conv}^{\leq y_\t{std}}$. If $\gamma$ does not meet $\mathcal P$, condition (B6) is given by Lemma \ref{epsilon_stab_retour_S_S_y}. We now assume that $\gamma$ intersects $\mathcal P$. Let $p_\t{in}^S$ and $p_\t{out}^S$ denote the endpoints of $\gamma$ and $p_\t{in}$ and $p_\t{out}$ the first entering and exiting point of $\mathcal C$. We assume that $p_\t{in}$ is in the connected component of $\mathcal C$ containing $\Gamma_0$. The proof of the other cases is similar.

If $\gamma$ does not intersect $\mathcal P$ after $p_\t{out}$, then $\gamma$ is contained in $\left[\pi-\frac{\lambda}{4},\pi+\frac{\lambda}{4}\right]\times[0,y_\text{std}]\times I_\text{max}$ after $p_\t{out}$ (Lemma \ref{lemme_cordes_Reeb} and Equation (\ref{eq_V})). Thus we have $x^S_\t{out}\in \left[\pi-\frac{\lambda}{2},\pi+\frac{\lambda}{2}\right]$ (Lemma \ref{lemme_lambda}). Then $x_\t{in}\in\left[\pi-\frac{\lambda}{3},\pi+\frac{\lambda}{3}\right]$ and $z_\t{in}>z_\t{out}$ (Corollary \ref{p_e_p_s}). Therefore $z_\t{in}^S>z_\t{in}>z_\t{out}>z_\t{out}^S $ as $R_z<0$. We obtain $x^S_\t{in}\in \left[\pi-\frac{\lambda}{2},\pi+\frac{\lambda}{2}\right]$ (Lemma \ref{lemme_lambda}).

If $\gamma$ intersects $\mathcal P$ after $p_\t{out}$, then $\gamma$ meets the connected component associated to 
$\left[\pi,\frac{3\pi}{2}\right]\times[0,y_\text{std}]\times[-z_\text{std},0]$ (Lemma \ref{lemme_cordes_Reeb}). Let $p'_\t{in}$ and $p'_\t{out}$ denote the second entering and exiting point. Between $p_\t{out}$ and $p'_\t{in}$, $\gamma$ is contained $\left[\pi-\frac{\lambda}{4},\pi+\frac{\lambda}{4}\right]\times[0,y_\text{std}]\times I_\text{max}$ (Lemma \ref{lemme_cordes_Reeb}). Thus we have $x_\t{in}\in\left[\pi-\frac{\lambda}{3},\pi+\frac{\lambda}{3}\right]$ and $z_\t{in}>z_\t{out}$ (Corollary \ref{p_e_p_s}). Therefore $x^S_\t{in}\in \left[\pi-\frac{\lambda}{2},\pi+\frac{\lambda}{2}\right]$ (Lemma \ref{lemme_lambda}). As $\gamma$ does not intersect $\mathcal P$ before $p_\t{in}$ and $R_z<0$, we obtain $z_\t{in}^S>z_\t{in}>z_\t{out}>z'_\t{out}$. In addition $x'_\t{out}\in\left[\pi-\frac{\lambda}{3},\pi+\frac{\lambda}{3}\right]$ and $z'_\t{in}>z'_\t{out}$ (Corollary \ref{p_e_p_s}). As 
$\gamma$ does not meet $\mathcal P$ after $p'_\t{out}$ (Lemma \ref{lemme_cordes_Reeb}) and $R_z<0$, we obtain $z_\t{in}^S>z_\t{in}>z_\t{out}>z'_\t{in}>z'_\t{out}>z_\t{out}^S$ and $x^S_\t{out}\in \left[\pi-\frac{\lambda}{2},\pi+\frac{\lambda}{2}\right]$ (Lemma \ref{lemme_lambda}).  

\emph{Condition (B7).}
Let $\gamma$ be a Reeb orbit in $\mathcal B_\t{conv}^{\leq y_\t{std}}$ with endpoints  $p_\t{in}^S$ and $p_\t{out}^S$ in $S_Z$ and $S_{y_\t{std}}$. If $\gamma$ does not meet $\mathcal P$ we obtain the desired result by Lemma~\ref{epsilon_stab_retour_S_S_y}. We now assume that $\gamma$ meets $\mathcal P$. The image of $\mathcal P$ on $S_Z$ is contained in $X\cup X+2\pi$. Thus $p_\t{in}^S\in X+2k\pi$ for $k\in\{0,1\}$.
In addition, there exists $k'\in\{0,1\}$ such that $p_\t{out}^S\in\left[\frac{\pi}{2}-\frac{\lambda}{2}+2k'\pi,\frac{\pi}{2}+\frac{\lambda}{2}+2k'\pi\right]\times I_\t{max}$ (Lemma \ref{epsilon_stab_retour_S_S_y}). It remains to prove that $k=k'$.
If $\gamma$ meets $\mathcal P$ once, then the $x$-coordinate of the exiting point is in  $\left[2k\pi-\frac{\lambda}{8},(2k+1)\pi+\frac{\lambda}{8}\right]$ as $\mathcal P\subset\mathcal V$ and thus $k=k'$ (Lemma \ref{lemme_lambda}). If $\gamma$ meets $\mathcal P$ twice, then the first exiting point has a $x$-coordinate in $\left[2k\pi-\frac{\lambda}{8},(2k+1)\pi+\frac{\lambda}{8}\right]$. Thus the $x$-coordinate of the second entering point is in  $\left[2k\pi-\frac{\lambda}{4},(2k+1)\pi+\frac{\lambda}{4}\right]$ (Lemma \ref{lemme_lambda}) and  the $x$-coordinate of the second exiting point is contained in $\left[2k\pi-\frac{\lambda}{3} ,(2k+1)\pi+\frac{\lambda}{3}\right]$ (Corollary \ref{p_e_p_s}). Thus $k=k'$. The proof of \emph{condition (B8)} is similar.

\emph{Condition (B5).} By Remark \ref{pas_de_retour_sur_S_y_std}, if $\gamma$ is a Reeb chord of $S_{y_\t{std}}$ in $\mathcal B_\t{conv}^{\leq y_\t{std}}$ then $\gamma$ intersects $S_{\frac{2}3{}y_\t{std}}$. Let $p_\t{in}^S$ and $p_\t{out}^S$ denote the endpoints of $\gamma$. By Lemma \ref{epsilon_stab_retour_S_S_y},  $x_\t{in}^S\in\left[\frac{3\pi}{2}-\frac{\lambda}{2}+2k\pi,\frac{3\pi}{2}+\frac{\lambda}{2}+2k\pi\right]$ and $x_\t{out}^S\in\left[\frac{\pi}{2}-\frac{\lambda}{2}+2k'\pi,\frac{\pi}{2}+\frac{\lambda}{2}+2k'\pi\right]$. In addition $\gamma$ intersects $\mathcal P$ (Lemma \ref{lemme_lambda}). Yet $\gamma$ intersects only one connected component of $\mathcal P$ (Lemma \ref{lemme_cordes_Reeb}). This leads to a contradiction.

\emph{Condition (B4).} This condition is a consequence from Lemma \ref{lemme_B4}.

\emph{Conditions (B1), (B2) and (B3).} These conditions derive from Lemma \ref{lemme_S_B} and Lemma \ref{lemme_perturbation_2}. Indeed, by Lemma \ref{epsilon_stab_retour_S_S_y}, all the Reeb chords which contribute to the map between $R_\lambda$ and $S_R$ intersect $\mathcal P$ and thus $\Sigma_+$ and $\Sigma_-$. In addition, the intersection points with $\Sigma_-$ are in $[u_2,u_3]\times I_v$ (Equation (\ref{eq3})) and the intersection points with $\Sigma_+$ in $[u_4,u_5]\times I_v$ (Equations (\ref{eq1}) and (\ref{eq2}) and Lemma \ref{epsilon_stab_retour_S_S_y}). Let $\delta$ be a curve in $\left[\frac{\pi}{2}+\lambda,\pi-\lambda\right]\times I_\t{max}$ with tangents in $\mathcal C(V,A)$. Then the tangents of the image of $\delta$ in $\Sigma_-$ are in $\mathcal C(V,B)$ (Lemma \ref{lemme_S_B}) and the tangents of the image of $\delta$ in $\Sigma_+$ are in $\mathcal C(H,\epsilon_B)$ (Lemma \ref{lemme_perturbation_2}). Thus the image of $\delta$ on $S_R$ is $\epsilon_R$ close to $\delta_1$. Similarly, the 
tangent of the image of an horizontal segment in $S_R$ are in $\mathcal C(H,\nu)$ (Lemma \ref{lemme_S_B}). Condition (B3) is a consequence of the definition of pre-convex bypasses. We obtain the rectangle structures on $\dom(\phi_i)$ and  $\im(\phi_i)$ by considering the images of vertical curves in $S_Z$ and the inverse images of horizontal curves in $S_R$. These curves are transverse (definition of pre-convex bypasses).

\section{Conley-Zehnder index}\label{section_Maslov}
In this section we prove Theorem \ref{theoreme_Maslov}: we compute the Conley-Zehnder index $\mu$ of the periodic orbit $\gamma_\mathbf{a}$ described in Theorem \ref{theoreme_principal}.
\subsection{Two technical lemmas}

\begin{lemma}\label{lemme_CZ_e_1}
Let $(R_t)_{t\in[0,1]}$ be a path of symplectic matrices in $\mathbb R^2$ such that $R_0=\Id$ and $R_1\in Sp^*$. Let $R_te_1=r(t)e^{i\alpha(t)}$. If $\alpha(1)\in\left[2k\pi+\frac{\pi}{2},2k\pi+\frac{3\pi}{2}\right]$ and $\mu(R)$ is odd, then
then $\mu\left(R\right)= 2k+1$. Similarly if $\mu(R)$ is even and $\alpha(1)\in \left[2k\pi -\frac{\pi}{2},2k\pi+\frac{\pi}{2}\right]$,
then $\mu\left(R_t\right)= 2k$.
\end{lemma}

\begin{proof}
We extend $R_t$, $\alpha_t$ and $r_t$ to $t\in[1,2]$ (see Section \ref{subsubsection_CZ}). Let $\theta_t$ denote the rotation angle associated to the polar decomposition $R_t=S_tO_t$. Without loss of generality $\theta_0=0$. As $S_t$ is positive-definite, $\theta_t-\frac{\pi}{2}< \alpha_t<\theta_t+\frac{\pi}{2}$. Additionally, if there exists $t\in[1,2]$ such that $\theta_t=0 [2\pi]$ then $R_t\in \Sp^-(2)$. Similarly if $\theta_t=\pi [2\pi]$ then we have $R_t\in \Sp^+(2)$.
Therefore if $\mu(R)$ is odd, $\theta_t\neq 0[2\pi]$ for all $t\in[1,2]$. Thus $\theta_1-\pi<\theta_2<\theta_1+\pi$ and
$\alpha_1-\frac{3\pi}{2}<\theta_2<\alpha_1+\frac{3\pi}{2}$. Therefore $\theta_2\in((2k-1)\pi,(2k+3)\pi)$.
The proof of the other case is similar.
\end{proof}

\begin{lemma}\label{lemme_matrice_hyperbolique}
Let $\theta_0>0$. There exists $\nu(\theta_0)>0$ such that if $R\in\Sp(2)$ and
\begin{itemize}
  \item $Re_1\in \mathcal C(e_1,\tan(\nu(\theta_0)))$;
  \item $\Vert Re_1\Vert\geq 3$;
  \item there exists $f\in C(e_2,\tan(\theta_0))$ such that $Rf\in C(e_2,\tan(\theta_0))$;
\end{itemize}
then $R$ is $\mathbb R$-diagonalizable and its eigenvalues are of the sign of $\langle e_1,Re_1\rangle$.
\end{lemma}

\begin{proof}
We prove that $\vert \tr(R)\vert>2$. Without loss of generality, $(e_1,f)$ is a direct basis. The matrix associated to the change of basis from $(e_1,e_2)$ to $(e_1,f)$ is
\[P=\left(\begin{array}{cc}
 1 &  \cos(\theta) \\
0 & \sin(\theta)
\end{array}\right) \]
where $\theta\in \left[\frac{\pi}{2}-\theta_0,\frac{\pi}{2}+\theta_0\right]$. 
In the basis $(e_1,e_2)$, the matrix of $R$ is
\[P^{-1}RP
=\frac{1}{\sin(\theta)}\left(\begin{array}{cc} 
 \mu_1\sin(\theta-\theta_1) & \mu_2\sin(\theta-\theta_2)  \\ 
 \mu_1\sin(\theta_1)& \mu_2\sin(\theta_2) 
\end{array}\right)\] 
where $\theta_1\in[-\nu(\theta_0),\nu(\theta_0)]$, $\theta_2\in\left[\frac{\pi}{2}-\theta_0,\frac{\pi}{2}+\theta_0\right]$ and $\vert\mu_1\vert\geq 3$. As $\mu_1\mu_2>0$, we obtain $\tr(R)> 2$ for $\nu(\theta_0)$ small enough.
\end{proof}

\subsection{Computation of the Conley-Zehnder index}
Let ${\mathbf{a}}=a_{i_1}\dots a_{i_k}$ be a word such that $l(\mathbf{a})<K$. Let $p_{\mathbf{a}}$ be an intersection point between $\gamma_{\mathbf{a}}$ and $S^-_Z$. We denote by $T({\mathbf{a}})$ the period of $\gamma_{\mathbf{a}}$. Let $R_t$ be the path of symplectic matrices along $\gamma_{\mathbf{a}}$ associated to the trivialisation described in Section \ref{section_two_improvements}. Let $\chi_{\mathbf{a}}$ be the map induced by the Reeb flow between $S_Z$ and $S_{\mathbf{a}}$, a surface tangent to $\xi(p_{\mathbf{a}})$ at $p_{\mathbf{a}}$. Let $G_{\mathbf{a}}=\phi_B\circ \psi_{i_k}\dots\circ \phi_B\circ \psi_{i_1}$. By definition of $\phi_B$ and $\psi$,
\begin{itemize}
  \item $\dom(G_{\mathbf{a}})$ and $\im(G_{\mathbf{a}})$ are rectangles with respectively horizontal and vertical fibres and $G_{\mathbf{a}} $ preserves the fibres;
  \item $\d G_{\mathbf{a}}\left(p_{\mathbf{a}},\frac{\partial}{\partial x}\right)\in\mathcal C(H,\nu)$;
  \item $\left\Vert\d G_{\mathbf{a}}\left(p_{\mathbf{a}},\frac{\partial}{\partial x}\right)\right\Vert\geq \frac{1}{(\eta M)}$.
\end{itemize}

\begin{lemma}\label{lemme_mu_nu_maslov}
There exists $\theta_0$ such that for $\mu$, $\nu$ and $\eta$ small enough, $R_{T({\mathbf{a}})}$ satisfies the hypothesis of Lemma \ref{lemme_matrice_hyperbolique}.
\end{lemma}

\begin{proof}
Note that $\d\chi_{\mathbf{a}}(p_{\mathbf{a}})\frac{\partial}{\partial x}=\pm e_1$. We choose $\theta_0$ such that ${\chi_{\mathbf{a}}}_*\left(\mathcal C(V,A)\right)\subset \mathcal C(e_2,\tan(\theta_0))$. For $\mu$ small enough, we have ${\chi_{\mathbf{a}}}_*\left(\mathcal C(H,\mu)\right)\subset \mathcal C(e_1,\tan(\nu(\theta_0))$. Let $l$ be such that $\Vert \d\chi_{p_{\mathbf{a}}}(p_{\mathbf{a}})\Vert<l$ and $\Vert \d\chi_{p_{\mathbf{a}}}(p_{\mathbf{a}})^{-1}\Vert<l$.
Then
\[\Vert R_{{\mathbf{a}},T({\mathbf{a}})}e_1 \Vert \geq \bigg\Vert d\chi_{p_{\mathbf{a}}}(p_{\mathbf{a}})^{-1}\bigg\Vert^{-1} \bigg\Vert\d G_{\mathbf{a}}\frac{\partial}{\partial x}\bigg\Vert\geq \frac{1}{l\eta M}\]
and $\Vert Re_1\Vert\geq 3$ for $\eta$ small enough.
\end{proof}

\begin{lemma}\label{lemme_dF_a}
For all $p\in\dom(G_{\mathbf{a}})$, $\left\langle\d G_{\mathbf{a}}(p)\frac{\partial}{\partial x},\frac{\partial}{\partial x}\right\rangle$ is of the sign of
$\prod_{j=1}^k (-1)^{\tilde{\mu}(a_{i_j})}$. 
\end{lemma}

\begin{proof}
Note that $\left\langle\d\phi_B(p)\frac{\partial}{\partial z},\frac{\partial}{\partial x}\right\rangle >0$ for all $p\in\dom(\phi_B)$. Therefore, we have \[\left\langle\d\phi_B(p)v,\frac{\partial}{\partial x}\right\rangle >0\] for all $v\in\mathcal C_p(V,A)$ such that $\left\langle v,\frac{\partial}{\partial z}\right\rangle>0$.

We prove the desired result by induction on $k$. If $k=1$ then $\d\psi_{a_1}(p)\frac{\partial}{\partial x}\in\mathcal C(V,A)$. If $\tilde{\mu}(a_1)$ is even, $\left\langle-\d\psi_{a_1}(p)\frac{\partial}{\partial x},\frac{\partial}{\partial z}\right\rangle<0$ (see Figure \ref{mu_pair_impair}) and we obtain \[\left\langle\d G_{a_1}(p)\frac{\partial}{\partial x},\frac{\partial}{\partial x}\right\rangle>0.\] Similarly, if $\mu(a_1)$ is odd, we obtain  $\left\langle-\d\psi_{a_1}(p)\frac{\partial}{\partial x},\frac{\partial}{\partial z}\right\rangle>0$ and \[\left\langle\d G_{a_1}(p)\frac{\partial}{\partial x},\frac{\partial}{\partial x}\right\rangle<0.\]  
\begin{figure}[here]
\begin{center}
 \includegraphics{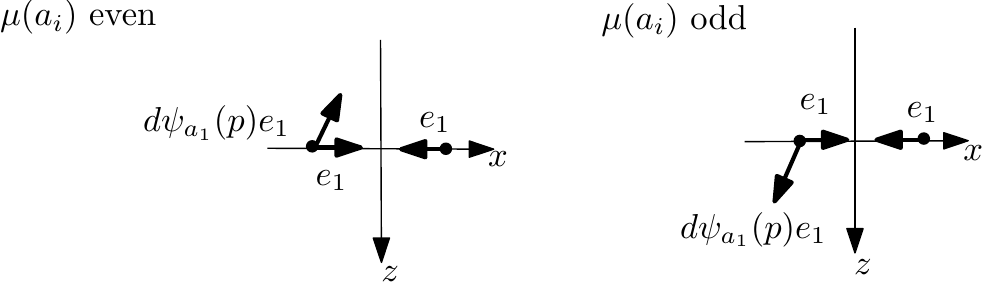}
\end{center}
 \caption{The vector $\d\psi_{a_1}(p)e_1$}\label{mu_pair_impair}
\end{figure}
We now prove the result for ${\mathbf{a}}=a_{i_1}\dots a_{i_{k+1}}$. Let $p\in\dom(G_{\mathbf{a}})$ and $v=\d G_{a_{i_1}\dots a_{i_{k}}}(p)\frac{\partial}{\partial x}$. By induction, $\left\langle v,\frac{\partial}{\partial x}\right\rangle$ is of the sign of $\prod_{j=1}^k (-1)^{\tilde{\mu}(a_{i_j})}$. Then, $\left\langle \d \psi_{a_{i_{k+1}}} v,\frac{\partial}{\partial z}\right\rangle$ is of the sign of $\prod_{j=1}^{k+1} (-1)^{\tilde{\mu}(a_{i_j})}$ and so is $\left\langle \d G_{a_{i_{k+1}}} v,\frac{\partial}{\partial x}\right\rangle$.
\end{proof}

\begin{corollary}\label{lemme_parite_maslov}  
For $\mu$, $\nu$ and $\eta$ small enough, $\sum_{j=1}^{k} \tilde\mu(a_{i_j})=\mu(\gamma_{\mathbf{a}})[2]$. 
\end{corollary}

\begin{proof}
The signs of $\left\langle \d G_{{\mathbf{a}}} \frac{\partial}{\partial x},\frac{\partial}{\partial x}\right\rangle$ and $\left\langle R_{T({\mathbf{a}})} \frac{\partial}{\partial x},\frac{\partial}{\partial x}\right\rangle$ coincide and $R_{T({\mathbf{a}})}$ is hyperbolic (Lemmas \ref{lemme_matrice_hyperbolique} and \ref{lemme_mu_nu_maslov}). Its eigenvalues are positive if $\sum_{j=1}^{k} \tilde{\mu}(a_{i_j})$ is even and negative if $\sum_{j=1}^{k} \tilde{\mu}(a_{i_j})$ is odd (Lemma \ref{lemme_dF_a}).
\end{proof}

\begin{lemma}\label{lemme_trivialisation_gamma_a}
There exists a collar neighbourhood $S_{\gamma_{\mathbf{a}}}$ of $\gamma_{\mathbf{a}}$ in the trivialisation class given in Section \ref{section_two_improvements} such that if $R_t$ is the associated path of symplectic matrices along $\gamma_{\mathbf{a}}$ and $R_t\frac{\partial}{\partial x}=r(t)e^{i\theta(t)}$  then $\theta(t)\neq 0[\pi]$ when $\gamma_{\mathbf{a}}(t)$ is in the bypass and $t>0$.
\end{lemma} 

\begin{proof}
Let $c$ be a Reeb chord in the bypass contained in $\gamma_\mathbf{a}$ with endpoints $c_+$ and $c_-$ on $S_Z$. We construct a strip $S_c$ along $c$ such that no Reeb chord with one endpoint on the vertical segment containing $c_+$ and close to $c_+$ intersects $S_c$. We then glue together the half of $S_c$ and the collar associated to the Reeb chords to obtain $S_{\gamma_\mathbf{a}}$. We choose $S_c$ such that
\begin{enumerate}
 \item between $S_Z$ and the convexification, $S_c$ is tangent to $\frac{\partial}{\partial x}$;
 \item in the convexification area, $S_c$ is tangent to $\frac{\partial}{\partial u}$;
 \item is the upper part of the bypass, $S_c$ is tangent $\frac{\partial}{\partial r}$.
\end{enumerate}
We smooth the resulting surface. Figure \ref{indice_C_Z} shows $c$ and a Reeb chord with one endpoint on the vertical segment contained $c_+$ (dotted curve). In this figure $S_c$ is transverse to the projection.
\begin{figure}[here]
\begin{center}
\includegraphics{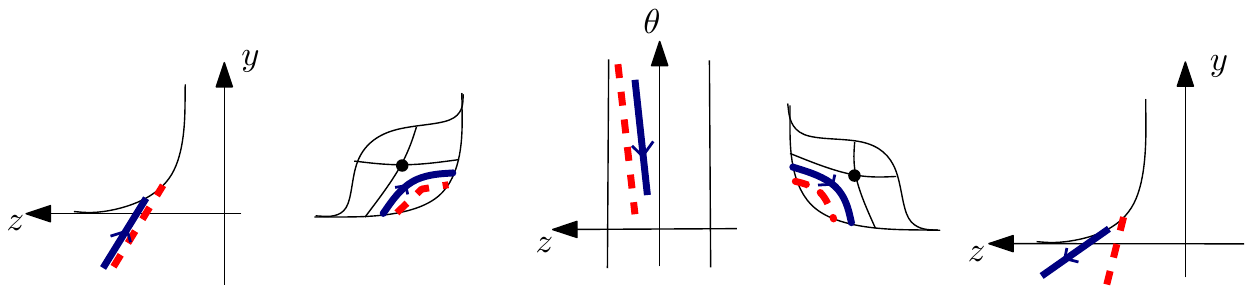}
\end{center}
\caption{The surface $S_c$ and the Reeb chords}\label{indice_C_Z}
\end{figure}

\end{proof}

\begin{proof}[Proof of Theorem \ref{theoreme_Maslov}]
We consider the trivialisation from Lemma \ref{lemme_trivialisation_gamma_a}. Without loss of generality $\theta(0)=0$. Let $0=t_1<t'_1<\dots<t_k<t'_k=T({\mathbf{a}})$ be the times associated to the intersection points between $\gamma_{\mathbf{a}}$ and $S_Z$ ($\gamma_{\mathbf{a}}(0)$ is the fixed point of $G_{\mathbf{a}}$).

We prove by induction that for all $j=1,\dots, k$
\begin{equation}\label{theta_j}
\theta(t_j)\in\left[\left(\sum_{l=1}^{j-1} \mu(a_{i_l})\pi\right)-\nu(\theta_0),\left(\sum_{l=1}^{j-1} \mu(a_{i_l})\pi\right)+\nu(\theta_0)\right].
\end{equation}
As $\theta(0)=0$, the condition (\ref{theta_j}) is satisfied for $j=1$. We now suppose that the equation (\ref{theta_j}) stands for $i\in1\dots j-1$. By definition of $\mu(a_{i_{j}})$, 
\[\theta(t'_j)\in\left[\left(\sum_{l=1}^{j} \mu(a_{i_l})\pi\right),\left(\sum_{l=1}^{j} \mu(a_{i_l})\pi\right)+\pi\right].\]
We obtain (Lemma \ref{lemme_trivialisation_gamma_a})
\begin{align*}
\theta(t_{j+1})&\in\left[\left(\sum_{l=1}^{j} \mu(a_{i_l})\pi\right)-\nu(\theta_0),\left(\sum_{l=1}^{j} \mu(a_{i_l})\pi\right)+\nu(\theta_0)\right]\\
\t{ or }\hskip.5cm \theta(t_{j+1})&\in\left[\left(\sum_{l=1}^{j} \mu(a_{i_l})\pi\right)+\pi-\nu(\theta_0),\left(\sum_{l=1}^{j} \mu(a_{i_l})\pi\right)+\pi+\nu(\theta_0)\right].
\end{align*}
By Lemma \ref{lemme_dF_a}, we obtain the equation (\ref{theta_j}) for $i=j$. Lemma~\ref{lemme_CZ_e_1} provides us with the desired Conley-Zehnder index.
\end{proof}

\noindent{\bf{Acknowledgements.}} I am deeply grateful to my advisor, Vincent Colin, for his guidance and support. I would also like to thank Samuel Tapie, Paolo Ghiggini and François Laudenbach for stimulating discussions. Thanks also go to Chris Wendl for his help with holomorphic curves, Jean-Claude Sikorav for suggesting numerous improvements and corrections and Marc Mezzarobba for proofreading this text.

\end{document}